\documentclass[12pt]{amsart} 
\usepackage{amscd,amssymb,mathrsfs}
\usepackage{mathtools}% needed for xmapsto
\usepackage{graphics}
\usepackage[dvipdf]{graphicx}
\usepackage{xypic}

\usepackage{pinlabel}

\usepackage{subfigure}
\usepackage{enumerate}
\usepackage{relsize} % Necessary to use "mathsmaller": see macros file

\usepackage{amsmath}   % the following code is for augmented matrices
\makeatletter
\renewcommand*\env@matrix[1][*\c@MaxMatrixCols c]{%
  \hskip -\arraycolsep
  \let\@ifnextchar\new@ifnextchar
  \array{#1}}
\makeatother

\newtheorem{thm}{Theorem}[subsection]
\newtheorem{cor}[thm]{Corollary}%[section]
\newtheorem{lemma}[thm]{Lemma}%[section]

\newtheorem{prop}[thm]{Proposition}
\newtheorem{conj}[thm]{Conjecture}

\newtheorem*{thm*}{Theorem}
\newtheorem*{fact*}{Fact}

\newtheorem{defn}[thm]{Definition}
\newtheorem*{defn*}{Definition}

\begin{document}
\input Margulis.macros
\title[Proper affine actions]
{Proper actions of discrete groups of affine transformations}
\author[Danciger]{Jeffrey Danciger}
    \address{{\it Danciger:\/}
    Department of Mathematics\\ Univeristy of Texas, Austin\\
	Austin,  TX  78712 USA}
    \email{jdanciger@math.utexas.edu}
\author[Drumm]{Todd A. Drumm}
    \address{{\it Drumm:\/}
    Department of Mathematics\\ Howard University\\
    Washington, DC 20059 USA }
    \email{tdrumm@howard.edu}
\author[Goldman]{William M. Goldman}
    \address{{\it Goldman:\/}
    Department of Mathematics\\ University of Maryland\\
    College Park, MD 20742 USA}
    \email{wmg@math.umd.edu}
\author[Smilga]{Ilia Smilga}
    \address{{\it Smilga:\/}
    Institut de Math{\'e}matiques de Marseille (UMR 7373) \\
    Aix-Marseille Universit{\'e} \& CNRS \\
    39 rue F. Joliot Curie \\
    13453 MARSEILLE Cedex, France}
    \email{ilia.smilga@normalesup.org}
    \thanks{Danciger was partially supported by NSF grants DMS 1510254, and DMS 1812216 and by an Alfred P. Sloan Foundation fellowship.
    Goldman 
gratefully acknowledges research support from NSF Grants DMS1065965, DMS1406281,
DMS1709791,
Smilga acknowledges support from NSF grant DMS1709952 and from the European Research Council (ERC) under the European Union Horizon 2020 research and innovation programme, grant 647133 (ICHAOS).  
The authors gratefully acknowledge support from the Research Network in the Mathematical Sciences  DMS 1107452, DMS 1107263, DMS 1107367 (GEAR).
}
\dedicatory{Dedicated to Grisha Margulis on the occasion of his seventieth birthday}

\date{\today}

\begin{abstract}
In the early 1980's Margulis startled the world by showing the
existence of proper affine actions of free groups on $3$-space,
answering a provocative and suggestive question Milnor posed in 1977.
In this paper we discuss the historical background motivating this question,
recent progress on this subject, 
and future directions inspired by this discovery.
\end{abstract}

\maketitle
\newpage
\tableofcontents
\noindent{Solicited for: 
\textbf{Dynamics, Geometry, Number Theory: the Impact of Margulis on
Modern Mathematics}}

\section*{Introduction}

The theory of flat Riemannian manifolds, also known as Euclidean manifolds, is well understood.
Starting with its nineteenth-century origins in theoretical crystallography,
Euclidean crystallographic groups and complete flat Riemannian manifolds have a satisfying
and cohesive structure theory. 
In particular, the Bieberbach theorems imply that every closed flat Riemannian manifold is finitely covered by a torus or equivalently, any Euclidean crystallographic group is virtually a lattice in $\Rn$, in particular is virtually free abelian. 
This survey concerns {\em complete affine manifolds,\/} 
a natural generalization of complete Euclidean manifolds whose structure theory, by contrast, 
remains tantalizingly mysterious and poorly understood.
The famous {\em Auslander conjecture\/} --- 
that every compact complete affine manifold has virtually polycyclic fundamental group --- 
has been a focal point for research in this field. 
Now known to be true in dimensions $< 7$, it remains open in general.

The last forty years have seen major advances in the theory of complete affine manifolds.
A significant breakthrough was Margulis's discovery in 1983 of proper affine actions
of nonabelian free groups in dimension $3$, 
and the subsequent classification of complete affine $3$-manifolds. 
As any proper affine action by a free group in dimension $3$ preserves a Lorentzian structure,
the corresponding complete affine $3$-manifolds are, more specifically, 
complete flat Lorentzian $3$-manifolds. They are known today as {\em Margulis spacetimes.\/}

Associated to a Margulis spacetime $M^3$ is a (necessarily noncompact) complete hyperbolic surface  $\Sigma$
homotopy-equivalent to $M^3$,
and we call $M^3$ an {\em affine deformation\/}  of  $\Sigma$.
Hence, the deformation space of Margulis spacetimes whose associated hyperbolic surface 
$\Sigma$ has a fixed topological type $S$ naturally projects down to the  Fricke-Teichm\"uller space 
$\mathfrak{F}(S)$ of $S$. 
The fiber of this projection consists of equivalence classes of 
{\em proper affine deformations\/} of $\Sigma$. 
A clear picture has emerged of the fiber of this projection 
as an open convex cone in the space of infinitesimal deformations of the
hyperbolic structure on $\Sigma$. 
Much of this paper describes this point of view. 
Crucial is the properness criterion for affine deformations developed
by Goldman, Labourie and Margulis~\cite{MR2600870}.
Along the way, we will also collect the known results 
on the topology and geometry of Margulis spacetimes and give an overview of the 
current state of the art in higher-dimensional affine geometry.

The fundamental problem in Euclidean crystallography was, in modern parlance, 
the classification of fundamental polyhedra for Euclidean crystallographic groups. 
However, in the setting of Margulis spacetimes, standard constructions for 
fundamental polyhedra do not work.
The introduction by Drumm of {\em crooked polyhedra\/} around 1990 
--- about a decade after Margulis' discovery --- 
provided tools for building fundamental domains and led in particular to the discovery that there exist Margulis 
spacetimes which are affine deformations of {\em any\/} non-compact complete hyperbolic surface of finite type.
This kindled momentum for the subject and marked the beginning of a classification program for Margulis spacetimes which was completed only recently.

The outline of this paper is as follows. 
Section 1 summarizes  the early history of the subject,
beginning with Bieberbach's ``classification'' of Euclidean manifolds,
and its subsequent generalizations.
These generalizations --- due to Zassenhaus, Wang, Auslander, Mostow and others ---
set the stage for the the classification of complete affine manifolds with virtually 
solvable fundamental group.
In 1977,
Milnor asked whether every complete affine manifold has 
virtually solvable fundamental group, or equivalently, if proper affine actions of 
the two-generator free group $\Ft$ {\em do not exist\/}.
Shortly thereafter, Margulis surprised everyone by showing the existence of 
complete affine manifolds with fundamental group $\Ft$.

Section 2 begins the construction and classification of Margulis spacetimes,
modeled on the geometric construction and classification of hyperbolic surfaces.
Crooked geometry is developed, including the disjointness criteria for crooked planes
which is fundamental in setting up the geometric conditions necessary for
building Schottky groups.
We briefly describe a compactification of $M^3$ as a flat $\rpthree$-manifold due to 
Suhyoung Choi, 
which implies that $M^3$ is homeomorphic to an open solid handlebody.

Section 3 introduces the {\em marked Lorentzian signed length spectrum,\/} 
or {\em Margulis invariant,\/} denoted $\alpha$. 
The Margulis invariant is an $\R$-valued class function on $\pi_1(M) \cong \Gamma$, 
and ever since Margulis introduced this quantity, 
it has played an important role in the geometry of Margulis spacetimes. 
The simplest type of Margulis spacetime occurs when the associated hyperbolic surface
has compact convex core (or equivalently, $\pi_1 \Sigma = \Gamma_0 < \Isom(\Ht)$ is
{\em convex cocompact.\/}) 
In this case,  every holonomy transformation is hyperbolic, 
and every essential loop is freely homotopic to a closed geodesic. 
A classical result in hyperbolic geometry asserts that such hyperbolic structures 
are determined up to isometry by their {\em marked length spectrum,\/}  
the $\Rplus$-valued class function on $\Gamma$
associating to $\gamma\in\pi_1(\Sigma)$ the {\em hyperbolic length\/}
$\ell(\gamma)$ of the unique closed geodesic  in $\Sigma$ that is homotopic to $\gamma$.
The {\em magnitude\/} of Margulis's invariant $\alpha(\gamma)$ 
\begin{align*}
\pi_1(\Sigma) &\xrightarrow{\vert\alpha\vert} \Rplus \\
\gamma &\longmapsto \vert \alpha(\gamma)\vert
\end{align*}
corresponds to the {\em Lorentzian length\/} of a closed geodesic homotopic to $\gamma$.
In particular, the isometry type of a Margulis spacetime $M^3$ is determined by
the marked length spectrum $\ell$ of $\Sigma$ and the absolute value 
$\vert\alpha\vert$ of the Margulis invariant.

In fact, only $\vert\alpha\vert$ is needed to determine the isometry type of $M^3$.
We discuss extensions of the definition of the Margulis invariant and of these 
results to the setting where $\Sigma$ has cusps.

Section 4 develops a properness criterion for actions of free groups in three-dimensional affine geometry. 
This turns out to be closely related to the {\em direction\/} or {\em sign\/} of $\alpha$. 
As originally noted by Margulis (the {\em Opposite Sign Lemma}), 
for any Margulis spacetime $M$, 
the sign of 
\[ \Gamma = \pi_1 M \xrightarrow{~\alpha~} \R \] is constant, 
either positive or negative.
A simple proof, given by Goldman-Labourie-Margulis, 
involves the continuous extension of the {\em normalized Margulis invariant\/}
$\alpha/\ell$ to the connected convex set of {\em geodesic currents\/} on $\Sigma$. 
This leads to the description of the deformation space of Margulis spacetimes
with fixed hyperbolic surface $\Sigma$ as an open convex cone in the  
vector  space of affine deformations of $\Gamma_0$, naturally the cohomology group 
$H^1(\Gamma_0,\Rto)$, where $$\Gamma_0 = \pi_1 \Sigma < \Isom(\Ht) = \SOto$$ 
is the holonomy group of $\Sigma$.

Section 5 develops the connection between affine actions in three-dimensional 
flat Lorentzian geometry and infinitesimal deformations of hyperbolic surfaces. 
Due to the low-dimensional coincidence that the standard action of $\SOto$ on 
$\R^3$ is isomorphic to the adjoint action on the lie algebra $\soto$, the space 
$H^1(\Gamma_0,\Rto)$ of affine deformations of the surface group $\Gamma_0 < \SOto$ 
is in natural bijection with the space of infinitesimal deformations of the 
representation $\Gamma_0 \hookrightarrow \SOto$, which in turn identifies 
with the space of infinitesimal deformations of the hyperbolic surface $\Sigma$. 
This interpretation leads to the fundamental result of  Mess that the hyperbolic surface $\Sigma$
associated to a Margulis spacetime cannot be closed.
In particular, if $\Gamma$ is a nonsolvable discrete group acting affinely on $3$-space,
then $\Gamma$ must be virtually free. 

The infinitesimal deformations of hyperbolic structures on $\Sigma$ which 
arise from proper affine actions
may be represented  by what Danciger--Gu\'eritaud--Kassel 
call {contracting \em lipschitz\/} vector fields,
which are the infinitesimal analogs of contracting Lipschitz maps on the hyperbolic plane. 

Section 5 develops the theory of lipschitz vector fields and 
a structure theorem for Margulis spacetimes:
{\em $M^3$ is a bundle of timelike lines over the hyperbolic surface $\Sigma$.\/}
This gave an independent proof of the topological characterization of Margulis 
spacetimes referenced above.
A discretized version of contracting lipschitz vector fields, known as infinitesimal 
strip deformations, was used by Danciger-Gu\'eritaud-Kassel to parameterize the 
deformation space of Margulis spacetimes associated to $\Sigma$ in terms of the 
\emph{arc complex} of $\Sigma$. 
We describe concretely the consequences of this general theory in the case that 
$\Sigma$ has Euler characteristic $-1$. The qualitative behavior depends on the 
topology of $\Sigma$, which is one of 
four possibilities:
the one-holed torus, the three-holed sphere, the two-holed projective plane 
(or cross-surface) or the
one-holed Klein bottle. 
Section 5 ends with a discussion of the construction, due to Danciger-Gu\'eritaud-Kassel, 
of proper affine actions of right-angled Coxeter groups in higher dimensions.
Similar in spirit to the 
case of Margulis spacetimes, these proper actions come from 
certain contracting deformations of hyperbolic and pseudo-hyperbolic reflection orbifolds.

Section 6 discusses 
other directions in higher dimensional affine geometry.
As in dimension three, 
a general approach to Auslander's conjecture involves classifying which 
groups can arise as Zariski closures
of the linear holonomy group.
Say that a connected subgroup $G\subgroup\GLn$ is {\em Milnor\/} if
no proper affine action of $\Ft$ 
with Zariski closure of the linear part equal to 
$G$ exists. 
Margulis's original result can be restated 
by saying 
that $\SOoto$ is {\em not\/} Milnor and it 
is the groups $G$ which are not Milnor which must be examined in order to study the 
Auslander conjecture.
Smilga gives a general sufficient condition for a linear representation of a 
semisimple Lie group to be non-Milnor. For example, the adjoint representation 
of a non-compact semisimple Lie group is not Milnor. 
Some other known results in higher dimensions are discussed in \S 6, concluding 
with a summary of the current state of Auslander's conjecture and a brief 
discussion of the proof of Abels-Margulis-Soifer for dimension $< 7$.

\subsection*{Acknowledgements}
 The authors thank Fran\c cois Gu\'eritaud, and the two anonymous referees for their extensive and 
excellent comments that helped to improve this manuscript.
Goldman also expresses thanks to the Clay Institute for Mathematical Sciences, 
Institute for Computational and Experimental Research in Mathematics (ICERM), 
and the Mathematical Sciences Research Institute (MSRI) 
where this paper was completed.

\section*{Notations and terminology}
We always work over the field $\R$ of real numbers, unless otherwise noted. 
Finitely generated free groups of rank $n \geq 1$  are denoted $\Fn$.
Discrete groups will be assumed to be finitely generated,
unless otherwise indicated. 
Denote the group of isometries of a space $X$ by $\Isom(X)$. 
If $G\subgroup \GL(N)$ is a matrix group, denote its Zariski closure (algebraic hull) by 
$\ZClosure{G} \subgroup \GL(N)$.

If $G$ is a group and  $S\subset G$, denote by $\langle S\rangle$ 
the subgroup of $G$ generated by $S$. 
Similarly, denote the cyclic group generated by an element $A\in G$ by $\langle A\rangle$. 

Denote the group of inner automorphisms of a group $G$ by $\Inn(G)$.
Denote the cohomological dimension of a group $\Gamma$ by
$\cd(\Gamma)$.
Denote the identity component of a topological group $G$ by $G^0$.

\subsection*{Vector spaces and affine spaces}
Let $\A$ be an affine space. 
The 
(simply transitive) group of translations of $\A$ is a vector space $\V$ called the {\em vector space underlying $\A$}.
We denote the group of linear automorphisms of $\V$ by
$\GL(\V)$ and the group of affine transformations of $\A$ by $\Aff(\A)$.
Let $o \in \A$ be a choice of basepoint. Then each element $g \in \Aff(\A)$ is given by a pair  $(A,\vb)$, 
where $A\in\GL(\V)$ is the {\em linear part\/} 
and $\vb\in\V$ is the {\em translational part:\/}
\[ g(x) = o +  A(x-o) + \vb.\]
We will henceforth suppress the basepoint $o \in \A$ and identify $\V$ with $\A$ via the map $v \in \V \mapsto o + v \in \A$. When $\dim~\A = \dim~\V = n$, we often further identify $\V$ with $\R^n$, and write $\A = \A^n$ to denote the affine space of $\V = \R^n$. Then, $\GL(\V)$ identifies with invertible $n\times n$ (real) matrices.
Writing $A = \LL(g)$ and $\vb = \uu(g)$, the affine transformation $g$ is the composition of an $n\times n$ matrix $\LL(g)$ and a translation $\uu(g)$ acting on $\A^n$:
\[
g(x) = \LL(g) (x) + \uu(g) 
\]
The linear part of $g$ identifies with the differential 
\[ \R^n \cong \T_x \A^n \xrightarrow{~\D_xg~} \T_{g(x)} \A^n \cong \R^n \]
of $g$, for every $x\in\A^n$.

Composing $g,h\in\Aff(\A)$, we find:
\begin{itemize}
\item $\LL$ is a homomorphism of groups: $\LL(g\circ h) = \LL(g) \LL(h)$ ;
\item $\uu$ is a $\V$-valued  $1$-cocycle,
where   $\V$ denotes the $\Aff(\A)$-module defined by $\LL$: 
\[ 
\uu(g\circ h) = \uu(g) + \LL(g) \uu(h). \]
\end{itemize}
 
If $\Gamma\xrightarrow{~\rho~} \Aff(\A)$ 
defines an affine action of $\Gamma$ on an affine space $\A$,
then the linear part $\L:= \LL\circ\rho$ defines a linear representation 
$\Gamma\xrightarrow{~\L~}\GL(\V)$
where $\V$ is the vector space underlying $\A$. 
Fixing $\L := \LL\circ\rho$, 
we say that $\rho$ is an {\em affine deformation\/} of $\L$.
Evidently an affine deformation of $\L$ is determined by the translational part 
$\u := \uu\circ\rho$:
\[ \Gamma \xrightarrow{~\u~}  \V \]
so that, for $x \in \A$: 
\begin{equation}\label{eq:AffineMap} 
\rho(\gamma)(x) = \L(\gamma) x+ \u(\gamma). \end{equation}
This map $\u$ satisfies the {\em cocycle  identity:\/}
\begin{equation}\label{eq:CocycleIdentity}
 \u( \gamma\eta ) = \u(\gamma) + \L(\gamma) \big(\u(\eta)\big) \end{equation}
for $\gamma,\eta \in \Gamma$. 
Denote the vector space of such cocycles $\Gamma\xrightarrow{~\u~} \V$
by $\Zz^1(\Gamma,\V)$.

If $\vv\in\V$, define its {\em coboundary\/} $\delta(\vv)\in \Zz^1(\Gamma,\V)$ 
as:
\[   \gamma \xmapsto{~\delta\vv~}  \vv -  \L(\gamma) \vv \]
Denote the image $\delta(\V)\subgroup \Zz^1(\Gamma,\V)$ by $\B^1(\Gamma,\V)$.
Two cocycles are {\em cohomologous\/} if their difference is a coboundary.

Conjugating an affine representation by translation by $\vv$ preserves the 
linear part but changes the translational part by adding $\delta\vv$.
Thus translational conjugacy classes of affine deformations with fixed linear part $\L$
are cohomology classes of cocycles, comprising the {\em cohomology\/} 
\[ \o{H}^1(\Gamma,\V)  := \Zz^1(\Gamma, V) /
\B^1(\Gamma, V).\]

\section{History and Motivation}
We  
briefly %
review the efforts of nineteenth-century crystallographers leading to Bieberbach's work on Euclidean manifolds and
lattices in $\Isom(\En)$, where we denote by $\En$ the Euclidean $n$-space (i.e.\ the affine space $\A^n$ endowed with a flat Euclidean metric).
Then we discuss extensions of these ideas to affine crystallographic groups,
and the question of Milnor on virtual polycyclicity of discrete groups of 
affine transformations acting properly.
Finally the section ends describing Margulis's unexpected discovery of
proper affine actions of non-abelian free groups.

\subsection{Euclidean crystallography}
In the nineteenth century crystallographers asked which groups of isometries
of Euclidean $3$-space $\Ethree$ can preserve a periodic tiling by polyhedra.
The symmetries of such a tiling form a group $\Gamma$ of isometries of
$\Ethree$ such that the quotient space, or orbit space, $\Gamma\backslash\Ethree$ is 
compact.

This led to a classification of {\em crystallographic space groups,\/} 
independently, by Sch\"onflies and Fedorov in 1891;
compare Milnor~\cite{MR0430101} for a historical discussion.
Since the interiors of the tiles are disjoint, the elements of 
$\Gamma$ cannot accumulate and the group must be discrete 
(with respect to the induced topology).
Henceforth, we  assume $\Gamma$ is discrete.

Define a {\em Euclidean space group\/} to be a discrete subgroup 
$\Gamma \subgroup \Isom(\En)$ satisfying any of the following equivalent properties:
\begin{itemize}
\item The quotient  \[ M = \Gamma\backslash\En\] is compact. 
\item There exists a compact {\em fundamental polyhedron\/} $\Delta\subset\En$
for the action of the group $\Gamma$:
\begin{itemize}
\item The interiors of the images $\gamma(\Delta)$, for $\gamma\in\Gamma$, 
are disjoint;
\item $\En =  \bigcup_{\gamma\in\Gamma} \gamma(\Delta)$.
\end{itemize}
\end{itemize}
Since the subgroup $\Gamma$ is discrete and acts isometrically on $\En$,
its action is {\em proper.\/} In particular the quotient $\Gamma\backslash\En$ is Hausdorff.
When $\Gamma$ is not assumed to be a group of isometries of a metric space,
criteria for a discrete group to act properly become a central issue. 

In 1911-1912 Bieberbach found a general group-theoretic criterion for such groups
in arbitrary dimension. In modern parlance, the discrete cocompact group $\Gamma$ 
is called a  {\em lattice\/}
in $\Isom(\En)$. 
Furthermore,  $\Isom(\En)$ decomposes as a semidirect product
$\R^n \rtimes \o{O}(n)$, where $\R^n$ is the vector space of {\em translations.\/}
Indeed, an affine automorphism is a Euclidean isometry if and only if its linear 
part lies in the orthogonal group $\o{O}(n)$.

Bieberbach showed:
\begin{itemize}
\item $\Gamma\cap \R^n$ is a lattice $\Lambda\subgroup \R^n$;
\item The quotient $\Gamma/\Lambda$ is a finite group, 
mapped isomorphically into $\o{O}(n)$ by $\LL$.
\item Any isomorphism $\Gamma_1 \longrightarrow \Gamma_2$ between 
Euclidean crystallographic groups $\Gamma_1,\Gamma_2\subgroup\Isom(\En)$ is induced
by an {\em affine automorphism\/} $\En \longrightarrow \En$.
\item There are only finitely many isomorphism classes of crystallographic
subgroups of $\Isom(\En)$. 
\end{itemize}

A {\em Euclidean manifold\/} is a flat Riemannian manifold, that is, 
a Riemannian manifold of zero curvature. 
A Euclidean manifold is {\em complete\/} if the  underlying metric  space is complete. 
By the Hopf-Rinow theorem, completeness is equivalent to geodesic completeness.

A torsion-free Euclidean crystallographic group $\Gamma\subgroup \Isom(\En)$ 
acts freely on $\En$, and the quotient $\Gamma\backslash\En$ is a compact complete 
Euclidean manifold.
Conversely, every compact complete Euclidean manifold is a quotient of $\En$ 
by a torsion-free
crystallographic group.  
Bieberbach's theorems have the following geometric interpretation:  %

\begin{itemize}
\item Every compact complete Euclidean manifold 
is a quotient of a flat torus $\Lambda\backslash\En$,
where $\Lambda\subgroup \R^n$ is a lattice of translations, by a finite group
of isometries acting freely on $\Lambda\backslash\En$.
\item Any homotopy equivalence $M_1 \longrightarrow M_2$ of compact complete
Euclidean manifolds is homotopic to an affine diffeomorphism.
\item There are only finitely many affine isomorphism classes of compact complete
Euclidean manifolds in each dimension $n$.
\end{itemize}

\subsection{Crystallographic hulls}\label{sec:CrystallographicHulls}

Bieberbach's theorems provide a satisfactory qualitative picture of compact Euclidean manifolds,
or (essentially) equivalently, cocompact Euclidean crystallographic groups.
Does a similar picture hold for {\em affine crystallographic groups, \/} 
that is discrete subgroups $\Gamma \subgroup \Aff(\A^n)$ which act 
properly and cocompactly on $\A^n$?

Auslander and Markus~\cite{MR0131842} constructed examples of 
{\em flat Lorentzian crystallographic groups\/} $\Gamma$ in dimension $3$  
for which all three Bieberbach theorems directly fail. 
In their examples, the quotients $M^3 = \Gamma\backslash\A^3$ are flat Lorentzian manifolds.
Topologically, these $3$-manifolds are all $2$-torus bundles over $S^1$; 
conversely every torus bundle over the circle admits such a structure.
Their fundamental groups are semidirect products $\Z^2 \rtimes \Z$,
and are therefore {\em polycyclic,\/} that is, iterated extensions of cyclic groups.

More generally, a group is {\em virtually polycyclic\/} if it contains a 
polycyclic subgroup of finite index. 
A {\em discrete\/} virtually solvable group of real matrices is virtually polycyclic.

These examples arise from a more general construction:
namely, $\Gamma$ embeds as a lattice in a closed Lie subgroup $G\subgroup\Aff(\A)$ %. 
with finitely many connected components, 
and  whose identity component $G^0$ acts {\em simply transitively\/} on $\A$.  %

Since $\Gamma^0 := \Gamma\cap G^0$ has finite index in $\Gamma$, 
the flat Lorentz manifold $M^3$ is finitely covered by the homogeneous space
$\Gamma^0\backslash G^0$. 
Necessarily, $G^0$ is simply connected and solvable. 
The group $G^0$ plays the role of the translation group $\R^n$ acting
by translations on $\A^n$.
The group $G$ is called the {\em crystallographic hull\/} in 
Fried-Goldman~\cite{MR689763}.

\subsubsection{Syndetic hulls}\label{sec:SyndeticHulls}
A weaker version of this construction  
was known to H.\ Zassenhaus, H.\ C.\ Wang and L.\ Auslander
(compare Raghunathan~\cite{MR0507234}),
and defined in \cite{MR689763} and improved in Grunewald--Segal~\cite{MR1305981}.

If $\Gamma\subgroup \GLn$ is a solvable group, then a {\em syndetic hull\/} 
for $\Gamma$ is a subgroup $G$ such that:
\begin{itemize}
\item $\Gamma\subgroup G \subgroup \ZClosure{\Gamma}$, 
where we recall that $\ZClosure{\Gamma}\subgroup \GLn$ is the Zariski closure (algebraic hull) of 
$\Gamma$ in $\GLn$;
\item $G$ is a closed subgroup having finitely many connected components;
\item $\Gamma\backslash G$ is compact (although not necessarily Hausdorff).
\end{itemize}
The last condition is sometimes called {\em syndetic,\/} since 
``cocompact'' is usually reserved for subgroups whose coset 
space is compact {\em and Hausdorff.\/} 
(This terminology follows Gottschalk--Hedlund~\cite{MR0074810}.)
Equivalently, $\Gamma \subgroup G$ is syndetic if and only if there exists 
$ K\subset G$ which is compact and meets every left coset
$g\Gamma$, for $g\in G$.

In general, syndetic hulls  fail to be unique. 

\subsubsection{Solvable examples and polynomial structures}
The theory of affine group actions is dramatically different 
for solvable and non-solvable groups.
Milnor~\cite{milnor1977fundamental} 
proved that every virtually polycyclic group admits a proper affine action.
Later Benoist~\cite{MR1316552} found examples of  virtually polycyclic groups
for which no crystallographic affine action exists.
Dekimpe and his collaborators~\cite{MR1422895,MR1439202,MR1676619,MR1757881,MR1774093,
MR1856924,MR1895707,MR1934685}
replace complete affine structures by {\em polynomial structures,\/}
that is, quotients of $\A^n$ by proper actions of discrete subgroups of the
group of polynomial diffeomorphisms $\A^n \longrightarrow \A^n$. 
They show that every virtually polycyclic group admits a polynomial crystallographic action.

Polynomial structures satisfy a suggestive {\em uniqueness property\/}
for affine crystallographic groups similar to the role complete affine structures
play for Euclidean crystallographic groups.
Fried-Goldman~\cite{MR689763} prove that two isomorphic affine crystallographic groups
are {\em polynomially equivalent.\/}

A simple example, seen in Figure~\ref{fig:CompleteTori}, occurs in 
dimension two, where a polynomial diffeomorphism
of degree two,
\begin{align*}
\A^2 &\xrightarrow{~f~} \A^2 \\
(x,y) &\longmapsto (x+y^2, y) \end{align*}
conjugates the affine crystallographic actions of $\Z^2$.
Namely, $f$ conjugates translation $\tau$ by $(u,v)\in\R^2$ to 
the affine transformation
\begin{equation}
f\circ \tau \circ f\inv: p \mapsto \begin{bmatrix}[cc] 
1 & 2 v 
\\ 0 & 1 \end{bmatrix} p +
\begin{bmatrix}[c] 
u + v^2 
\\ v \end{bmatrix}.
\end{equation}
The conjugate $f\V f\inv$ is a simply transitive vector group of  affine transformations, 
where $\V \cong \R^2$ is the group of translations.

For different choices of lattices $\Lambda\subgroup\V$, the group $f\Lambda f\inv$ achieves all 
affine crystallographic actions of $\Z^2$, other than lattices of translations. 
Baues~\cite{MR1796130} showed that the deformation space of {\em marked  complete affine structures\/}
is homeomorphic to $\R^2$. 
The effect of changing the marking is the usual linear action of 
$\GLtZ$, the mapping class group of the torus, on~$\R^2$. 
Compare also Baues-Goldman~\cite{MR2181958}.
These structures were first discussed by Kuiper~\cite{kuiper1953surfaces}.

\begin{figure}{\centering
\subfigure{\includegraphics[scale=.55]{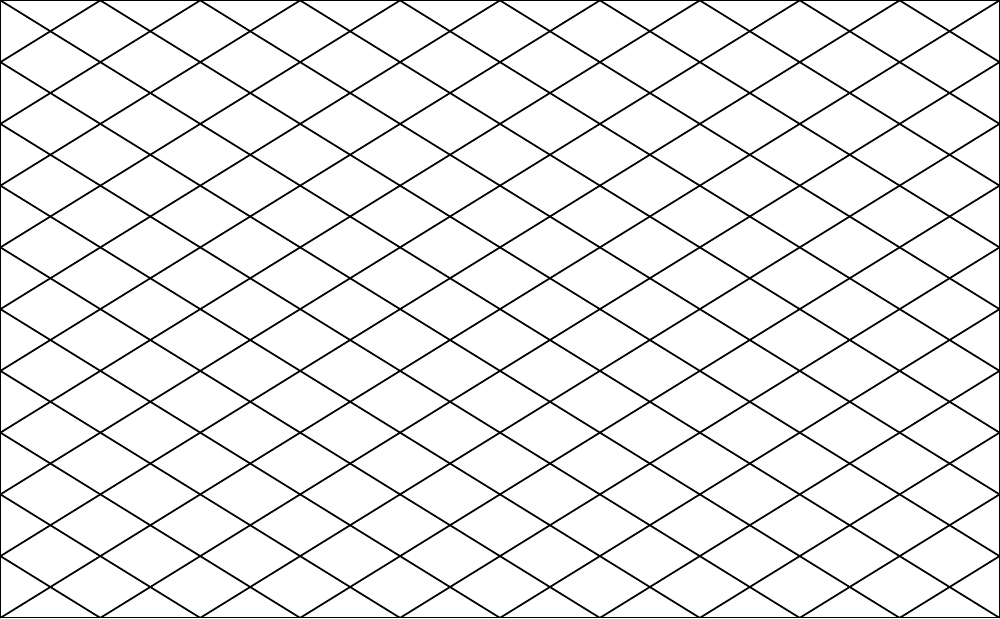}
} \qquad
\subfigure{\includegraphics[scale=.55]{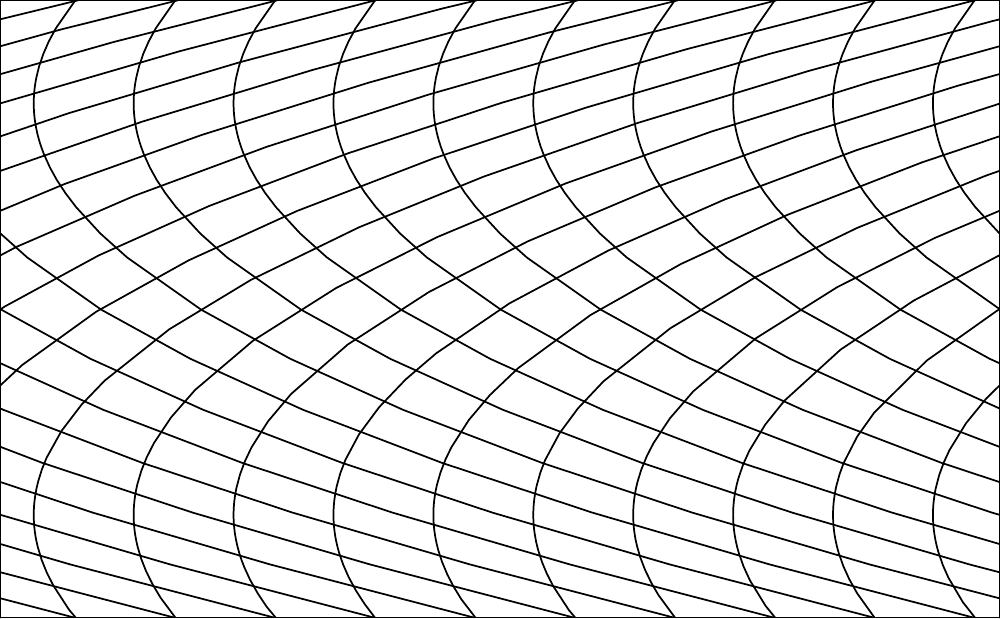}}
}
 \caption{\label{fig:CompleteTori}Tilings corresponding to some complete affine structures on the $2$-torus}
\end{figure}

\subsection{Auslander's conjecture and Milnor's question} 
In \cite{MR0161255},
Auslander asserted that every  discrete subgroup 
$\Gamma\subgroup\Aff(\A^n)$  acting properly and cocompactly on $\A^n$  is virtually solvable.
This was his approach to proving Chern's  conjecture that 
{\em the Euler characteristic of a compact affine manifold vanishes,\/}
in the case the manifold is complete. 
The general theory described in \S\ref{sec:CrystallographicHulls}
implies that if $\Gamma$ is virtually polycyclic,
then $M = \Gamma\backslash\A^n$ has a particularly tractable algebraic structure as a 
{\em solvmanifold,\/} a homogeneous space of a $1$-connected solvable Lie group by a lattice.

Namely, $M$ identifies with the quotient $\Gamma\backslash G$,
where $G$ is a crystallographic hull.
The simply transitive affine action of $G$ on $\A^n$ identifies
$M = \Gamma\backslash\A^n$ with $\Gamma\backslash G$. 
Furthermore,  $M$ admits the finite covering space $(\Gamma\cap G^0)\backslash\A^n$,
which identifies with the solvmanifold $(\Gamma\cap G^0)\backslash G^0$.
This gives a satisfying picture of virtually polycyclic affine crystallographic groups
generalizing Bieberbach's theorem. 
See Grunewald-Segal~\cite{MR1305981} for more details.

Unfortunately, Auslander's proof is incomplete.
His assertion that {\em every affine crystallographic group is virtually polycyclic} 
remains unsolved,
and following Fried-Goldman~\cite{MR689763}, 
has been called the {\em Auslander conjecture.\/}
It is one of the fundamental open questions in the theory of affine manifolds.
The main result of \cite{MR689763} is the proof of this conjecture in dimension $3$.

Vanishing of the Euler characteristic of a complete compact  affine manifold was 
later proved by Kostant-Sullivan~\cite{MR0375341} 
independently of Auslander's conjecture.

\subsubsection{Proper affine actions of $\Ft$}
Affine geometry is significantly more complicated than Euclidean geometry
in that discrete groups of affine transformations need not act properly. 
Suppose 
\[ \Gamma \xhookrightarrow{~\rho~}\Isom(\En) \]  
defines a faithful isometric action of $\Gamma$. 
This action is properly discontinuous (that is, {\em proper\/} with respect to the discrete topology on $\Gamma$) 
if and only if the image of $\Gamma$ is a discrete subgroup of $\Isom(\En)$ 
(that is, $\rho$ is a {\em discrete embedding\/}).
However, if $\rho$ is only affine (that is, the linear part $\L(\Gamma)$ is not assumed to lie in $\On$),
then discrete embeddings do not necessarily define proper actions.

Milnor realized that the assumption of compactness in Auslander's conjecture was 
not necessary to raise an interesting question.
Tits~\cite{MR0286898} proved that every subgroup $\Gamma$ of $\Aff(\A^n)$  of affine transformations is either:
\begin{itemize}
\item  virtually solvable or,
\item contains a subgroup isomorphic to a two-generator free group.
\end{itemize}
If $\Gamma$ is also assumed to be discrete, 
then the first condition of virtual solvability can be strengthened to virtual polycyclicity. 

Milnor then asked whether proper affine actions exist when $\Gamma$ is a 
two-generator free group $\Ft$. 
Nonexistence implies Auslander's conjecture, 
which would result in a satisfying structure theory generalizing the Bieberbach theory. 
Attacking this question requires a criterion  for solvability.

Evidently,
$\Gamma$ is virtually solvable if and only if the {\em Zariski\/} closure $\ZClosure{\rho(\Gamma)}$ 
of $\rho(\Gamma)$ in $\Aff(\A)$ is virtually solvable. 
Since Zariski closed subgroups have finitely many connected components in the 
classical topology, the identity component (in the classical topology) 
$\Big(\ZClosure{\rho(\Gamma)}\Big)^0$ 
is a connected solvable closed (Lie) subgroup,
which has finite index in  $\ZClosure{\rho(\Gamma)}$.
In turn, this is equivalent to  its linear part
\[ G := \LL   \bigg(
\Big(\ZClosure{\rho(\Gamma)}\Big)^0
\bigg) \] 
being a connected solvable closed subgroup of $\GL(\V)$. 

Recall the {\em Levi decomposition:\/}
A connected Lie group is the semidirect product
of a maximal normal solvable connected subgroup, called its \emph{radical\/},
by a semisimple subgroup, called its 
{\em semisimple part\/} or its \emph{Levi subgroup\/}.
In particular, a group is solvable if and only if its semisimple part  is trivial.

Summarizing:

\begin{prop}\label{prop:vsolvable}
Let $\Gamma\hookrightarrow \Aff(\A)$ be an affine representation.
The following conditions are equivalent:
\begin{itemize}
\item $\Gamma$ is virtually solvable;
\item $\L(\Gamma)$ is virtually solvable;
\item $\ZClosure{\L(\Gamma)}$ is virtually solvable;
\item The identity component 
$\big(
\ZClosure{\L(\Gamma)} 
\big)^0$ is solvable;
\item The  semisimple part 
of $\big(
\ZClosure{\L(\Gamma)}\big)^0$ is trivial.
\end{itemize}
\end{prop}
\noindent 
This raises the question of which groups can arise as semisimple parts
of $\big(\ZClosure{\L(\Gamma)}\big)^0$ for a proper affine action
$\Gamma\xrightarrow{~\rho~}\Aff(\A)$ where $\Gamma$ is {\em not\/} virtually solvable. 
Following Smilga, 
we say that a closed connected subgroup $G < \GL(\V)$ is {\em Milnor\/}  if no such proper affine deformation with 
\[ \big(\ZClosure{\L(\Gamma)}  \big)^0 = G\]
exists. 
By Tits~\cite{MR0286898} , 
we can replace $\Gamma$ in the above definition by the two-generator free group $\Ft$.

\begin{defn}\label{def:Milnor}
Let $\rho: G \to \GL(\V)$ be a linear representation of an algebraic group 
$G$ on a vector space $\V$. Then $G \ltimes_\rho \V$ acts 
affinely on $\V$. We call $\rho$ \emph{Milnor} if there does not exist a 
subgroup $\Gamma < G \ltimes_\rho \V$ which is isomorphic to a non-abelian 
free group, has linear part $\L(\Gamma) < G$ Zariski dense, and which acts properly 
discontinuously on $\V$.
\end{defn}

A solvable subgroup is (trivially) Milnor.
Similarly, Bieberbach's structure theorem implies compact groups are Milnor.
Thus Milnor's question can be rephrased as whether non-Milnor subgroups exist.

In fact, even many non-solvable subgroups, for example $\GL(\V)$ and $\SL(\V)$, 
are easily seen to be Milnor from the following
(see Proposition~\ref{prop:EigenvalueOne} below).

\begin{lemma}\label{lem:OneEigenvalue}
Suppose that $\Gamma\xrightarrow{~\rho~}\Aff(\A)$ defines a free
action on $\A$. 
Then every element of $
\ZClosure{\L(\Gamma)}$ has $1$ as an eigenvalue.
\end{lemma}

\noindent
This lemma was first used by Kostant-Sullivan~\cite{MR0375341} in the proof that the
Euler characteristic of a compact complete affine manifold vanishes.

Lemma~\ref{lem:OneEigenvalue}  follows from two elementary observations:
\begin{itemize}
\item 
If $g\in\Aff(\A)$ acts freely on $\A$,  then $\L(g)\in\GL(\V)$ has $1$ as an eigenvalue;
\item
The condition that $A\in\GL(\V)$ has $1$ as an eigenvalue,
namely that $\det(A - \Id) = 0$ is a polynomial condition on $A$ and thus passes
to the Zariski closure.
\end{itemize}

\noindent Summarizing:
\begin{prop}\label{prop:EigenvalueOne} 
Suppose that $G\subgroup\GL(\V)$ is not Milnor.
Then every element of $G$ has $1$ as an eigenvalue.
\end{prop} 
       
\subsubsection{Complete affine $3$-manifolds}

Fried--Goldman~\cite{MR689763} classify which connected semisimple subgroups $G$ 
can arise as semisimple parts of $\ZClosure{\L(\Gamma)}$, 
when $\dim\,\V = 3$. 
It follows from their work that the only connected semisimple subgroup $G\subgroup\GLthreeR$ 
which is not Milnor is $G= \SOoto$. We recall the argument here.

By an easy 
calculation, 
the only connected semisimple subgroups of $\GLthreeR$ are (up to conjugacy):
\begin{itemize}
\item $\SL(3,\R)$;
\item $\o{SO}(3)$;
\item $\SL(2,\R)$;
\item $\SOoto$,
\end{itemize}
embedded in the standard ways.
 We have already excluded the case $G= \SL(3,\R)$.
The case $G=\o{SO}(3)$ is excluded by the Bieberbach theorems 
(since $\L(\Gamma)$ must be a finite group, $G$ is trivial).

Suppose $G = \SL(2,\R)$.
Then $\L(\Gamma)$ can be conjugated into one of the forms:
\[
\bmatrix   * & * & * \\ 0 & * & * \\ 0 & * & * \endbmatrix, \ 
\bmatrix   * & 0 & 0 \\ * & * & * \\ * & * & * \endbmatrix 
\]
The condition that these matrices have $1$ as an eigenvalue implies that
the $(1,1)$-entry equals $1$. 

In the first  case, the vector field $\partial/\partial x$ is a 
$\Gamma$-invariant parallel vector field which descends to a parallel vector field on $M$.
In the second case, the $1$-form $dx$ is a $\Gamma$-invariant parallel $1$-form
which descends to a parallel $1$-form on $M$. 

These cases are eliminated as follows.
A parallel $1$-form can be perturbed to have rational periods, and integrates to give 
a fibration of $M$ over $S^1$ with fibers closed complete affine $2$-manifolds,
from which the virtual solvability follows by the $2$-dimensional case. 
In the case of a parallel vector field $\xi$, the Zariski density implies the existence 
of two elements $\gamma_1,\gamma_2$ which generate a nonabelian free group, and correspond to
closed orbits of the flow of $\xi$. These closed orbits are hyperbolic in the sense of 
hyperbolic dynamics but their stable manifolds intersect (by lifting them to $\A$),
which is a contradiction.  
See Fried-Goldman~\cite{MR689763} for further details.

Finally, consider the most interesting case, namely $G = \SOoto$. 
Margulis's breakthrough~\cite{margulis1983free}, \cite{margulis1987complete}
may be restated that $\SOoto$ is {\em not\/} Milnor. 
Thus $\SOoto$ was the first example of a non-Milnor group.
Suppose $M^3 = \Gamma\backslash\A^3$ is a complete affine $3$-manifold whose fundamental 
group $\Gamma$ is nonsolvable.
By the above, $\ZClosure{\L(\Gamma)}$ is (conjugate to) $\SOoto$.
Hence, the $\Oto$-invariant inner product on $\V^3$ defines a flat Lorentzian 
metric on $\A^3$ invariant under the action of %
$\Gamma$. 
When equipped with this metric, we denote the affine space by $\Eto$. 
Hence, $M^3 = \Gamma\backslash\Eto$ inherits a {\em flat Lorentzian structure\/} 
from the $\SOto$-invariant Lorentzian
inner product on $\V$.
Such a complete flat Lorentzian three-manifold $M^3$ is called a {\em Margulis spacetime.\/}

\subsection{The associated hyperbolic surface}

\begin{prop}[Fried-Goldman~\cite{MR689763}]\label{thm:DiscreteEmbedding}
Suppose that $\Gamma\subgroup\Isom(\Eto)$ is discrete, and acts properly on $\Eto$.
Either $\Gamma$ is virtually polycyclic or 
$\Gamma \xrightarrow{~\L~} \Oto$ is an isomorphism of $\Gamma$
onto a discrete subgroup of $\SOto < \Oto$.
\end{prop}
\noindent
Selberg~\cite{MR0130324} proved every finitely generated matrix group contains a 
torsion-free subgroup of finite index 
(compare Raghunathan~\cite{MR0507234}, Corollary~6.13).
Thus $\Gamma$ contains a finite index subgroup containing no elliptic elements. 
Henceforth, we restrict to torsion-free discrete subgroups.

Hyperbolic geometry enters here, as
$\SOto$ is the isometry group of $\Ht$  and every discrete subgroup
of $\SOto$ acts properly on $\Ht$.
The quotient 
\[ \Sigma^2 := \L(\Gamma)\backslash\Ht\]
is a complete hyperbolic surface. 
Since $M^3$ and $\Sigma^2$ are both quotients of contractible spaces by proper and free actions of $\Gamma$, 
$M^3$ and $\Sigma^2$ are homotopy equivalent.
We call $\Sigma^2$ the {\em hyperbolic surface associated to $M^3$,\/}
and $M^3$ an {\em affine deformation\/} of $\Sigma^2$. 

\subsubsection{Margulis spacetimes are not closed}\label{sec:CohomologicalDimension}

Note that $M^3$ can not be compact by the following cohomological dimension argument. 
If $M$ were compact then
\[
2 = \dim(\Sigma)  \ge \cd\big(\L(\Gamma)\big) = \cd (\Gamma) = \dim(M^3) = 3 . \]
This contradiction completes the proof of Auslander's conjecture in dimension $3$.

Using similar arguments,
Goldman--Kamishima~\cite{MR739789} proved Auslander's conjecture 
for flat Lorentzian manifolds \big(linear holonomy in $\OnOne$\big) 
and Grunewald--Margulis~\cite{MR1075720} proved Auslander's conjecture 
for affine deformations for which the linear holonomy lies in other rank $1$ subgroups.
See also Tomanov~\cite{MR1062291,MR1072918}.

Around 1990, Mess proved:

\begin{thm}[Mess~\cite{MR2328921}]\label{thm:Mess}
The fundamental group of a closed surface admits no proper affine action on $\A^3$.
\end{thm}
\noindent
In particular, $\Sigma$ is not a closed surface and $\Gamma_0$ is not a uniform lattice.
In fact,  $\Gamma$ must be a free group:
every Margulis spacetime is an affine deformation of a 
{\em noncompact\/}  complete hyperbolic surface $\Sigma$.

In 1999, Goldman-Margulis~\cite{MR1796129}  gave alternate proofs of Theorem~\ref{thm:Mess}; 
compare the discussion in \S\ref{sec:MargulisLength}.
Later, Labourie~\cite{MR1909247}  and Danciger-Zhang~\cite{dz} 
generalized Mess's theorem to show that for a certain class of linear surface 
group representations, called {\em Hitchin representations,} 
affine deformations are never proper. 
On the other hand, proper affine actions by surface groups do exist by recent 
work of Danciger--Gu\'eritaud--Kassel, see Theorem~\ref{thm:racg}.
(Compare \cite{dgk-racg,Smi16b}.)

\subsubsection{Affine deformations of hyperbolic surfaces}
Sections 2--5 discuss the geometry and classification of Margulis spacetimes.
To facilitate this discussion, we use Proposition~\ref{thm:DiscreteEmbedding} to recast
these questions in a more convenient form.

Suppose that $M^3 = \Gamma\backslash\Eto$ is a Margulis spacetime,
where $\Gamma\subgroup \Isom^+(\Eto)$ is a discrete subgroup acting properly on $\Eto$.
By Proposition~\ref{thm:DiscreteEmbedding}, 
we can assume the linear holonomy group $\L(\Gamma)$ is a Fuchsian subgroup 
$\Gamma_0 \subgroup \SOto$. 

Fix $\Gamma_0$ and consider $\Gamma$ as an affine deformation of $\Gamma_0$.
Affine deformations of $\Gamma_0$ are determined by the translational part
$\Gamma_0 \xrightarrow{~\u~} \Rto$, 
and we denote the affine deformation determined by the cocycle $\u$ by $\Gamma_\u$.
In particular, the zero cocycle determines $\L(\Gamma)$,
so the notation for $\Gamma_0 = \L(\Gamma)$ is consistent.
Translational conjugacy classes of affine deformations form the vector space 
$\HHh^1(\Gamma_0,\Rto)$, which has dimension $3(r-1)$ if $\Gamma_0$ is a free group
of rank $r>1$. 

More geometrically,  consider $M^3$ to be an affine deformation of the hyperbolic surface
$\Sigma$. Then identify $\HHh^1(\Gamma_0,\Rto)$ with the cohomology
$\HHh^1(\Sigma,\V)$, where $\V$ denotes the local system (flat vector bundle) over $\Sigma$
determined by the linear holonomy homomorphism
\[
\pi_1(\Sigma) \xrightarrow{~\cong~} \Gamma_0 \subgroup \Isom(\Rto) = \Oto. \]
The main goal now becomes determining which elements of the
vector space $\HHh^1(\Gamma_0,\Rto)$ determine {\em proper affine deformations.\/}
This was foreshadowed in Milnor~\cite{{milnor1977fundamental}}, 
where he proposed a possible way of constructing proper affine actions of 
non-virtually solvable groups:
\begin{quote}{\sl ``Start with a free discrete subgroup of $\Oto$ and add 
translation components to obtain a group of affine transformations which  acts 
freely.                                                                         
However it seems difficult to decide whether the resulting group action is 
properly discontinuous.''}
\end{quote}
In retrospect, these are the only ways of constructing such actions in dimension three.

\section{Construction of Margulis spacetimes}
We turn now to a direct construction of Margulis spacetimes via fundamental 
domains bounded by piecewise linear surfaces called \emph{crooked planes}. 
While Margulis's original examples were constructed from a dynamical point of view 
(we defer discussion of his original proof until Section~\ref{sec:OriginalProof}), 
crooked fundamental domains bring a geometric perspective to the subject. Their 
introduction by Drumm in 1990 launched a classification program for Margulis spacetimes 
that was completed recently by Danciger-Gu\'eritaud-Kassel 
(see Section~\ref{sec:strips-and-CPs}).

Continuing with the discussion of the previous section, we consider affine 
deformations of Fuchsian subgroups $\Gamma_0 < \SOto$.
Associate to $\Gamma_0$  a complete hyperbolic surface $\Sigma = \Gamma_0\backslash\Ht$. In order for an affine deformation to have a chance at being proper, 
$\Sigma$ must be non-compact, and we assume this going forward.
The section will give a brief overview of the fundamentals of crooked geometry leading to a discussion of Drumm's theorem that every complete non-compact hyperbolic surface $\Sigma$ admits a proper affine deformation as a Margulis spacetime.
This requires models for both the hyperbolic plane $\Ht$
and the $3$-dimensional Lorentzian affine  space,
which we call {\em Minkowski space $\Eto$.}
(Geometrically, $\Eto$ is characterized as the unique simply connected, geodesically complete,
flat, Lorentzian manifold.) 
We motivate the discussion by relating $3$-dimensional Lorentzian geometry
to $2$-dimensional hyperbolic geometry.

After discussing these models, 
we introduce {\em crooked halfspaces\/}
to build fundamental polyhedra for Margulis spacetimes.
Our exposition follows Burelle-Charette-Drumm-Goldman~\cite{MR3262435}
as modified by Danciger-Gu\'eritaud-Kassel~\cite{MR3480555}. 
Schottky's classical construction of hyperbolic surfaces   
(now called {\em ping-pong\/}) extends to crooked geometry, 
giving a geometric construction of Margulis spacetimes.  
This was first developed by Drumm~\cite{MR2638637,MR1191372,MR1243791}.
For more details and background, see \cite{MR2987620},  \cite{MR3379819} 
and \cite{MR3379833}.

Whereas in the initial examples of Margulis the topology of the quotients is unclear, 
Margulis spacetimes which have a crooked fundamental polyhedron are topologically
equivalent to a solid handlebody, and thus \emph{topologically tame}. The tameness 
of all Margulis spacetimes is discussed at the end of \S 2.

\subsection{The geometry of $\Ht$ and $\Eto$}
This introductory section describes the geometry of the hyperbolic plane and its relation to
the Lorentzian geometry of Minkowski $3$-space. 
In particular, we discuss the basic geometric objects needed to build hyperbolic surfaces
and their extensions to Minkowski space. 
Then we discuss the classical theory of Schottky groups 
which, in the next section, we extend to proper affine deformations of Fuchsian groups.
\subsubsection{The projective model for $\Ht$}
Start with the familiar model of $\Ht$ as the upper halfplane in $\C$, consisting
of $x + i y \in \C$ with $y>0$ with the Poincar\'e metric. 
The group $\PSLtwoR$ acts on $\Ht$ by linear fractional transformations, 
comprises all orientation-preserving isometries, and 
 is the identity component of $\Isom(\Ht)$. The  
complement of $\PSLtwoR$ in $\Isom(\Ht)$ is the other connected component,
comprised of orientation-reversing isometries. 

A natural model for the Lorentzian vector space $\Rto$ is the set of
{\em Killing vector fields\/} on $\Ht$, or the  
the Lie algebra $\sltwoR$ of $\PSLtwoR$. The Lie algebra $\sltwoR$ is identified with
the set of  traceless $2\times 2$ real matrices, and the action of $\Isom(\Ht)$  on 
$\sltwoR$ is  by  $\Ad$, the adjoint representation (see Section~\ref{subsec:deform}).
The (indefinite) 
inner product on $\sltwoR$ is defined by:
\[ 
\vv \cdot \vw := \frac12 \tr(\vv \vw);\]
this is $1/8$ the Killing form on $\sltwoR$. 
The basis 
\[
  \vx_1 = \begin{bmatrix} 0 & 1 \\ 1 & 0 \end{bmatrix}    , \,
  \vx_2 = \begin{bmatrix} 1 & 0 \\ 0 & -1 \end{bmatrix}    , \,  
  \vx_3 = \begin{bmatrix} 0 & -1 \\ 1 & 0 \end{bmatrix}   .
\]
is {\em Lorentzian-orthonormal\/} in the sense that
\[ 
\vx_1 \cdot  \vx_1  = \vx_2\cdot  \vx_2  =1,  
\quad \vx_3 \cdot  \vx_3  = -1, \]
and $\vx_i \cdot \vx_j  =0$ for $i \neq j$.

A natural model for $\Ht$ is one of the two components of the quadric
\[ 
\vu \cdot \vu = -1. \] 
A natural isometry from the upper half-plane 
$ \{ x + i y\in \C \mid  y > 0 \}$ with the Poincar\'e metric,  
to the Lie algebra $\sltwoR \longleftrightarrow \Rto$ with the above inner product is:
\[
\begin{array}{ccccc}
\Ht &\longrightarrow & \sltwoR & \longleftrightarrow & \Rto \\
x + i y & \longmapsto & \frac{1}{y}
\bmatrix x & - (x^2 + y^2) \\ 1 & -x \endbmatrix  
&\longleftrightarrow &
\frac{1}{2y} \bmatrix 
 1- x^2-y^2  \\  2x \\ 1+ x^2+y^2  \endbmatrix.
\end{array}
\]
Here the vector on the right-hand side represents the coordinates with 
respect to the basis $\vx_1,\vx_2,\vx_3$ of $\Rto$:
\[
 \frac{1}{y}
\bmatrix x & - (x^2 + y^2) \\ 1 & -x \endbmatrix  =    
\frac{1- x^2-y^2}{2y}\ \vx_1 \ + \  \frac{x}{y}\   \vx_2 \ +\   
\frac{1+ x^2+y^2}{2y}\   \vx_3.
\]
Note that the Lie algebra $\sltwoR$ comes naturally equipped with an orientation, 
defined in terms of the bracket. Indeed we define the ordered basis 
$(\vx_1, \vx_2, \vx_3)$ to be positive since $[\vx_1, \vx_2] = + 2\vx_3$. 
This orientation of $\sltwoR$ is naturally associated with an orientation of $\Ht$, 
namely the orientation for which $\vx_3$ is an infinitesimal rotation in the 
positive direction. Note that the orientation reversing isometries of the 
Lorentzian structure on $\sltwoR$ flip the sign of the Lie bracket.
The adjoint action of an orientation reversing isometry of $\Ht$ preserves the 
Lie bracket and hence the orientation of $\sltwoR$, however it exchanges 
$\Ht$ with the other component of the quadric $\vu \cdot \vu = -1$.

In relativistic terminology,
a vector $\vv\in \Rto \setminus \{ \mathbf{0} \}$ is called
{\em spacelike\/} if
$\vv\cdot\vv$ is positive, {\em timelike\/} if it is negative, {\em null\/} or {\em lightlike\/} if it is zero.
A spacelike (respectively, timelike) vector $\vv$ is
{\em unit-spacelike\/} (respectively,  {\em unit-timelike\/}) 
if and only if $\vv\cdot \vv = 1$ (respectively, $\vv\cdot\vv = -1$).
A Killing vector field $\xi\in\sltwoR$ is spacelike (respectively, null, timelike)
 if and only it generates a hyperbolic (respectively, parabolic, elliptic)
 one-parameter group of isometries. 

The set of null vectors (including $\mathbf{0}$) is a cone, called the 
\emph{light cone} and denoted $\mathcal{N}$. The set $\mathcal{N}\setminus\mathbf{0}$
has two components, or {\em nappes\/}. 

Choosing a preferred nappe is equivalent to choosing a {\em time orientation.\/}
For example, we choose lightlike vectors with $v_3>0$ to be   {\em future  \/} pointing
and lightlike vectors  with  $v_3 < 0$ to be {\em past pointing.\/} The connected 
components of timelike vectors are similarly defined to be future and past pointing.
One model for $\Ht$, already described above, is the space of  \emph{future-pointing} unit-timelike  vectors.

An equivalent model for $\Ht$ is the subset of the projective space 
$\P(\Rto)$ comprised of timelike lines, with  $\partial\Ht$ the set of null lines. 
Spacelike vectors $\vw\in\Rto$ determine geodesics and halfplanes in $\Ht$ in 
the projective model as follows: 
\[
\h{\vw} := \{ \vv\in\Ht \mid \vv\cdot \vw > 0 \} \]
is the open halfplane defined by $\vw$. The  boundary  
$\partial\h{\vw}=\Ht \cap \vw^\perp$ is the geodesic 
corresponding to $\vw$. The orientation of $\Ht$ together with $\h{\vw}$
determines a natural orientation on $\partial\h{\vw}$. 
That is,  unit-spacelike vectors in $\Rto = \sltwoR$ correspond to {\em oriented geodesics\/} 
in $\Ht$.
Note that as a Killing vector field, $\vw$ is an infinitesimal translation 
along $\partial\h{\vw}$ in the positive direction. Generally, the Killing field on $\Ht$ associated to $\vw \in \Rto$ is given explicitly in terms of the Lorentzian cross-product on $\Rto$, which is just the Lie bracket $[\cdot, \cdot]$ on $\slt$, by restricting the vector field $x \mapsto [\vw, x]$ to the hyperboloid of timelike future pointing vectors. 

\subsubsection{Cylinders and fundamental slabs}\label{sec:cylinders}
A basic hyperbolic surface is a {\em hyperbolic cylinder,\/} arising as the quotient
$\Sigma := \langle A\rangle\backslash \Ht$ where $A\in\Isom(\Ht)$ is an isometry 
which is \emph{hyperbolic},
meaning $A$ leaves invariant a unique geodesic $l_A$
along which it translates by a distance $\ell_A$. 
The  image $\langle A\rangle\backslash l_A$ in $\langle A\rangle\backslash\Ht$ is a closed geodesic in $\Sigma$ of length $\ell_A$. 
The scalar invariant $\ell_A$ completely describes the isometry type of $\Sigma$.

One can build fundamental domains for the action of $\langle A\rangle$  as follows. 
Choose any geodesic $l_0\in\Ht$ meeting $l_A$ in a point $a_0\in\Ht$,
and let $\Hh_0$ be the halfplane bounded by $l_0$ containing $A(a_0)$.
Then $A(\Hh_0) \subset \Hh_0$ and
the complement 
\[ 
\Delta := \Hh_0 \setminus A(\Hh_0)\]  is a fundamental domain
for the cyclic group $\langle A\rangle$ acting on $\Ht$. 
If $\vw$ is unit-spacelike and $\Hh_0 = \h{\vw}$ is the open half-plane defined by 
$\vw$ as above, then the fundamental domain takes the form
\[ \Delta = \h{\vw} \cap \h{{-A(\vw)}}\]
and we call $\Delta$ a {\em fundamental slab\/} for $\langle A\rangle$.

\subsubsection{Affine deformations of cylinders}\label{sec:AffDefCylinders}
Recall the {\em Minkowski $3$-space,\/} $\Eto$, namely the complete $1$-connected 
flat Lorentzian manifold in dimension three.
Equivalently, $\Eto$ is an affine space whose underlying vector space is equipped
with a Lorentzian inner product.
As above, we model the underlying Lorentzian vector space $\Rto$ on the 
Lie algebra
$\sltwoR$ of Killing vector fields on $\Ht$.

The group of {\em linear\/} orientation preserving isometries $\Isom(\Rto)$ equals 
the special orthogonal
group $\SOto\cong \Isom(\Ht)$.
Its identity component $\Isom^+(\Ht)$, comprising the orientation preserving isometries of $\Ht$, is naturally identified via the adjoint representation with $\PSLtwoR$.  %
A hyperbolic element $A\in\PSLtwoR$  pointwise fixes one spacelike line, 
and this  line contains exactly two unit-spacelike vectors that are 
negatives of each other.

In terms of Killing vector fields, 
the  line fixed by $A$ is just the infinitesimal centralizer of $A$. 
Indeed, 
\[ 
A = \exp\left( \frac{\ell_A}{2} \ \vw_A \right) 
\sim \bmatrix   e^{\ell_A/2} & 0 \\ 0 & e^{-\ell_A/2} \endbmatrix \] 
where $\vw_A$ is one of the two unit-spacelike generators of this line, and $\ell_A > 0$ is the translation length of $A$ in $\Ht$. 
Since $\ell_A > 0$, 
$A \mapsto \vw_A$ is well-defined and equivariant under 
the action of $\Isom^+(\Ht)$, in the sense that $w_{BAB^{-1}} = B \vw_A$, for any $B \in \Isom^+(\Ht)$.

Let $g\in\Isom(\Eto)$ be an affine deformation of $A$.  
That is, 
\[
g(p) = A(p) + \u ,
\] 
where the vector $\u$ is the 
{\em translational part\/} of $g$.
There is a unique $g$-invariant line, denoted $\Axis(g)$, 
parallel to the fixed-line $\R \vw_A$ of the linear part $A$ of $g$.
The line $\Axis(g)$ inherits a natural orientation 
induced from $\vw_A$.

The restriction of $g$ to $\Axis(g)$ is a translation,
and the {\em signed\/} displacement along this spacelike geodesic
is a scalar defined by 
\begin{equation}\label{eq:FormulaForMargulisInvariant}
\alpha(g) = \vw_{\L(g)} \cdot \u(g).
\end{equation}
Here, $\L(g):= A$ is the linear part of $g$ and $\u(g)$ is the translational part of $g$, 
as in \eqref{eq:AffineMap}. Clearly,  $g$ acts freely if and only if $\alpha(g)\neq 0$.

\begin{defn}
The scalar quantity $\alpha(g)$ is called the {\em Margulis invariant\/} of $g$.
\end{defn}
\noindent
Through a translational change of coordinates, 
the origin may be located on  $\Axis(g)$ so that
\[
\u(g) = \alpha(g) \vw_{\L(g)}. \]
By an orientation preserving linear change of coordinates, the linear part 
$A = \L(g)$ diagonalizes and $g$ takes the form:
\[
g(x) = \begin{bmatrix}[ccc] 
e^{\ell_A} & 0 & 0  \\
0 & e^{-\ell_A} & 0  \\
0 & 0 & 1  \end{bmatrix}x + \begin{bmatrix}[c] 
 0  \\
0  \\
 \alpha(g)  \end{bmatrix}
\]

\subsubsection{The role of orientation}\label{sec:roleOfOrientation}
We note that the definition of the vector $\vw_A$, and hence of the 
Margulis invariant $\alpha(g)$, depends on more than just the structure of 
$\Rto\longleftrightarrow \sltwoR$ as a Lorentzian vector space. 
It depends, more specifically, on the Lie algebra structure,
where the operation of the {\em Lorentzian cross-product\/} 
is determined entirely by the Lorentzian structure and orientation. 
Margulis' original work does not use the Lie algebra $\sltwoR$. 
There the definition of $\alpha$ is given directly in terms of the Lorentzian structure and a choice  of orientation, as follows. 
The positive direction $\vw_A$ of the $1$-eigenspace of $\L(g)$ is the one making the basis 
$(\vw^+, \vw^-, \vw_A)$ positive, 
where $\vw^+, \vw^-$ denote representatives of the attracting and repelling 
eigenlines of $\L(g)$ which have negative inner product $\vw^+ \cdot \vw^- < 0$.

\subsubsection{Parallel slabs}\label{sec:ParallelSlabs}
Affine lines parallel to $\Axis(g)$ describe a $g$-invariant foliation,
and the foliation $\Ff$ by planes orthogonal to these lines is a  $g$-invariant $2$-dimensional
foliation defined by the $g$-equivariant orthogonal projection 
\[ \Eto \xrightarrow{~\Pi~}  \Axis(g) \longleftrightarrow \R . \]
In particular, since $g$ acts by translation by $\alpha(g)$ on $\Axis(g)$, 
the preimage $\Pi\inv\Big[0,\vert\alpha(g)\vert\Big]$ of the closed interval 
\[\Big[0,\vert\alpha(g)\vert\Big]\ \subset\ \R\]
is a fundamental domain for $\langle g\rangle$.
Since the faces of this fundamental domain are the parallel hyperplanes
$\Pi\inv(0)$ and $\Pi\inv\vert\alpha(g)\vert$,
we call these fundamental domains {\em parallel slabs.\/}

\subsubsection{Schottky groups and ping-pong}

Having discussed actions of the infinite cyclic group $\Z$ on both 
$\Ht$ and $\Eto$, we now turn to non-abelian free groups.
We recall Schottky's \cite{MR1579739} construction of discrete free groups acting  on 
the hyperbolic plane.

For brevity, let us focus on the two generator case.

Suppose $A_1, A_2 \in \PSL(2,\R)$ are hyperbolic elements with respective 
translation axes $l_1$ and $l_2$. Let $\vw_1$ and $\vw_2$ be unit 
spacelike vectors associated to halfspaces $\Hh_{\vw_1}, \Hh_{\vw_2}$, such that for each $i = 1,2$, the boundary of $\Hh_{\vw_i}$ crosses $l_i$ in $\Ht$ and satisfies 
$A_i \cdot \Hh_{\vw_i} \subset \Hh_{\vw_i}$ so that a fundamental 
domain for the cyclic group $\langle A_i\rangle$  is given by 
$\Delta_i = \Hh_{\vw_i} \cap \Hh_{-A_i\vw_i}$, as in Section~\ref{sec:cylinders}.

\begin{prop}[Schottky]
Suppose that the four half-spaces 
\begin{align}\label{halfspaces}
\Hh_{-\vw_1}, \Hh_{A_1 \vw_1}, \Hh_{-\vw_2}, \Hh_{A_2 \vw_2}
\end{align} are pairwise disjoint.
Then $A_1$ and $A_2$ generate a discrete free subgroup 
$\Gamma = \langle A_1, A_2\rangle < \PSL(2,\R)$.
\end{prop}
Let us sketch the proof of this well-known fact. Consider the polygon
\[ \Delta = 
\Hh_{\vw_1} \cap \Hh_{-A_1 \vw_1} \cap \Hh_{\vw_2} \cap \Hh_{-A_2 \vw_2} \subset \Ht, \]
which is bounded by four disjoint lines. Then the image of $\Delta$ 
under any nontrivial reduced word $w$ in $A_1, A_1^{-1}, A_2, A_2^{-1}$ 
lies in one of the four halfspaces~\eqref{halfspaces}. 
Indeed observe the following relations:
\begin{align*}
A_1 \cdot (\Hh_{-\vw_1})^c &= A_1 \cdot \overline{\Hh_{\vw_1} } =  \overline{\Hh_{A_1 \vw_1}}\\
A^{-1}_1 \cdot (\Hh_{A_1 \vw_1})^c &= A_1^{-1} \cdot 
\overline{\Hh_{-A_1\vw_1} } =  \overline{\Hh_{-\vw_1}}\\
A_2 \cdot (\Hh_{-\vw_2})^c &= A_2 \cdot \overline{\Hh_{\vw_2} } =  \overline{\Hh_{A_2 \vw_2}}\\
A_2^{-1} \cdot (\Hh_{A_2\vw_2})^c &= A_2^{-1} \cdot 
\overline{\Hh_{-A_2\vw_2} } =  \overline{\Hh_{-\vw_2}}.
\end{align*}

Then by induction on the length of the reduced word $w$, 
the image of $\Delta$ (the ``ping-pong ball") under the action of $w$ lies: 
\begin{itemize}
\item inside 
$\Hh_{A_1 \vw_1}$ if the first (that is, leftmost) letter of $w$ is $A_1$;
\item inside $\Hh_{-\vw_1}$ if the first letter is $A_1^{-1}$;
\item inside $\Hh_{A_2 \vw_2}$ if the first letter is $A_2$;
\item  inside $\Hh_{-\vw_2}$ if the first letter is $A_2^{-1}$. 
\end{itemize}
This proves the proposition.

Since $\Gamma$ is discrete it acts properly on $\Ht$, however 
$\Delta$ might not be a fundamental domain for the action. 
Indeed, in some cases, 
\[\bigcup_{\gamma \in \Gamma} \gamma\cdot \overline{\Delta}\] is a proper subset of $\Ht$. 
It is always possible, however, to choose a polyhedron $\Delta$ as above which is a fundamental domain.

The disjointness of the halfspaces~\eqref{halfspaces} is essential in this construction.

However in affine space at most {\em two\/} halfspaces can be disjoint.
Hence, a ping-pong construction using affine halfspaces does not work in affine geometry. 
Nonetheless, Schottky fundamental domains in $\A^3 = \Eto$ do exist,
and are constructed from \emph{crooked halfspaces}.

\subsection{Crooked geometry}\label{sec:crooked}
Milnor~\cite{milnor1977fundamental} essentially proposed building proper actions of free 
groups by combining
proper actions of cyclic groups. 
However, deciding whether multiple proper actions by cyclic groups generate a 
proper action of the free product is quite delicate.
As we observed in the previous section, hyperplanes, which are perhaps the most 
natural separating surfaces in affine geometry, are not well suited for a Schottky style 
construction of fundamental domains for free groups. 
Observe that, by contrast to Euclidean geometry, in our Lorentzian setting the linear 
part of an affine transformation dominates the translational part for 
``most'' points. Hence building fundamental polyhedra adapted more to the 
linear part than the 
translational part seems preferable. 

\subsubsection{Crooked planes and crooked halfspaces}
In \cite{MR2638637}, Drumm introduced so-called crooked planes
in order to build fundamental domains for proper affine actions of non-abelian free groups.
A crooked plane disconnects $\Eto$ into two regions, called crooked halfspaces. 
Unlike a linear plane, a crooked plane has a distinguished point, called the 
\emph{vertex}. In particular, crooked planes are not homogeneous. 
Drumm's original construction was given purely in terms of Lorentzian geometry. 
Here, however, we make use of the identification $\Rto \cong \sltwoR$ and 
define them in terms of Killing fields on the hyperbolic plane, following 
Danciger-Gu\'eritaud-Kassel~\cite{MR3480555}.

Let $\vw \in \sltwoR$ be a spacelike unit vector, 
let 
\[ \ell = \ell_{\vw} = \vw^\perp \cap \Ht\]  
be the oriented geodesic 
associated to $\vw$, and let $\vw^+$ and $\vw^-$ be future oriented 
lightlike vectors respectively representing the forward and backward endpoints 
$[\vw^+], [\vw^-]$ of $\ell_\vw$ in $\partial \Ht$. Here we think of the 
ideal boundary $\partial \Ht$ as the projectivized null-cone in $\Rto = \sltwoR$. 
We first define the crooked plane $\CP(\mathbf{0}, \ell)$ with vertex the origin 
$\mathbf{0}$. The crooked plane $\CP(\vx, \ell)$ with vertex $\vx$ is 
just the translate $\CP(\vzero, \ell) + \vx$.

The \emph{crooked plane} $\CP(\mathbf{0}, \ell)$ is the union of three linear pieces, a 
\emph{stem}, and two \emph{wings}, described as follows. See Figure~\ref{fig:crookedplane}.

\begin{figure}[ht]
\centerline{\includegraphics{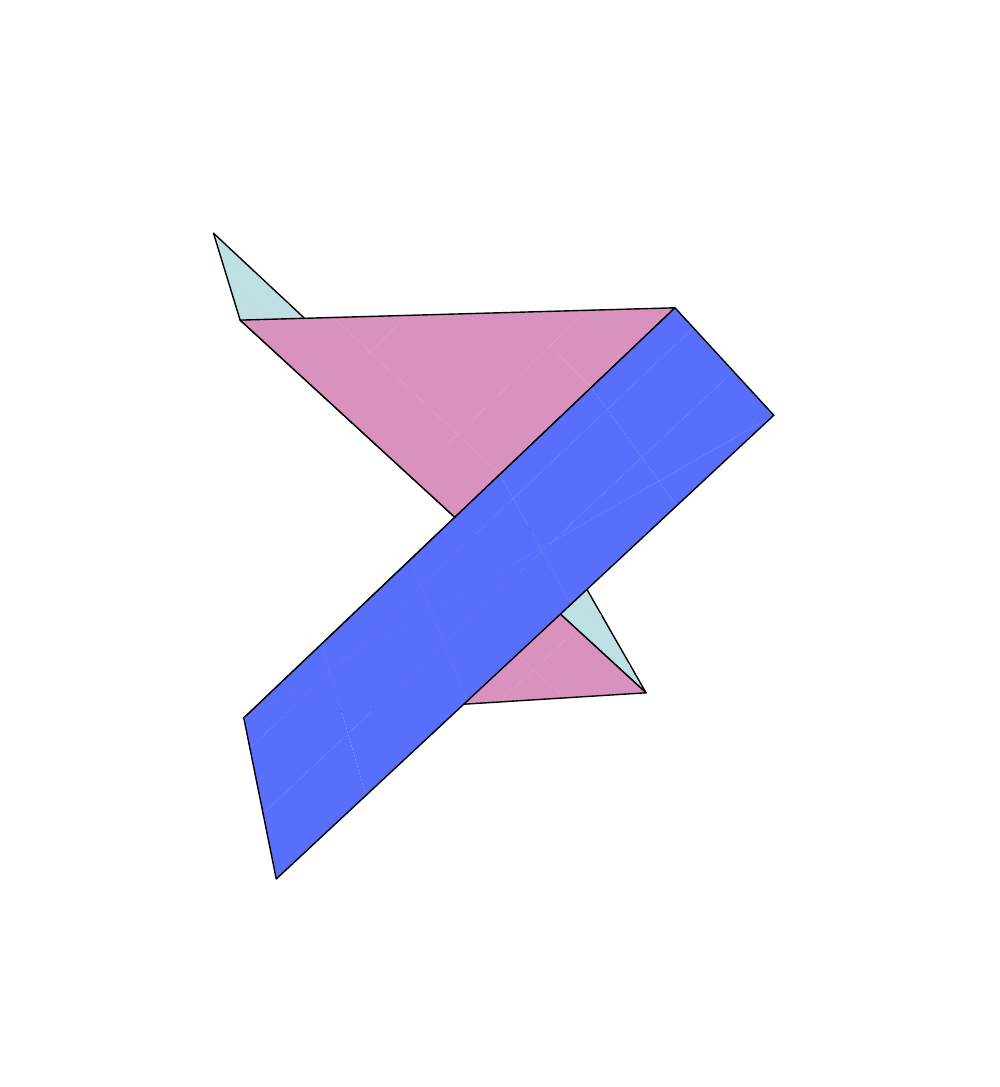}}
\caption{A crooked plane.}
\label{fig:crookedplane}
\end{figure}

\begin{itemize}
\item The stem is the closure of the collection of all elliptic Killing fields whose 
fixed point in $\Ht$ lies on $\ell_\vw$.
\item The wing associated to the forward endpoint $[\vw^+]$ of $\ell_\vw$ is the 
union of the parabolic Killing fields $\R \vw^+$ which fix $[\vw^+]$ and the 
hyperbolic Killing fields for which $[\vw^+]$ is a \emph{repelling} fixed point. 
This wing meets the stem along the hinge $\R \vw^+$.
\item Similarly, the wing associated to $[\vw^-]$ is the union of the parabolic 
Killing fields $\R \vw^-$ fixing $[\vw^-]$ and the hyperbolic Killing fields for 
which
$[\vw^-]$ is a repelling fixed point. 
This wing meets the stem along the hinge $\R \vw^-$.
\item 
The line $\R \vw$, called the {\em spine,\/}
lies in $\CP(\vzero, \vw)$ 
and %
crosses the stem perpendicularly at $\vzero$. 
The positive ray $\R^+ \vw$  
lies %
in the wing associated to $[\vw^-]$ and 
the negative ray $\R^- \vw$ 
lies %
in the wing associated to $[\vw^+]$.
\end{itemize}

More succinctly, the crooked plane $\CP(\vzero, \ell)$ is the collection of 
all Killing fields with a \emph{non-attracting fixed point} on the closure 
\[ \overline{\ell} = \ell \cup \{ [\vw^+], [\vw^-]\}\] 
of $\ell$ in $\overline \Ht.$ 

A crooked plane $\CP(\vzero, \ell)$ divides $\sltwoR$ into two components, 
called crooked halfspaces. 
The \emph{crooked halfspace} $\CH(\vzero, \ell)$ 
is the collection of Killing vector fields with a non-attracting fixed point 
contained in the closure $\overline{\Hh_{\vw}} \subset \overline{\Ht}$ of the 
positive half-plane $\Hh_{\vw}$ bounded by $\ell_{\vw}$. Note that 
\begin{align*}
\CH(\vzero, \ell) \cup \CH(\vzero, -\ell) &= \sltwoR, \\
\CH(\vzero, \ell) \cap \CH(\vzero, -\ell) &= \CP(\vzero, \ell) = 
\CP(\vzero, -\ell)
\end{align*}
where 
$-\ell = \ell_{-\vw}$ is the same geodesic $\ell$, 
but with the opposite orientation.
 
More generally, 
the crooked plane $\CP(\vx, \vw)$ and crooked halfspace $\CH(\vx, \vw)$ with vertex $\vx$ 
are obtained by translating by $\vx$:
\begin{align*}
\CP( \vx, \vw) &:= \vx + \CP( \mathbf{0}, \vw),  \text{ and }\\
\CH( \vx, \vw) &:= \vx + \CH( \mathbf{0}, \vw).
\end{align*}

\subsubsection{Crooked ping-pong}\
The following Lorentzian ping-pong lemma was proved in \cite{MR1191372}.

\begin{lemma}\label{lemma:disjointCPs}
Let $\Gamma = \langle \gamma_1 , \gamma_2 , \ldots , \gamma_n \rangle $ 
be a group in 
$\Isom (\E)$, and  $\{ \CH_{\pm 1 },\CH_{\pm 2 } , ..., \CH_{\pm n } \} $ be 
$2n$ disjoint crooked halfspaces such
that 
\[ \gamma_i ( \CH_{-i} ) = \overline{\E \setminus \CH_{+i}}. \]
Then $\Gamma$ is 
a free group that acts properly on $\E$, with fundamental domain 
\[ 
\Delta\ := \ \overline{\E\setminus \cup_{i=1}^n (\CH_{-i} \cup \CH_{+i})}. \]
In particular,
the quotient $\Gamma\backslash\Eto$ is homeomorphic to an open solid handlebody.
\end{lemma}
The conditions in the lemma immediately imply that the groups satisfying
these conditions act properly on a subset of $\E$. 
The difficult part of the proof is to demonstrate that 
\[
\E \ = \Gamma\left( \Delta \right) \ := \ \bigcup_{\gamma \in \Gamma} \gamma\Delta.
\]  
See Drumm~\cite{MR2638637,MR1191372,MR1243791}, 
Charette-Goldman~\cite{MR1796126}, and 
Danciger-Gu\'eritaud-Kassel~\cite{MR3480555}, Lemma~7.6 (pp.178--179).

Using Lemma~\ref{lemma:disjointCPs}, Drumm proved:
\begin{thm}\label{thm:linearholonomy}  
Every finitely generated free discrete subgroup 
of $\SOto$ admits a proper affine deformation
with a fundamental domain bounded by crooked planes.
\end{thm} 

\subsection{Disjointness of crooked halfspaces and planes}
\label{sec:disjointness}
The application of  Lemma~\ref{lemma:disjointCPs} requires crooked planes to be disjoint.
We now give a criterion for disjointness, originally due to Drumm--Goldman~\cite{MR1660333},
and later conceptually clarified by Burelle--Charette--Drumm--Goldman~\cite{MR3262435}.

Consider a set of pairwise disjoint geodesics $\{ \ell_1, \ell_2 , ..., \ell_n \}$ 
in $\Ht$ which bound 
a common region. The geodesics can be oriented consistently so that 
the interiors of the 
crooked halfspaces $\CH( \bf{0}, \ell_i)$ are disjoint. All of the 
crooked halfspaces meet at the origin $\bf{0}$ and pairs of the 
corresponding crooked planes, boundaries of the crooked
halfspaces, may share a wing.  

Translations $\vu_i$ exist for which the sets $\{ \CH(\vu_i, \ell_i ) \}$ are pairwise disjoint. 
This situation is exactly the one 
described in Lemma~\ref{lemma:disjointCPs}. 
To this end, define the following:
\begin{defn}
For an oriented geodesic $\ell$, the \emph{(open) stem quadrant $Q(\ell)$}  
is the open quadrant of the plane containing the stem of $\CP(\bf{0}, \ell)$ which
lies inside the interior of $\CH(\bf{0}, \ell)$.
\end{defn}

The stem quadrant $Q(\ell)$ is composed of spacelike vectors and bounded by 
two null rays inside the plane that contains the stem  of $\CP(\bf{0},\ell)$. 
In the Lie algebra interpretation, the spacelike vectors in $Q(\ell)$ are
hyperbolic Killing vector fields
whose invariant geodesics are perpendicular to $\ell$ and point into the 
interior of the halfspace defined by $\ell$. The null rays on the boundary 
are the parabolic Killing vector fields whose fixed points are the endpoints 
of $\ell$.

Stem quadrants were defined in \cite{MR3262435} and 
used to show the following:
\begin{lemma}\label{lemma:StemQuadrants}
$\CH(\vu, \ell) \subset \CH(\bf{0},\ell)$
if $\vu\in Q(\ell)$.  %
Furthermore, 
$\CH(\vu_1, \ell_1)$ and $\CH(\vu_2, \ell_2)$ are disjoint if and only if
$\vu_1 - \vu_2\in Q(\ell_1)-Q(\ell_2)$.
\end{lemma}
See Figure~\ref{fig:StemQuad}.
\begin{figure}[ht]
\centerline{\includegraphics[scale=1.0]{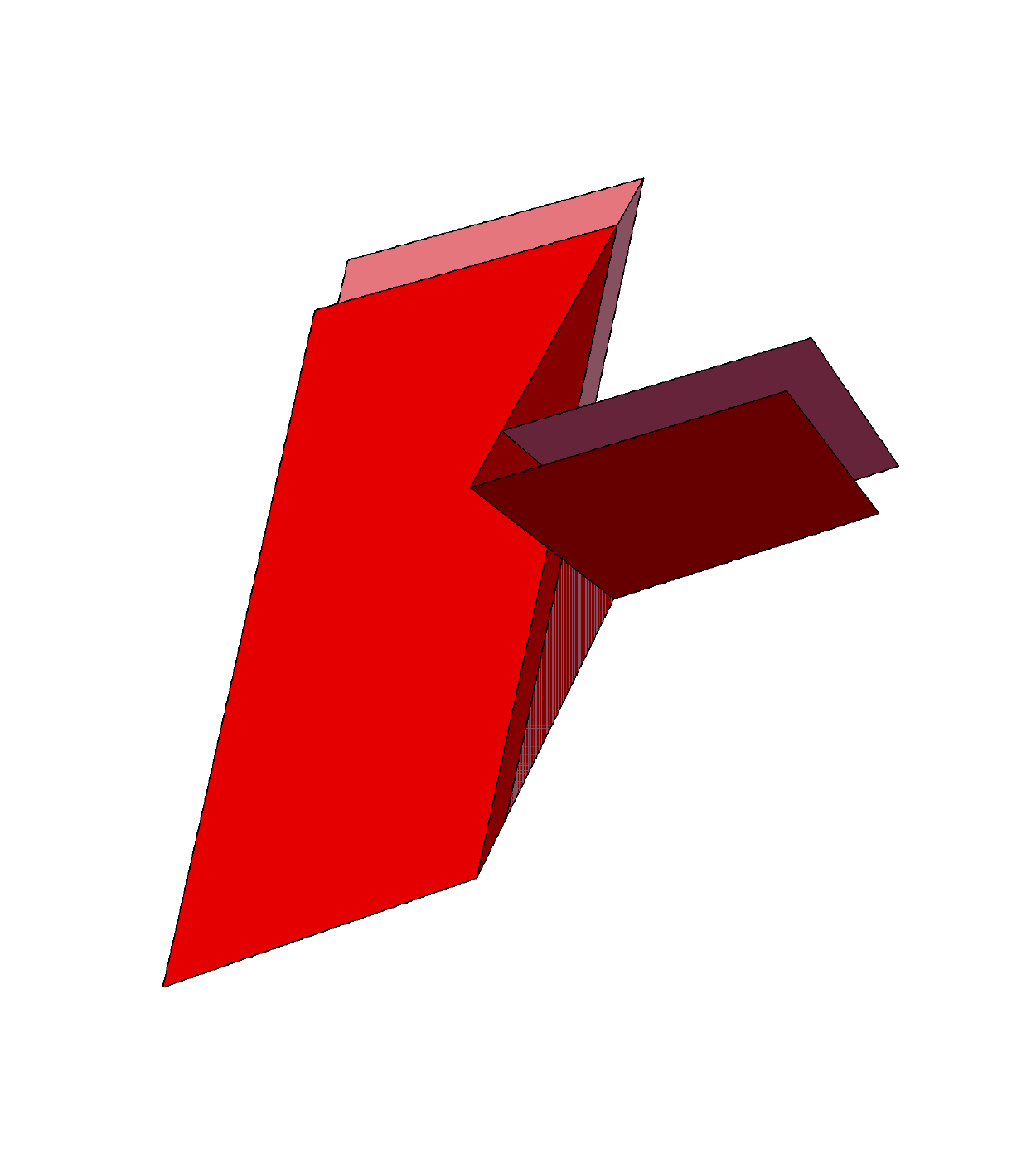}}
\caption{A crooked plane and a translation of the crooked plane by
a vector in the stem quadrant.}
\label{fig:StemQuad}
\end{figure}

In particular, start with the collection of crooked halfspaces 
$\{ \CH(\mathbf{0}, \ell_i ) \}$ whose
interiors are disjoint. Translate each crooked halfpace in a stem quadrant  direction
 $\vu_i\in Q(\ell_i)$, to create a collection  $\{ \CH(\vu_i, \ell_i ) \}$  of disjoint crooked 
halfspaces as in Lemma~\ref{lemma:disjointCPs}.

Burelle-Charette-Drumm-Goldman~\cite{MR3262435} 
introduce foliations by crooked planes;
Burelle-Francoeur~\cite{MR3916271} show that every crooked slab admits a foliation by crooked planes,
answering a question raised by Charette-Kim~\cite{MR3266531}.

\subsection{Tameness}\label{sec:tameness}
A natural question, in a direction converse to Theorem~\ref{thm:linearholonomy},
is whether every Margulis spacetime arises from a crooked polyhedron,
that is, whether Drumm's construction gives all Margulis spacetimes.
This question, first asked by
Drumm--Goldman~\cite{MR1336695},  
motivated much of the recent work on Margulis spacetimes.
This {\em Crooked Plane Conjecture\/} 
was established by Danciger--Gu\'eritaud--Kassel~\cite{MR3480555, dgknew} in general, 
following earlier work for two-generator groups, 
by Charette-Drumm-Goldman~\cite{MR2653729,MR3569564}. 
See \S\ref{sec:strips} for a discussion of these ideas.

This has the following purely topological consequence:

\begin{thm}\label{thm:TopologicalTameness}
A complete affine $3$-manifold with fundamental group $\Gamma$ free of rank $r$ is homeomorphic to 
a handlebody of genus $r$. 
\end{thm}

Theorem~\ref{thm:TopologicalTameness}  is the analog of the {\em Marden Conjecture\/} for hyperbolic $3$-manifolds, 
proved independently by Agol~\cite{Agol-tame} and Calegari-Gabai~\cite{MR2188131}, 
which implies that every complete hyperbolic $3$-manifold with free fundamental group
is homeomorphic to an open solid handlebody.
There are two proofs of Theorem~\ref{thm:TopologicalTameness}  due independently to 
Choi-Drumm-Goldman and to Danciger-Gu\'eritaud-Kassel, which do not use crooked 
planes and which preceded the resolution of the Crooked Plane Conjecture.

Choi--Goldman~\cite{MR3636632} proved Theorem~\ref{thm:TopologicalTameness} 
in the case that $\Sigma$ has  compact convex core (i.e. the linear holonomy 
group is convex cocompact). 
This was later extended by Choi--Drumm--Goldman~\cite{ChDG-tame} to include 
the case that $\Sigma$ has cusps.
The proof involves compactifying
a Margulis spacetime $M^3$  with convex cocompact linear holonomy,
as an $\rpthree$-manifold with geodesic (ideal) boundary. 
The boundary is an $\rptwo$-manifold obtained by grafting annuli
to two copies of $\Sigma$ along its boundary, as in Goldman~\cite{MR882826} 
and Choi~\cite{MR1779499}. 
The boundary $\rptwo$ surface is naturally the quotient of a domain in the 
projective sphere at infinity for $\Eto$. 
Given that the $\Gamma$ action on $\Eto$ and on this domain at infinity are 
both proper, the difficulty lies in proving that the 
$\Gamma$-action on the union is also proper.
This is accomplished by using the dynamics of the lifted geodesic flow as
in Goldman-Labourie-Margulis~\cite{MR2600870} and the fact that the
linear holonomy group $\Gamma_0$ acts on $\overline{\Ht}$ as a 
{\em convergence group.\/} 
When $\Gamma_0$ is no longer convex cocompact (but still finitely generated),
then the proof requires a detailed technical analysis of the geometry near a cusp.

From a different point of view, Danciger--Gu\'eritaud--Kassel~\cite{MR3465975,dgknew} 
proved (Proposition~\ref{prop:proper}) that any Margulis spacetime $M$ is fibered 
in affine (timelike) lines over the associated surface $\Sigma$.
This also gives a proof of Theorem~\ref{thm:TopologicalTameness}. 
See \S\ref{sec:Deformations} for further discussion.

In another direction, 
Frances~\cite{MR1989275} defines an ideal boundary for Margulis spacetimes,
using the action on the (Lorentzian) conformal compactification $\EinThree$ of Minkowski space,
sometimes called the {\em Einstein Universe.\/}
This extends the local conformal Lorentzian geometry of $\Eto$ in the same way
that the conformal geometry of $S^n$ extends conformal Euclidean geometry on $\En$.
The Einstein Universe is diffeomorphic to the mapping torus of the antipodal map on $S^2$.
Its automorphism group is the projective orthogonal group $\POThreeTwo$.
(Compare \cite{MR2436232,MR3323635}.)

The main result is that the action extends to the conformal boundary in much the same way that actions of discrete isometry groups on hyperbolic $n-1$-space extends to its ideal boundary $S^n$. 
He defines a {\em limit set $\Lambda$\/} such that $\Gamma$ acts properly discontinuously
on the complement $\EinThree\setminus\Lambda$ and describes a compactification
for the quotient $\big(\EinThree\setminus\Lambda\big)/\Gamma$
(which is not a manifold).

\section{The Margulis spectrum}\label{sec:MargulisSpectrum}

\subsection{The marked signed Lorentzian length spectrum}\label{sec:msLlspectrum}

The {\em marked length spectrum\/}  of a hyperbolic surface $\Sigma$ is an important
invariant, which determines the isometry type of $\Sigma$.
Recall that this is the function 
\[ \pi_1(\Sigma) \xrightarrow{~\ell_\Sigma~} \Rnonneg \]
which associates to the homotopy class of a based loop $\gamma$ the infimum
of the lengths of loops (freely) homotopic to $\gamma$. 
When $\Sigma$ is closed, then $\ell(\gamma)$ equals the length of the closed
geodesic in $\Sigma$ homotopic to $\gamma$;
in particular $\ell(\gamma) > 0$. 
In general, $\gamma$ has parabolic holonomy if and only if  $\ell(\gamma) = 0$.

This function is part of a general construction defined on the group
$\Isom(\Ht)$.
The {\em geodesic displacement function\/} 
\begin{equation}\label{eq:GeodesicDisplacementFunction} 
\Isom(\Ht)  \xrightarrow{~\ell~} \mathbb \Rnonneg \end{equation}
associates to $g$ the infimum $d\big(p,g(p)\big),$ where $p\in\Ht$. 
If $g$ is elliptic or parabolic, then $\ell(g) = 0$. 
If $g$ is hyperbolic, then $\ell(g)$ equals the length of the
shortest closed geodesic in the cylinder $\Ht/\langle g\rangle$,
as in \S\ref{sec:cylinders}.

If $M^3 = \Gamma\backslash\Eto$ 
is a Margulis spacetime with associated hyperbolic surface $\Sigma\sim M^3$,
then each homotopy class of closed curve $\gamma\in \Gamma$ with non-parabolic 
holonomy is represented by a unique
spacelike geodesic whose Lorentzian length is $\vert \alpha(\gamma)\vert$, where 
$\alpha$ is defined by
\eqref{eq:FormulaForMargulisInvariant} in \S~\ref{sec:AffDefCylinders}.
More generally, the function
\begin{equation}\label{eq:MargulisInvariant}
\pi_1(M) \xrightarrow{~\alpha~} \R \end{equation}
is an important invariant of $M^3$ called the 
{\em marked Lorentzian length spectrum\/}, and is 
analogous to the marked length spectrum  $\ell_\Sigma$ for the associated
hyperbolic surface.  
As we shall see in the next subsection, the signs of the Margulis invariants 
$\alpha(\gamma)$ also play a central role in the theory.

\subsection{Properties of the Margulis invariant}\label{sec:MargInv}
Margulis defined the function $\alpha$ in 
\cite{margulis1983free}, \cite{margulis1987complete}.
Recall
from \S~\ref{sec:AffDefCylinders} that if $g\in\Isom^+(\Eto)$ is an 
orientation preserving Lorentzian isometry with 
$\LL(g)$ hyperbolic, then $g$ leaves invariant a unique spacelike line $\Axis(g)$
which carries a natural orientation induced from the orientation of $\Eto$, 
see Section~\ref{sec:AffDefCylinders}.
Furthermore, the restriction of $g$ to $\Axis(g)$ is  a translation by a multiple
$\alpha(g) \vw_g$, 
where $\vw_g$ is the unit-spacelike vector parallel to 
$\Axis(g)$ determined by the orientation of $\Axis(g)$. 
The {\em sign\/} of $g$ is defined as the sign of $\alpha(g)\in\R$
(positive, negative, or zero).

Margulis's invariant has the following important properties:
\begin{lemma}\label{lem:alpha}
Suppose  $g \in \Isom^+(\Eto)$  
with $\L(g)$ hyperbolic.
\begin{enumerate}
\item\label{alpha0} $\alpha(g) = 0$ if and only if $g$ has a fixed point.
\item\label{alphainv} $\alpha (g ) = \big(g(p)- p\big) \cdot  \vw_g$  for any 
$p \in \E$.
\item\label{alphaconj} $\alpha (g ) = \alpha ( \eta g \eta^{-1})$ for any 
$\eta\in\Isom(\Eto)$.    
\item\label{alphan} $\alpha (g^n) = | n | \alpha (g)$ for $n\neq 0$. % 
\end{enumerate}
\end{lemma}

\noindent
While  the definition of $\alpha (\gamma)$ provides the conceptual meaning  
of the Margulis  invariant, 
\ref{lem:alpha}.(\ref{alphainv}) is a useful formula for its computation.
\ref{lem:alpha}.(\ref{alphan}) implies that the sign of a power is independent of the exponent,
and, in particular, 
\begin{equation}
\label{eq:scalar_margulis_invariant_of_inverse}
\alpha ( \gamma^{-1}) = \alpha( \gamma).
\end{equation}
The four properties of Lemma~\ref{lem:alpha} are elementary.
In contrast, the following {\em Opposite Sign Lemma\/} is deep, 
playing an important role in characterizing proper affine deformations (Theorem~\ref{thm:GLM}).

\begin{thm}[Opposite Sign Lemma]\label{thm:OppositeSign}
If $g,h $ are 
isometries with hyperbolic  linear part, with opposite signs, that is, 
$\alpha(g)\alpha(h) \leq 0$,
then $\langle g, h\rangle$ does {\em not\/} act properly on $\A^3$.
\end{thm}
\noindent
Abels's survey paper~\cite{MR1866854} 
provides a detailed proof of Margulis's Opposite Sign Lemma,
along the lines of the original proof in 
(\cite{margulis1983free}, \cite{margulis1987complete}).

\subsection{Margulis's original construction}\label{sec:OriginalProof}
The Margulis invariant is also key in his original construction of proper
affine deformations of free discrete groups in $\Gamma_0\subgroup\SOto$.

To that end, first
define the {\em hyperbolicity\/} of a hyperbolic element 
$g \in \mathsf{SO}(2,1)$ as the Euclidean distance
\[
d( S^2 \cap \langle g^+ \rangle, S^2 \cap \langle g^- \rangle) 
\]
where $S^2$ is the Euclidean unit sphere, and $\langle g^+ \rangle$ 
(respectively  $\langle g^- \rangle$) is the attracting 
(respectively repelling) eigenline for $g$.
Hyperbolicity  is related to the distance of a fixed basepoint $0\in \Ht$
to the invariant geodesic $l_g \subset \Ht$ of $g$. 
Call an element {\em $\epsilon$-hyperbolic\/} if its hyperbolicity is greater that $\epsilon$.

Moreover, two elements $g,h\in\mathsf{SO}(2,1)$ are said to be $\epsilon$-transverse if they
are $\epsilon$-hyperbolic and 
\[
d( S^2 \cap \langle g^{\pm} \rangle, S^2 \cap \langle h^{\pm} \rangle)  > \epsilon .
\]

Margulis showed that, for any two $\epsilon$-hyperbolic, $\epsilon$-transverse elements 
$g, h \in \mathsf{SO}^+(2,1)$ that are ``sufficiently contracting''
(this basically means that their largest eigenvalues are sufficiently large), we have
\begin{equation}
\label{eq:scalar_margulis_invariant_additivity}
\alpha(gh) \approx \alpha(g) + \alpha(h).
\end{equation}
 
Now consider a free, two-generator discrete group
$\Gamma_0\subgroup\SO^+(2,1)$ whose limit set $\Lambda$ is not all of $\partial\Ht$
(equivalently, $\Gamma_0$ is not a lattice). 
Then there exists 
\[\eta\in \SO^+(2,1) / \Gamma_0 \]
so that every element in the coset $\eta\Gamma_0$ is $\epsilon$-hyperbolic. 
(In particular $\eta$ is $\epsilon$-hyperbolic, with attracting fixed 
point outside of $\Lambda$.)
Then using \eqref{eq:scalar_margulis_invariant_of_inverse} and
\eqref{eq:scalar_margulis_invariant_additivity}, Margulis showed that,  
for an affine deformation $\Gamma$ whose translational parts of the 
generators satisfy a suitable condition,
$\vert \alpha (\eta \gamma)\vert$  
grows roughly like the word-length of $\gamma$, for $\gamma\in \Gamma$.
Once the hyperbolicity is bounded below by $\epsilon$,
the Margulis invariant $\alpha(\eta\gamma)$ controls 
the minimum Euclidean distance $\eta\gamma$ moves any point. 
For any compact $K \subset\Eto$, 
\[
\{ \gamma\in\Gamma \mid  \eta\gamma(K) \cap  K \neq \emptyset\}
 \] 
is finite. This implies that  $\Gamma$ acts properly on $\Eto$. 
For further details, compare Drumm-Goldman~\cite{MR1060633}.

\subsection{Length spectrum rigidity}

The \emph{marked length spectrum} of a hyperbolic structure on a surface $\Sigma$ is the map which assigns to each free homotopy class $[\gamma]$ of loop, the length $\ell(\gamma)$ of the unique closed geodesic in that homotopy class. 
Regarding $ \Isom^+(\Ht) = \PSL(2,\R)$, suppose 
$\pi_1(\Sigma) \xrightarrow{~\rho_0~} \PSL(2,\R)$ is the holonomy
representation of the hyperbolic structure on $\Sigma$. 
Then $\ell(\gamma)$ relates to the {\em character\/} of $\rho_0$ by:
\[
\tr \big( \rho_0(\gamma) \big) \ = \ \pm 2 \cosh\left(\frac{\ell\big(\rho_0(\gamma)\big)}{2}\right).
\]
Hence, a hyperbolic structure on $\Sigma$ is determined by its length spectrum, simply because the holonomy representation $\rho_0$ is determined by its character. 
This is a general algebraic fact about irreducible linear representations;
see, for example Goldman~\cite{MR2497777} for a general proof.
For details on this question see Abikoff~\cite{MR590044}.
More recently Otal~\cite{MR1038361} 
and Croke~\cite{MR1036134} proved marked length spectrum rigidity for 
surfaces of {\em variable negative curvature,\/} 
where the algebraic methods are unavailable.
For length spectrum rigidity for locally symmetric spaces,
see Inkang Kim~\cite{MR1867246,MR1826662} and
Cooper-Delp~\cite{MR2653965}.

Now we discuss to what extent the marked Lorentzian length spectrum determines
the isometry type of a Margulis spacetime. 
As a consequence of Theorem~\ref{thm:OppositeSign},  either the $\alpha(g)$ are all positive or all negative.
By changing the orientation of $\Eto$, we may assume they are all positive.

Suppose $M^3$ is a Margulis spacetime whose associated (complete) hyperbolic surface $\Sigma$
has a compact convex core.
(In this case the holonomy group $\L(\Gamma)$ of $\Sigma$ is said to be 
{\em convex cocompact.\/})
As in \S\ref{sec:msLlspectrum}, 
every element of  $\L(\Gamma)\backslash\{1\}$ is hyperbolic and 
every closed curve in $M^3$ is freely homotopic to a unique closed geodesic in $M^3$.
The absolute value $\vert\alpha(\gamma)\vert$ equals the {\em Lorentzian length\/}  
of this closed geodesic in $M^3$. 
Thus the function 
\[
\pi_1(M^3) \xrightarrow{~\alpha\circ\rho~} \R
\]
represents the analogous {\em marked Lorentzian length spectrum\/} of $M^3$.

\begin{thm}\label{thm:StrongIsospectrality}
Consider two affine deformations $\rho, \rho'$ of $\Fn$
with the same convex-cocompact representation as linear part.
Suppose that $\alpha \circ \rho = \alpha\circ \rho'$.
Then $\rho$ and $\rho'$ are conjugate in $\Isom(\E)$.
\end{thm}
This was proved by Drumm-Goldman~\cite{MR1906782} for $n=2$,
to which we shall refer. 
We give below the modifications needed to prove this for general $n>2$.
Charette-Drumm~\cite{MR2106472} proved the stronger statement
without the assumption that $\rho$ and $\rho'$ have the same linear part,
only assuming that $\alpha \circ \rho = \alpha\circ \rho'$. 
See also Kim~\cite{MR2108366} and Ghosh~\cite{Gho19}.

Assume inductively the result for all free groups of rank at most $n$, 
where $n \ge 2$.
For $\Fnp = \langle x_1, x_2, ..., x_{n+1} \rangle$ 
consider the three $n$-generator subgroups
\[
\begin{array}{rcl}
S_1 & =& \langle  x_2,x_3, x_4, ...., x_{n+1} \rangle \\
S_2 & =& \langle x_1, x_3, x_4, ...., x_{n+1} \rangle \\
S_3 & =& \langle x_1, x_2, x_4, ...., x_{n+1} \rangle 
\end{array}
\]
In the following we will only be concerned with the generators $x_1, x_2, x_3$ 
which do not occur inside every such subgroup. Without loss of generality, we may choose the generators $x_1, x_2, x_3$ so that the $1$-eigenspaces of the linear parts $\L(x_1), \L(x_2),$ and $\L(x_3)$ do not have a non-trivial linear dependence. If this is not the case, we simply replace $x_1$ by $x_2 x_1 x_2^{-1}$ and the assumption will hold.

We consider two representations $\rho,\rho'$ of $\Fnp$ and their restrictions
to $S_i$ for $i=1,2,3$.
Denote the translational parts of $\rho$ and $\rho'$ by
\[
\u, \u' \in \Zz^1(\Fnp,\Rto) \]
respectively. 
Let $i=1,2$ or $3$. 
Since the Margulis invariants $\alpha, \alpha'$ agree, their restrictions to the
$n$-generator subgroup $S_i$ also agree. 
Thus the restrictions of $\u$ and $\u'$ to $S_i$ are cohomologous 
in $\Zz^1(S_i,\Rto)$.
That is, there exists $a_i\in\Rto$ so that
\begin{equation}\label{eq:CohomologousCocycles}
 \u'(\gamma) - \u(\gamma) = \delta(a_i) (\gamma) = a_i - \L(\gamma) a_i
\end{equation}
for $\gamma\in S_i$.

We show that the vector $a_2 - a_3$ lies in the fixed line $\Fix\big(\L(x_1)\big) = 
\Ker\big(\Id - \L(x_1)\big)$. 
Apply \eqref{eq:CohomologousCocycles} to $\gamma = x_1$ and $i=2,3$:
\[
a_2 - \L(x_1) a_2  = \u'(x_1) - \u(x_1) = a_3 - \L(x_1) a_3, \]
from which follows:
\[
a_2 - a_3 = \L(x_1) (a_2 -a_3) \]
as claimed. 
Similarly,
$a_3 - a_1 \in \Fix\big(\L(x_2)\big)$ and 
$a_1 - a_2 \in \Fix\big(\L(x_3)\big)$.
By our assumption above, the three lines $\Fix\big(\L(x_i)\big)$ for $i=1,2,3$ are not coplanar, 
so in particular they form a direct sum decomposition of $\R^3$.
Observing that
\[ 
(a_2 - a_3) + (a_3 - a_1)  + (a_1 - a_2)  = 0, \]
we deduce that the vectors 
$a_2 - a_3$, $a_3 - a_1$, $a_1 - a_2$
must each be zero. 
Thus the vectors $a_1 = a_2 = a_3$ are all equal, and~\eqref{eq:CohomologousCocycles} 
holds over the entire group $\F_{n+1}$.

\subsection{Further remarks on the Margulis length spectrum}\label{sec:misc} \ 
\newline Charette--Goldman~\cite{MR3732683} proved an analog of McShane's identity~\cite{MR1625712}, 
a relation on the marked length spectrum for hyperbolic punctured tori. 

If a discrete group $\Gamma$ of affine isometries with hyperbolic linear part 
acts properly on $\E$, 
then the Margulis invariants $\alpha(\gamma)$ are either all positive or all negative.
In general, infinitely many positivity conditions are needed to ensure 
properness (but see \S\ref{sec:Classification} for the two examples of $\Sigma_0$
where only finitely many conditions suffice). 
Charette~\cite{MR2206247} found a sequence of affine deformations 
$\rho_n$ of a two-generator Fuchsian group $\Gamma_0$ with the following property: 
for any given integer~$n$,
\begin{itemize}
\item $\alpha\big(\rho_n(\gamma)\big) > 0$  for all $\gamma \in \Gamma_0$ with word length less than $n$;
\item $\alpha\big(\rho_n(\gamma')\big) < 0$ for some $\gamma' \in \Gamma_0$.
\end{itemize}
Using {\em strip deformations\/}, Minsky~\cite{Lams} explicitly showed there exist free groups with convex cocompact linear part 
with the property that the Margulis invariants of all elements have one sign but which do not act properly on $\E$. 
See the discussion of Theorem~\ref{thm:strips} in Section~\ref{sec:strips}. 

The sign of an affine deformation is undefined for elliptic affine transformations.
Charette-Drumm~\cite{MR3180618}
extended Margulis's sign   to {\em parabolic\/} affine transformations. 
A parabolic element $\gamma$ of $\SOto$ fixes no spacelike vectors and 
no closed geodesic has holonomy $\gamma$.
Charette and Drumm find a subspace of null vectors fixed by $\gamma$ with a natural orientation, 
and extend the {\em sign\/} of the Margulis invariant to $\gamma$.
Lemma~\ref{lem:alpha} can be  adapted to 
parabolic transformations. 
Moreover, Theorem~\ref{thm:OppositeSign}
extends  to the case where either or both transformations are parabolic,
using this extension of Margulis's invariant.

\section {Diffusing the Margulis invariant}

In this section we describe the extension of Margulis's marked Lorentzian length spectrum
to the space of geodesic currents and state the properness criterion of 
Goldman-Labourie-Margulis which leads to a description of the deformation space of 
Margulis spacetimes associated to a given hyperbolic surface.

\subsection{Normalizing the Margulis invariant}\label{Normalizing}
The (signed) Margulis invariant and the geodesic length function enjoy the
same homogeneity \big(Lemma~\ref{lem:alpha}, (\ref{alphan})\big):
\begin{align*}
\ell(\gamma^n) = \vert n \vert \ell(\gamma) \\
\alpha(\gamma^n) = \vert n \vert \alpha(\gamma) \end{align*}
Thus the quotient 
\begin{align*} 
\Gamma_0 & \xrightarrow{~\halpha~} \R \\
\gamma &\longmapsto   \frac{\alpha(\gamma)}{\ell(\gamma)}
\end{align*}
is constant on cyclic subgroups of $\Gamma \cong \pi_1(\Sigma)$.

\newcommand{\HOneG}{\HHh^1(\Gamma_0,\Rto)} 

Cyclic hyperbolic subgroups of $\pi_1(\Sigma)$ correspond to closed geodesics
on $\Sigma$. 
Closed geodesics on $\Sigma$  correspond to periodic trajectories of the geodesic 
flow $\Phi$ on the unit tangent bundle $U\Sigma$,
and hence determine
$\Phi$-invariant probability measures on $U\Sigma$ supported on the velocity vector 
field of the closed geodesic.

Recall that a {\em geodesic current\/} on $\Sigma$ is a 
$\Phi$-invariant probability measure on $U\Sigma$.
See Bonahon~\cite{MR931208}.
The convex set $\Cc(\Sigma)$ of all geodesic currents is equipped with the weak-* 
topology. It is compact if 
$\Sigma$ has compact convex core. 
Geodesic currents corresponding to closed geodesics are dense in $\Cc(\Sigma)$.

For a fixed affine deformation $\rho$ of a Fuchsian representation $\rho_0$,
the above function $\halpha$ extends to a continuous map
\[
\Cc(\Sigma) \xrightarrow{~\halpha~} \R.\]
Moreover, if we let the $\rho=\rho_{[\vu]}$ vary over the space
$\HOneG$
of affine deformations:

\begin{thm} [Goldman-Labourie-Margulis~\cite{MR2600870}]\label{thm:GLM}
Fix a hyperbolic surface $\Sigma$ with holonomy representation $\rho_0$. 
\begin{itemize}
\item 
There exists a continuous map
\[
\HOneG
\times \Cc(\Sigma) \xrightarrow{~\Psi~} \R \]
such that for a fixed affine deformation $\rho$ corresponding to
$[\u]\in \HOneG$
and an element $\gamma\in \Gamma$ corresponding to an $\Phi$-invariant probability
measure $\mu$,
\[
\Psi([\vu], \mu)  = \halpha_{\rho} (\gamma) \]
as above.
Furthermore this function is {\em bi-affine\/} with respect to the linear
structure on $\HOneG$ and the affine structure on
$\Cc(\Sigma)$.
\item The affine deformation $\rho = \rho_{[\vu]}$ is proper if and only if 
the image 
$\Psi\Big( \{\vu\} \times \Cc(\Sigma)\Big)$ is bounded away from $0$.
\end{itemize}
\end{thm}
Since probability measures supported on periodic trajectories are dense,
properness is equivalent to bounding $\halpha(\gamma)
= \alpha(\gamma)/\ell(\gamma)$ away from zero. 
In particular,  
Minsky's \cite{Lams} construction of non-proper affine deformations 
has the property that every element has the same sign, 
but there is a sequence of elements whose normalized
Margulis invariants approach zero.

Theorem~\ref{thm:GLM} immediately implies 
the Opposite Sign Lemma (Theorem~\ref{thm:OppositeSign}) as follows.
Suppose that $\alpha(\gamma_1) < 0 < \alpha(\gamma_2)$. 
Let $\mu_i$ denote the invariant probability measure corresponding to $\gamma_i$.
Then 
\[
\halpha(\mu_1) = \frac{\alpha(\gamma_1)}{\ell(\gamma_1)} 
< 0 < 
\frac{\alpha(\gamma_2)}{\ell(\gamma_2)} = \halpha(\mu_2). 
\]
Since $\Cc(\Sigma)$ is convex, 
 a continuous path $\mu_t$ (for $1\le t\le 2$) joins $\mu_1$ to $\mu_2$.
Continuity of $\halpha$ and 
the Intermediate Value Theorem imply that $\halpha(\mu_t) = 0$ 
for some $1< t < 2$ and  $\mu_t\in\Cc(\Sigma)$.
By Theorem~\ref{thm:GLM}, the affine deformation is not proper.

Let
us briefly contrast the properness criterion of Theorem~\ref{thm:GLM} with the
properness criterion of Benoist~\cite{MR1418901} and Kobayashi~\cite{MR1424629} for
reductive homogeneous spaces. 
In the setting of reductive homogeneous spaces $G/H$, properness of the action of a discrete group $\Gamma < G$ is characterized by the behavior of the Cartan projection (singular values) of $\Gamma$, specifically that the Cartan projection of $\Gamma$ goes away from the Cartan projection of $H$. 
There is no known analogue of this simple criterion in non-reductive settings, such as Minkowski geometry $\Eto$. In Minkowski geometry, the Margulis invariant of an element is less like a Cartan projection, and more like an infinitesimal Jordan projection (eigenvalues). 
The work of Danciger-Gu\'eritaud-Kassel~\cite{MR3465975} interprets an action of a group $\Gamma$ on Minkowski space $\Eto$ as an infinitesimal action on the $3$-dimensional anti-de Sitter space, the model for constant negative curvature Lorentzian geometry in dimension three. Note that anti-de Sitter space is a reductive homogeneous space, so the Benoist-Kobayashi properness criterion applies. However, Kassel~\cite{kasPhD} and Gu\'eritaud-Kassel~\cite{MR3626591} give a different properness criterion in terms of uniform behavior of the Jordan projections. 
The new proof of Theorem~\ref{thm:GLM} given in Danciger-Gu\'eritaud-Kassel~\cite{MR3465975} interprets the uniform behavior of the Margulis invariant as an infinitesimal analogue of the Gu\'eritaud-Kassel properness criterion in anti-de Sitter geometry. 
See the discussion in \S\ref{sec:Lipschitz}.

\subsection{Classification of Margulis spacetimes}\label{sec:Classification}
Theorem~\ref{thm:linearholonomy} implies that
proper affine deformations exist, 
whenever $\Sigma = \Gamma_0\backslash\Ht$ is  noncompact.
Another consequence of Theorem~\ref{thm:GLM} is a determination of the deformation space of Margulis spacetimes
as a convex domain.
Theorem~\ref{thm:GLM} implies that the space of {\em all\/}  
proper affine deformations of $\Gamma_0$ equals the subspace of $\HOneG$ 
comprised of $[\vu]$ such that $\Phi(\big[\vu], \mu\big)$ is either always positive or always 
negative, for all $\mu\in \Cc(\Sigma)$.

The positive affine deformations, those $[\vu]$ for which $\Phi(\big[\vu], \mu\big) > 0$ 
for all $\mu \in \Cc(\Sigma)$, form an open and convex cone. 
This cone is the interior of the intersection over $\gamma \in \Gamma$ of the set of halfspaces 
defined by $\alpha_{[\vu]}(\gamma) > 0$. 
In fact, it suffices to take this intersection over $\gamma$ corresponding to
{\em simple\/} loops on $\Sigma$. 
See Goldman-Labourie-Margulis-Minsky~\cite{Lams} and Danciger-Gu\'eritaud-Kassel~\cite{MR3465975}. 
Indeed, properness is implied by positivity (or negativity) of the Margulis invariant $\Psi([\vu], \mu)$ over all 
\emph{measured laminations} $\mu$ (those currents with self-intersection zero).

The following four figures depict the deformation space for four hyperbolic surfaces $\Sigma$ representing the four different topological types for which $\chi(\Sigma) = -1$,
or equivalently  $\pi_1(\Sigma) \cong \Ft$.  
Although in these cases 
\[ \dim~\HOneG = 3,  \]
we may projectivize the set and draw the image of this cone in the $2$-dimensional projective space 
$\P\big(\HOneG\big)$.
The lines drawn in these pictures are defined by $\alpha(\gamma) = 0$,
where $\gamma$ is a {\em primitive\/} element of $\Ft$ 
(that is, an element living in a free basis of $\Ft$).
However, only in the case of the one-holed torus, 
do the primitive elements correspond to simple nonseparating loops.

\begin{figure}[h]
\centering
\subfigure{\includegraphics[scale=.45]{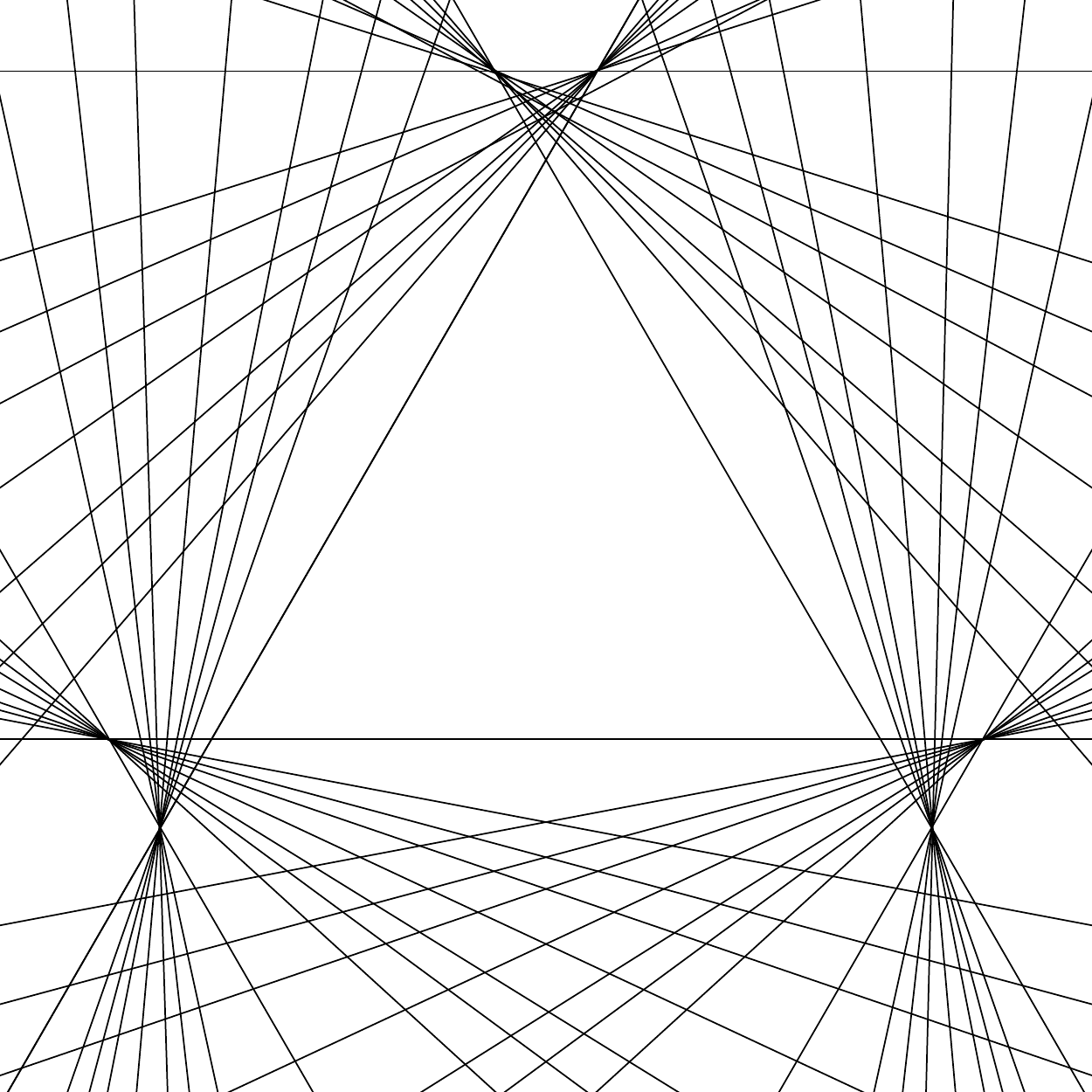}
} \qquad
\subfigure{\includegraphics[scale=.45]{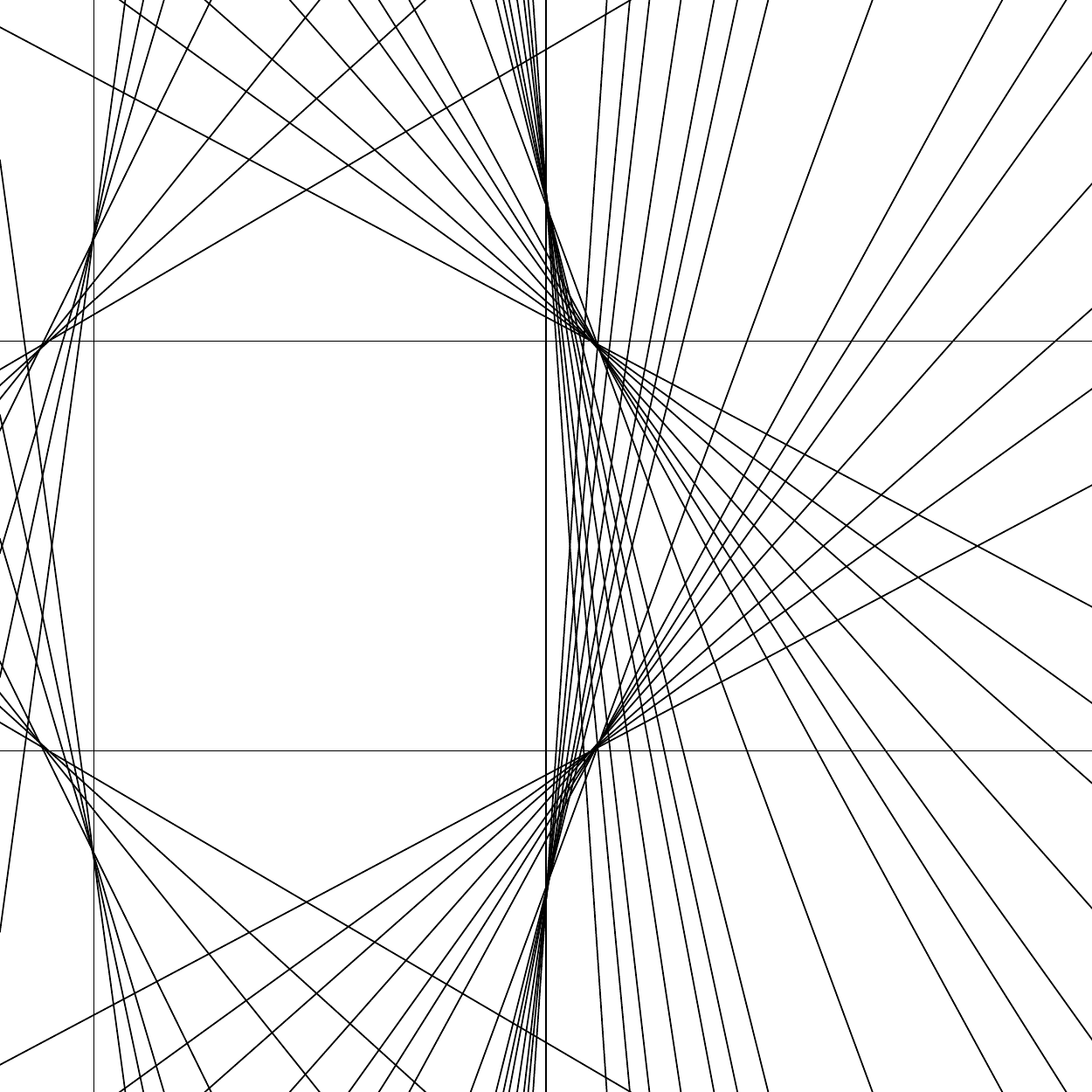}
}
\caption{\label{fig:TwoGeneratorDefSpaces1}Deformation spaces for the three-holed sphere and two-holed 
cross-surface (projective plane)}
\end{figure}
\begin{figure}[h]
\centering
\subfigure{\includegraphics[height=5.7cm,width=5.7cm]{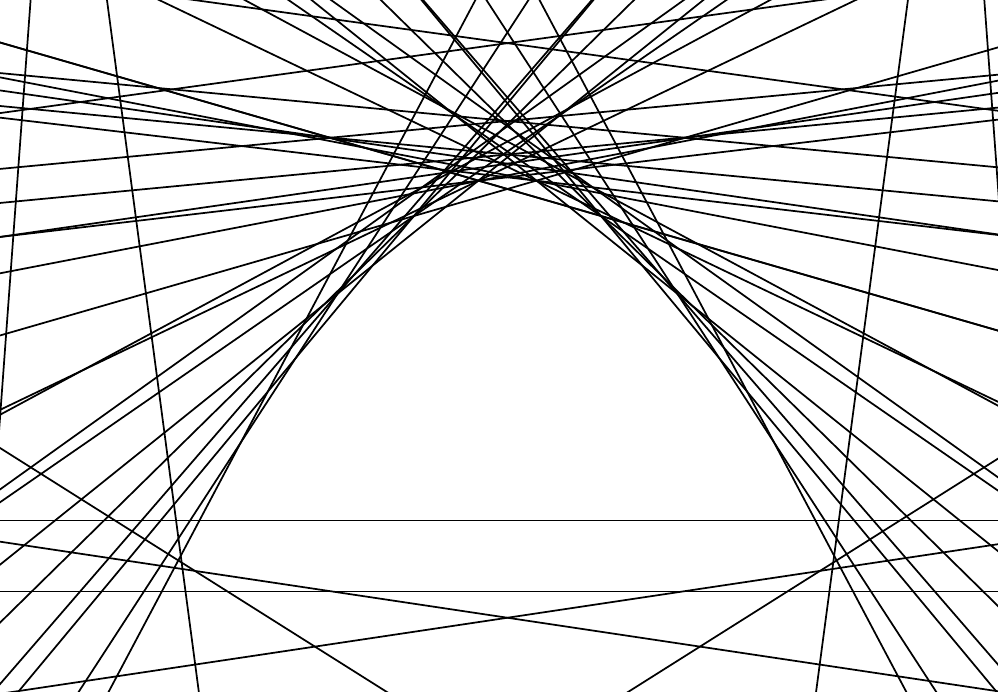}
} \qquad
\subfigure{\includegraphics[height=5.7cm,width=5.7cm]{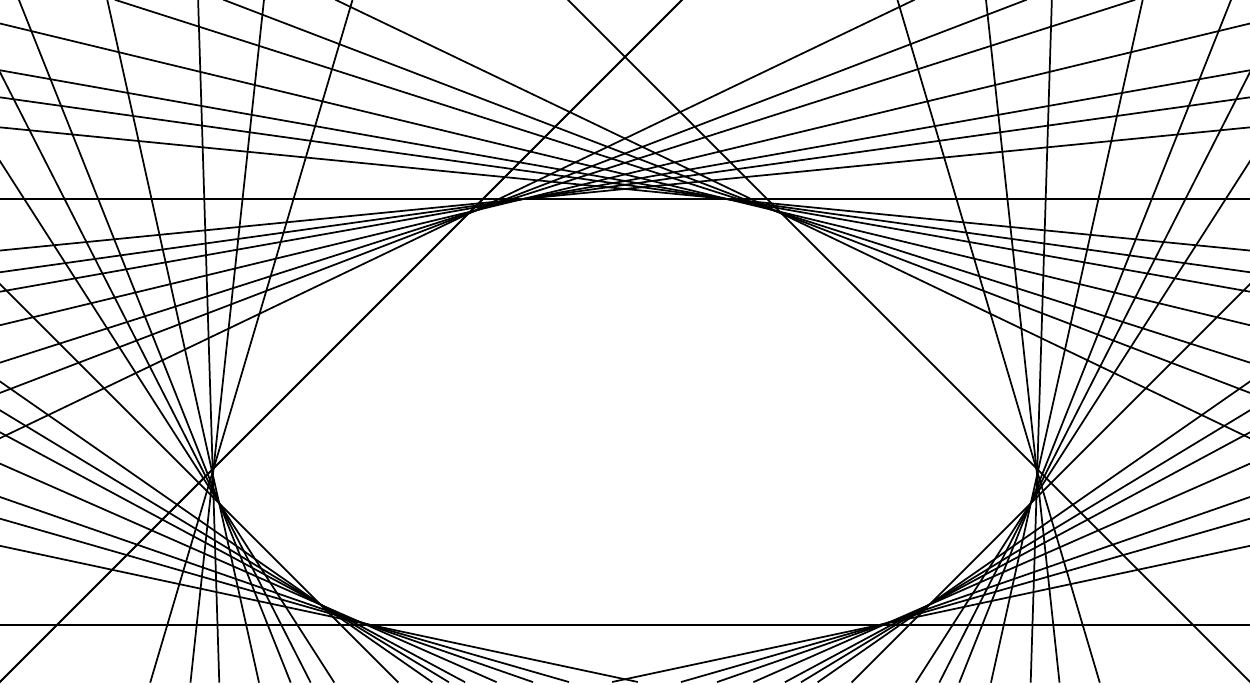}
}
\caption{\label{fig:TwoGeneratorDefSpaces2}Deformation spaces for the one-holed Klein bottle and one-holed torus}
\end{figure}

\subsection{Dynamical ideas from the proof of Theorem~\ref{thm:GLM}}

Let us sketch some ideas from Goldman-Labourie-Margulis~\cite{MR2600870}. 
The basic idea of the properness criterion in Theorem~\ref{thm:GLM} is to translate properness of the discrete group action into a question of properness of a continuous flow.
Consider the flat affine bundle $\E_\rho \to U\Sigma$, defined as the quotient $(\Eto \times U \Ht)/\Gamma$ where $\Gamma$ acts by $\rho = \rho_{[\vu]}$ in the $\Eto$ factor and by $\rho_0$ on the $U \Ht$ factor. The geodesic flow $\varphi_t$ on $U \Sigma$ lifts, via the flat connection, to a flow $\Phi_t$ on $\E_\rho$. Then the properness of the $\rho$-action of $\Gamma$ on $\Eto$ is equivalent to properness of the flow $\Phi_t$ on $\E_\rho$.
To determine properness of this flow, it suffices to consider only the recurrent part $U_{rec} \Sigma$ of the unit tangent bundle $U \Sigma$. 

Let us describe how to extend the Margulis invariant function to the space of currents.
Let $s: U \Sigma \to \E_\rho$ be a smooth section. Using the flat connection, one may measure how much $s$ changes along a path in the unit tangent bundle. The Margulis invariant $\alpha(\gamma)$ of an element $\gamma \in \Gamma$ is equal to the amount the section $s$ changes in the direction of the translation axis $\Axis(\rho(\gamma))$ after going once around the geodesic representative for $\gamma$ in $U \Sigma$. 
To generalize this, for $v \in U_{rec} \Sigma$, we may ask how much the section $s$ changes in the \emph{neutral direction} $\nu(v)$ along (some finite piece of) a trajectory $\Phi_t v$ of the geodesic flow. Here $\nu$ is the canonical section of the associated vector bundle $V_\rho = (\Rto \times U \Ht )/\Gamma$ which maps a tangent vector $v$ to the spacelike unit vector $\nu(v)$ dual to the geodesic in $\Sigma$ tangent to $v$. 
Hence, to each point $v \in U \Sigma$ we may associate the neutral variation $\langle d_{\varphi_t} s, \nu \rangle$ of $s$ in the direction of the flow, a real valued function on $U \Sigma$.  
The definition of the Margulis invariant $ \halpha(\mu)$ of a geodesic current $\mu$ on $\Sigma$ is simply the integral over $U_{rec}$ of this function agains the measure $\mu$:
\[ \halpha(\mu) := \int_{U_{rec}}  \langle d_{\varphi_t} s, \nu \rangle d\mu. \]
It can be shown that the definition does not depend on the choice of section $s$. Indeed, the functional $\mu \mapsto \halpha(\mu)$ is continuous and on the dense subset of currents $\mu_\gamma$ which are supported on a closed geodesic~$\gamma$, the definition gives $\halpha(\mu_\gamma) = \halpha(\gamma)$.

The properness criterion is proved as follows. On the one hand, if there is a current $\mu$ with $\halpha(\mu) = 0$, then the section $s$ (or rather a modified section whose variation along flow lines is only in the neutral direction) takes the support of $\mu$ to a compact subset of $\E_\rho$ which is not taken away from itself by the flow $\Phi_t$, hence the flow is not proper, and hence the action of $\Gamma$ on $\Eto$ is not proper.  Conversely, if the action of $\Gamma$ on $\Eto$ is not proper, then the flow $\Phi_t$ is not proper and it is possible to find a sequence of longer and longer flow lines which make less and less progress in the fiber with respect to the flat connection. One constructs a geodesic current with zero Margulis invariant in the limit. 

\subsection{Dynamical structure of Margulis spacetimes}

While
we do not explore the ideas here, we mention some further work on the dynamical structure of Margulis spacetimes.

In \cite{MR2901364}, 
Goldman and Labourie show that the union of closed geodesics in a Margulis spacetime
with convex cocompact linear holonomy is dense in the projection of the {\em nonwandering
set\/} for the geodesic flow. 
(Goldman and Labourie call the projections of the nonwandering orbits {\em recurrent.\/}) 
In other words, a Margulis spacetime with convex cocompact linear holonomy group has a compact ``dynamical core". 
This is completely analogous to the behavior of nonwandering orbits for the geodesic flow
for a convex cocompact hyperbolic surface.

Further, Ghosh~{\cite{MR3668058} proves an Anosov property for the geodesic flow in the dynamical core
of a Margulis spacetime.
From that he constructs a {\em pressure metric\/} on the moduli space  (\cite{MR3770278}), 
analogous to the pressure metric defined in higher Teichm\"uller theory constructed by
Bridgeman-Canary-Labourie-Sambarino~\cite{MR3833339,MR3385630}.

\section{Affine actions and deformations of geometric structures}\label{sec:Deformations}
The finer structure of the classification of Margulis spacetimes
is deeply tied to the theory of {\em infinitesimal deformations\/} of hyperbolic surfaces.
The connection comes from the low-dimensional coincidence that the standard 
representation of $\SOto$ acting on $\Rto$ is isomorphic to the 
\emph{adjoint representation} of $\SOto$ acting on its Lie algebra~$\mathfrak{so}(2,1)$. 
In general, if $G$ is a Lie group, affine actions with linear part in the adjoint 
representation of $G$ are in direct correspondence with infinitesimal deformations 
of representations into $G$, which in many cases correspond to deformations of 
geometric structures modeled on some homogeneous space of $G$. The dynamics of 
the affine action is often closely related to the geometry of the associated 
deformation of geometric structures. 
For the case of $G = \SOto$ corresponding to Margulis spacetimes, 
Danciger--Gu\'eritaud--Kassel~\cite{MR3465975} showed that 
properness of the affine deformaion  %
is equivalent to the condition that  
the associated deformation of hyperbolic structure is \emph{contracting}, 
in a sense to be made precise later in this section.
 
\subsection{Affine actions and deformation theory}\label{subsec:deform}

For the moment, let us work in the general setting that $G$ is an arbitrary finite dimensional Lie group.
Recall that the Lie algebra  $\gg$ of $G$ is the Lie subalgebra of $\Vect(G)$ consisting  of {\em right-invariant vector fields\/}  
on $G$.  
These vector fields generate   flows by left-multiplication by one-parameter subgroups of $G$. 
The action of $G$ on itself by left-multiplication induces a (left-) action on $\Vect(G)$.
Since left- and right-multiplication commute,
the left-action on $\Vect(G)$ preserves $\gg \subgroup \Vect(G)$.
The resulting action is the adjoint representation 
\[
G \xrightarrow{~\Ad~} \Aut(\gg).
\]

The linear action of $G$ on $\mathfrak g$ extends to an affine action of the semi-direct product $\Ggg$ 
where $\gg$ acts by translations:
\begin{equation}\label{eqn:Ad-affine-action}
\vv \xmapsto{~(g,\u)~}   \Ad(g) \vv + \u.
\end{equation}
The adjoint action of $G$ on $\gg$ preserves the Killing form $B (\cdot, \cdot)$. 
If $G$ is semisimple, which   will be the case in all of our applications, 
$B$ is a non-degenerate symmetric bilinear form of indefinite signature $(p,q)$. 
Hence the affine action of $\Ggg$ on $\gg$ is by isometries of a 
flat pseudo-Riemannian metric of signature $(p,q)$. This gives a map
\begin{align}\label{eq:map}
\Ggg \xrightarrow[]{\Phi_G} \Isom(\E^{p,q}),
\end{align} which maps $G$ into the stabilizer of a point, a copy of $\SO(p,q)$, 
and maps $\gg$ to the translation subgroup $\R^{p,q}$.

Affine actions of the form~\eqref{eqn:Ad-affine-action} closely relate to infinitesimal deformations 
of geometric structures and representations.
The affine group $\Ggg$ is naturally isomorphic to the total space of the tangent bundle 
$\T G$ of the Lie group~$G$, under the map which associates an element 
$g \in G$ and $\u \in \gg$ to the evaluation $\u_g \in \T_g G$ of $\u$ at $g$:
\begin{align}\label{eqn:iso}
\T G \cong \Ggg.\end{align}

Given a representation $\rho_0: \Gamma \to G$ of a discrete group $\Gamma$ in 
$G$, an \emph{infinitesimal deformation} of $\rho_0$ is a homomorphic lift 
$\rho$ to the tangent bundle~$\T G$:
$$\xymatrix{& \T G\ar[d]^{\Pi_G}\\ \Gamma\ar[r]^{\rho_0}\ar[ur]^{\rho} & G}.$$
Infinitesimal deformations arise naturally as tangent vectors to paths in 
the analytic set $\Hom(\Gamma, G)$.
Indeed, if $\rho_t \in \Hom(\Gamma, G)$ is a smooth path, then 
\begin{align}\label{eqn:derivative}
\rho(\gamma) := \frac{\dd}{\dd t}\Big|_{t=0} \rho_t(\gamma) \ \in \ \T_{\rho_0(\gamma)}G
\end{align}
defines an infinitesimal deformation. 
Using the isomorphism~\eqref{eqn:iso}, 
an infinitesimal deformation $\rho$ of a fixed representation 
$\rho_0 \in \Hom(\Gamma, G)$ is  efficiently described as a cocycle 
\[ \u\in\Zz^1(\Gamma,\gg_{\Ad\rho_0}), \] 
where $\gg_{\Ad\rho_0}$ denotes the $\Gamma$-module defined by the composition
\[
\Gamma\xrightarrow{~\rho_0~} G \xrightarrow{~\Ad~} \Aut(\gg). 
\]
(Compare Raghunathan~\cite{MR0507234}, \S VI.) 

We refer to  cocycles in $\Zz^1(\Gamma,\gg_{\Ad\rho_0})$ as \emph{deformation cocycles}. 
When $\rho$ is the derivative of a conjugation path 
\[ \rho_t(\cdot) = g_t \rho_0(\cdot) g_t^{-1},\] 
where $g_t$ is a smooth path in $G$ based at the identity, 
the associated cocycle $\u$ is the coboundary $\delta v$, 
where $v\in\gg$ extends the tangent vector
\[ \frac{\dd}{\dd t}\Big|_{t=0} g_t\in \T_e(G).\]
The set $\Bb^1(\Gamma, \gg_{\Ad\rho_0})$ of such coboundaries makes up the 
{\em infinitesimal conjugations}, or trivial infinitesimal deformations. The 
cohomology group $\mathsf{H}^1(\Gamma,  \gg_{\Ad\rho_0})$ describes the 
equivalence classes of infinitesimal deformations up to infinitesimal 
conjugation. For further details, see also Sikora~\cite{MR2931326} or Labourie~\cite{MR3155540}. 
  
\subsection{Margulis invariants and length functions}\label{sec:MargulisLength}

For the remainder of this section (except Section~\ref{sec:racg}) 
we consider the specific case when 
\[ G := \Isom(\Ht) \cong \PGLtwoR \cong \SOto,  \]
\[ \gg :=  \sltwoR \cong \soto,\]
and $\Gamma_0 
\xhookrightarrow{~\rho_0~} G$
is the inclusion of a finitely generated, torsion-free, discrete subgroup 
corresponding to the hyperbolic surface $\Sigma = \Gamma_0 \backslash \Ht$. 

Since the adjoint action of $G$ on $\gg$ is isomorphic to the standard 
representation of $\SOto$ on $\Rto$, 
the action of $\Isom^+(\Eto)$ on $\Eto$ identifies with the 
affine action of $\Ggg$ on the Lie algebra $\gg$: 
in other words, the map 
\[
\Ggg \xrightarrow[]{\Phi_G} \Isom^+(\Eto)
\]
from~\eqref{eq:map} is an isomorphism.
In particular, a cocycle in 
\[ \Zz^1(\Gamma_0,\Rto_{\rho_0}) \cong \Zz^1(\Gamma_0,\soto_{\Ad}) \]
corresponds both to an affine deformation $\Gamma_\u$ of $\Gamma_0$,
and to an infinitesimal deformations of the representation 
$\Gamma_0 \xhookrightarrow{~\rho_0~} G$.
By the Ehresmann-Weil-Thurston principle (see Goldman~\cite{History}), 
infinitesimal deformations of $\rho_0$ correspond to infinitesimal 
deformations of the hyperbolic structure on $\Sigma$.
Goldman-Margulis~\cite{MR1796129}  observed the first key entry in the dictionary 
between the dynamics of the affine action and the geometry of the 
associated infinitesimal deformation. 

Recall the geodesic displacement function $G \xrightarrow[]{~\ell~} \R_{\geq 0}$ 
defined in \eqref{eq:GeodesicDisplacementFunction}. Its restriction to the 
hyperbolic elements of $G$ (an open subset) is a smooth function, whose 
differential we denote by $\T G \xrightarrow[]{~\dd \ell~} \R$.

\begin{lemma}[\cite{MR1796129}]\label{lem:length-formula}
Let \[ (g,\u) \in \Ggg = \T G\] be an infinitesimal deformation 
of the hyperbolic element $g \in G$, and let $\Phi_G(g,\u)$ be the 
corresponding orientation preserving affine isometry of $\Eto$. 
Then the Margulis invariant (see Section~\ref{sec:MargInv}) 
equals %
the derivative of the length. That is,
\[
\alpha\big(\Phi_G(g,\u)\big) = \dd \ell (g,u). \]
\end{lemma}
Let $\Gamma_0 < G$ be a fixed convex cocompact subgroup. Recall the extension 
of the Margulis invariant to the space of currents from 
Theorem~\ref{thm:GLM}: 
\[ 
\HOneG
\times \Cc(\Sigma) \xrightarrow[]{~\Psi~} \R. 
\]
Lemma~\ref{lem:length-formula} implies that this function is exactly the 
differential of the length function for geodesic currents. More specifically, 
to each geodesic current $\mu$ on $\Sigma = \Gamma_0 \backslash \Ht$ is 
associated a length function $\ell_\mu$ on the space $\Hom_{cc}(\Gamma_0, G)$ 
of convex cocompact representations.  The map $\mu \mapsto \ell_\mu$ 
taking currents to continuous $G$-invariant functions on 
$\Hom_{cc}(\Gamma_0, G)$ is continuous. 
Density of geodesic currents corresponding to closed  geodesics implies that 
$\ell_\mu(\cdot)$ may be approximated in terms of the usual 
length functions $\ell_\gamma(\cdot)$ for $\gamma \in \Gamma_0$, 
defined by: 
\[\ell_\gamma(\rho) := \ell\big(\rho(\gamma)\big). \]
Lemma~\ref{lem:length-formula} implies that under the identification 
\[
\HOneG \cong \mathsf{H}^1(\Gamma_0, \gg_{\Ad}),\] 
the diffused Margulis invariant function $\Psi$ is precisely the map
\[ 
\mathsf{H}^1(\Gamma_0, \gg_{\Ad}) \times \Cc(\Sigma) \xrightarrow[]{~\dd \ell~} \R
\]
taking a cohomology class $[\u]$ of infinitesimal deformations and a current 
$\mu \in \Cc(\Sigma)$ to the derivative $\dd\ell_\mu(\u)$ of the 
length of $\mu$ in the $\u$ direction. Therefore,
the Goldman--Labourie--Margulis properness criterion (Theorem~\ref{thm:GLM}) may be restated:

\begin{prop}\label{prop:PropernessLengths}
Let $\Gamma_0 < G$ be a convex cocompact subgroup 
and $\u\in\Zz^1(\Gamma_0,\gg_{\Ad \rho_0})$ a cocycle defining an infinitesimal 
deformation of the inclusion $\Gamma_0 \xrightarrow{~\rho_0} G$.   
Then the corresponding affine action $\Phi_G(\rho_0, \u)$ is properly discontinuous if and only if
$\dd\ell_\mu(\u) \neq 0$ for all geodesic currents $\mu\in\Cc(\Sigma)$.
\end{prop}
By exchanging $\u$ with $-\u$ (which gives an affine equivalent action),
we may assume that $\dd\ell_\nu(\u) \leq 0$ for some current $\nu \in \Cc(\Sigma)$. 
Hence, since the space of currents is connected, 
the condition that $\dd\ell_\mu(\u) \neq 0$ for all $\mu\in\Cc(\Sigma)$ is 
equivalent to the condition that $\dd\ell_\mu(\u) < 0$ for all $\mu\in\Cc(\Sigma)$.
Equivalently, 
\begin{align}\label{eqn:shrinking}
\sup_{\gamma \in \Gamma \setminus \{e\}} \frac{\dd\ell_\gamma(\u)}{\ell(\gamma)} &< 0.
\end{align} 
In other words, 
all closed geodesics on $\Sigma$ become uniformly shorter under the infinitesimal deformation.

Mess's Theorem~\ref{thm:Mess}, 
which states that any affine deformation of a cocompact surface group 
$\Gamma_0 < G$ fails to be proper, 
follows easily from Proposition~\ref{prop:PropernessLengths}. 
Indeed, suppose that $\Sigma_0 = \Gamma_0\backslash \Ht$ is a closed surface. 
There exists a geodesic current  $\mu_{\Gamma_0}$,  
the {\em Liouville current\/} associated to $\Sigma_0$, 
whose length in any hyperbolic structure $\Sigma$ 
is minimized for $\Sigma = \Sigma_0$.
Hence $\dd \ell_{\mu_{\Gamma_0}}(\u) = 0$ 
for any infinitesimal deformation $\u$. 

In fact, a slightly stronger statement is true:  
any nontrivial infinitesimal deformation $\u$ of a closed hyperbolic surface must 
increase the lengths of some closed geodesics while decreasing the lengths of others.
A hint as to why that should be true is that the area of a closed hyperbolic 
surface of genus $g \geq 2$ is constant, equal to $4 \pi (g-1)$ by the Gauss--Bonnet formula. 
Thus, if a deformation contracts in some directions, then it
should stretch/lengthen in others directions.

The same basic idea underpins  Thurston's theory of the 
Lipschitz metric on Teichm\"uller space~\cite{ThurstonStretch}.
This metric measures distance between two hyperbolic structures on a 
closed surface according to the minimum Lipschitz constant of 
Lipschitz maps between the two structures.
Gu\'eritaud-Kassel~\cite{kasPhD, MR3626591}  
extended  Thurston's theory to finite-type hyperbolic surfaces, 
as well as  higher-dimensional hyperbolic manifolds.
An application is a properness criterion 
in the setting of  $G \times G$ acting on $G$ by right and left multiplication, 
where $G = \SOoto$ is the identity component of the isometry group of 
hyperbolic $2$-space~$\Ht$.

This theory is the starting point for 
the work of Danciger--Gu\'eritaud--Kassel~\cite{MR3465975} on Margulis spacetimes,
so we  digress briefly  
to explain it.

\subsection{Contracting deformations and proper actions on Lie groups}\label{sec:Lipschitz}
Consider the identity component $G = \SOoto$ of $\Isom(\Ht)$ and 
the action of $G \times G$ on $G$ by 
{\em  right/left multiplication:\/}
\begin{align}\label{eq:right-and-left}
(g_0, g_1) \cdot h := g_1h g_0^{-1}.
\end{align}
Thanks to the exceptional isomorphism $\mathfrak{so}(2,2) \simeq \mathfrak{so}(2,1) \oplus \mathfrak{so}(2,1)$,
the action of $G \times G$ on $G$ then models three-dimensional 
Lorentzian geometry of constant negative curvature, 
also known as {\em anti-de Sitter (AdS) geometry.\/} 

Consider a discrete (geometrically finite, non-elementary) embedding 
$ \Gamma \xrightarrow{~\rho_0~}G $
defining a proper action on $\Ht$ whose quotient 
\[\Sigma_0 = \rho_0(\Gamma) \backslash \Ht\] is a complete hyperbolic surface.
Consider a second discrete embedding 
$\rho_1: \Gamma \hookrightarrow G$. 
Then via~\eqref{eq:right-and-left}, 
the pair $(\rho_0, \rho_1)$ defines an  action of $\Gamma$ on $G$. 
Such an action is not necessarily properly discontinuous: 
for instance,  
if    
\[\rho_1 = \o{Inn}(h)\circ \rho_0\]
for $h\in G$, 
then $(\rho_0,\rho_1)(\gamma)$ fixes $h$ for all $\gamma\in \Gamma$.

Here is the properness criterion for such $G \times G$ actions on $G$,
due to Gu\'eritaud--Kassel~\cite{kasPhD, MR3626591}.
Say that $\rho_1$ is a {\em contracting Lipschitz deformation\/} of $\rho_0$ if and only if
there exists a Lipschitz map 
\[  \Ht \xrightarrow{~f~} \Ht \] 
with Lipschitz  constant $\mathrm{Lip}(f) < 1$ which is $(\rho_0, \rho_1)$-equivariant, 
that is, 
\begin{align}\label{eqn:equivariance}
f \circ \rho_0(\gamma) &= \rho_1(\gamma) \circ f
\end{align} for all $\gamma \in \Gamma$. 

\begin{thm}[Gu\'eritaud--Kassel~\cite{kasPhD, MR3626591}]
Up to switching the roles  of $\rho_0$ and $\rho_1$, 
the $(\rho_0, \rho_1)$-action of $\Gamma$ on $G$ is proper 
if and only if $\rho_1$ is a contracting Lipschitz deformation of $\rho_0$.
\end{thm}
\noindent
Note that if $\rho_1$ is also 
injective and discrete with quotient $\Sigma_1 = \rho_1(\Gamma)\backslash \Ht$, 
then $f$ corresponds to 
a Lipschitz deformation of hyperbolic surfaces $\Sigma_0 \to \Sigma_1$.
To see why the existence of such a map $f$ suffices for properness, 
observe that the $(\rho_0, \rho_1)$ action on $G$ projects equivariantly 
down to the $\rho_0$ action on $\Ht$ via the \emph{fixed point map} 
\[g \mapsto \Fix(g^{-1}\circ f),\] 
which is well defined by the contraction property.
Proper discontinuity of the action on the base $\Ht$ then implies proper 
discontinuity on~$G$.

Further generalizing Thurston's theory, Gu\'eritaud--Kassel also studied the 
relationship between the optimal Lipschitz constant over all 
$(\rho_0, \rho_1)$-equivariant maps and the factor by which translation 
lengths are stretched in $\rho_1$ compared with $\rho_0$. In particular, 
when $\rho_0$ is convex cocompact,  
a contracting Lipschitz $(\rho_0, \rho_1)$-equivariant map 
exists
if and only if
\begin{align}\label{eqn:macro-length}
 \sup_{\gamma\in\Gamma\setminus\{e\}} 
\frac{\ell\big (\rho_1(\gamma)\big)} {\ell\big (\rho_0(\gamma)\big)} &< 1.
\end{align}
Observe the similarity with the properness criterion 
Proposition~\ref{prop:PropernessLengths}. 
Indeed, condition~\eqref{eqn:shrinking}
may be viewed as the infinitesimal version of~\eqref{eqn:macro-length}.
The similarity exemplifies a more fundamental principle at work: 
The affine action of $\Ggg$ on $\gg$
is the \emph{infinitesimal analogue} of the action by right-and-left multiplication of 
$G \times G$ on $G$. 
This guiding principle led Danciger--Gu\'eritaud--Kassel to develop an 
infinitesimal analogue of the theory of Lipschitz contraction and proper actions. 
The next sections will dive into that theory and its consequences for the 
structure and classification of Margulis spacetimes.

One consequence, in line with the guiding principle above, is the following 
geometric transition statement: \emph{every Margulis spacetime is the rescaled 
limit of a family of collapsing AdS spacetimes.} For the sake of brevity, we do 
not give details here, 
see~\cite[Thm 1.4]{MR3465975}. 
For a different survey which develops more 
thoroughly the parallel between Margulis spacetimes and complete AdS $3$-manifolds, 
see Gu\'eritaud~\cite{MR3379833}.

Let us also mention that this geometric transition result has recently been generalized to the setting of the affine group $\SO(2n+2,2n+1) \ltimes \R^{4n+3}$ seen as an infitesimal analogue of the reductive group $\SO(2n+2, 2n+2)$ (Danciger, Gu\'eritaud and Kassel's result is the case $n=1$). Of course for arbitrary~$n$ the exceptional isomorphism then no longer applies, so the general case no longer fits into the framework of $G \ltimes_{\Ad} \gg$ and $G \times G$. This generalization has been done independently in Danciger and Zhang~\cite{dz} on the one hand, and Ghosh~\cite{Gho18} on the other hand.

\subsection{Deformation vector fields and infinitesimal contraction}\label{sec:vector-fields}
Let us return now to the setting of Margulis spacetimes. 
Let $G = \SOto$ and $\gg = \mathfrak{so}(2,1)$ be  its Lie algebra.
Fix a discrete faithful representation $\rho_0: \Gamma \hookrightarrow G$ of the 
finitely generated, torsion free, group $\Gamma$ determining a hyperbolic surface 
$\Sigma_0 = \big(\rho_0(\Gamma)\big)\backslash \Ht$. 
Let $\u\in\Zz^1(\Gamma,\gg_{\Ad\rho_0})$ be a cocycle tangent to a smooth deformation
path $\rho_t\in \Hom(\Gamma,G)$ based at $\rho_0$, as in~\eqref{eqn:derivative}. 
Assume further that $\rho_0$ is convex cocompact, so that  for small $t \geq 0$, 
the representations $\rho_t$ are also discrete and faithful and determine a 
family of hyperbolic surfaces  $\Sigma_t :=  \big(\rho_t(\Gamma)\big)\backslash \Ht$. 
These hyperbolic surfaces may be organized into a smoothly varying family of 
hyperbolic structures on a single surface $\Sigma := \Sigma_0$ by finding a 
smoothly varying family of \emph{developing maps} $f_t :  \Ht \to \Ht$ which satisfy
\begin{itemize}
\item $f_0 = \mathrm{id}$,
\item $f_t$ is a homeomorphism for all $t$, and
\item $f_t$ is $(\rho_0, \rho_t)$-equivariant:
\begin{align}\label{equivariance}
f_t \circ \rho_0(\gamma) &= \rho_t(\gamma) \circ f_t
\end{align}
\end{itemize}
These conditions ensure that for each $t$, $f_t$ descends to a homeomorphism 
$\Sigma \to \Sigma_t$ that becomes ``close to the identity" as $t \to 0$.
Consider the tangent vector field $X \in \Vect(\Ht)$ to the deformation 
of developing maps, defined by 
\begin{align*}
X(p) := \left. \frac{\dd}{\dd t} \right|_{t = 0} f_t(p).
\end{align*}
This vector field satisfies an equivariance condition coming from taking 
the derivative of Condition~\eqref{equivariance}.

Before stating the equivariance condition, observe that the group $G$ acts 
on the space $\Vect(\Ht)$ of all vector fields on~$\Ht$. Indeed, $\Ht = G/K$ 
is the space of right cosets of the maximal compact subgroup $K < G$, and 
so $\Vect(\Ht)$ identifies with the subspace of right-$K$-invariant vector 
fields in $\Vect(G)$. 
The left action of $G$ on $\Vect(G)$ determines an 
action of $G$ on $\Vect(\Ht)$. 
Thinking of the Lie algebra~$\gg$ as the
space of  %% 
right-$G$-invariant vector fields on $G$ as in Section~\ref{subsec:deform}, 
we have a natural embedding $\gg \hookrightarrow \Vect(\Ht)$. The image of 
$\gg$ in $\Vect(\Ht)$ is precisely the space of \emph{Killing vector fields}, 
that is  those vector fields whose flow preserves the hyperbolic metric. 
We will denote the image of $\u \in \gg$ in $\Vect(\Ht)$ again by $\u$. 

The group $\Gamma$ acts on $\Vect(\Ht)$ via the representation 
$\rho_0: \Gamma \hookrightarrow G$. Differentiating 
Condition~\eqref{equivariance} yields that
\begin{align}\label{eqn:automorphic}
X - \rho_0(\gamma)\cdot X &= \u(\gamma)
\end{align}
holds for all $\gamma \in \Gamma$; in other words, while $X$ is not 
invariant under the $\rho_0$-action of $\Gamma$, it differs from 
any translate by a Killing vector field determined by the deformation 
cocycle~$\u$. A vector field satisfying this condition is called 
\emph{$\u$-equivariant} or also \emph{automorphic}. We observe the following:

\begin{prop}\label{prop:AffineDeformationByX} Let $X \in \Vect(\Ht)$ be a 
$\u$-equivariant vector field.
Then the coset
$X-\gg$ is an affine subspace of $\Vect(\Ht)$ invariant under the 
$\rho_0$-action of $\Gamma$.
Furthermore, the action of $\Gamma$ on $X - \gg$, 
\begin{align}\label{eqn:action}
X - \xi \ \xmapsto{~\gamma~} \  X - \u(\gamma) - \Ad(\rho_0(\gamma))\xi,
\end{align}
identifies %
with the affine action $\Phi_G(\rho_0, \u)$ of $\Gamma$ on $\mathsf{E}^{2,1}$.
\end{prop}
\noindent
Each element $X - \xi \in X - \gg$ satisfies the 
equivariance property~\eqref{eqn:automorphic}, but for a different 
cocycle, namely the cocycle 
\[ \gamma \mapsto \u(\gamma) + \Ad(\rho_0(\gamma))\xi - \xi, \]
which is cohomologous to~$\u$. 
The affine space $X - \gg$ 
bijectively corresponds to 
the cohomology class $ [\u] \subset \Zz^1(\Gamma,\gg_{\Ad\rho_0})$.
Note that we insist on writing $X - \gg$ rather than $X + \gg$ so that the 
action, as written in~\eqref{eqn:action}, matches that of  $\Phi_G(\rho_0, \u)$ 
in~\eqref{eq:map}.

Properness of the affine action of $\Gamma$ on $X - \gg$ may be expressed in 
terms of an infinitesimal version of the Lipschitz contraction condition 
of Section~\ref{sec:Lipschitz}. 
Suppose the maps 
\[ \Ht \xrightarrow{~f_t~} \Ht \]
above are $K_t$-Lipschitz with Lipschitz constant 
\[ K_t = 1 + k t + O(t^2) \] 
converging smoothly to~$1$ as $t \to 0$.  
In this case, the deformation vector field $X$ satisfies an infinitesimal version of Lipschitz,
which Danciger--Gu\'eritaud--Kassel~\cite{MR3465975} call {\em $k$-lipschitz,\/}  
with lower-case ``l'': for all $x \neq y$ in $\Ht$,
\begin{equation}\label{eq:lipschitz2}
d'_X(x,y) \le k d(x,y), \end{equation}
where   $k \in \R$ is a constant and
\begin{equation}
d'_X(x,y) := \left.\frac{\dd}{\dd t}\right|_{t=0} 
d\left(f_t(x),f_t(y)\right)
\end{equation} 
is the rate at which the vector field $X$ pushes the points $x$ and $y$ away from each other.
Note that this depends only on $X$, not on the particular path of maps $f_t$. 

For any Killing vector field $\xi\in\gg$, 
the family $\exp(-t\xi) \circ {f_t}$ is also $K_t$-lipschitz.
The corresponding deformation vector field $X - \xi$ is then also $k$-lipschitz 
for the same constant~$k$. 
Thus, the entire affine space $X - \gg$ consists of $k$-lipschitz vector fields.
Properness of the action on $X - \gg$ occurs in the case that these vector 
fields are \emph{contracting}, i.e. $k < 0$.

\begin{prop}[Danciger--Gu\'eritaud--Kassel~\cite{MR3465975}]\label{prop:proper}
Suppose $X$ is $k$-lipschitz for some $k < 0$. 
Then  
the affine space $X-\gg$ admits % 
a $\Gamma$-equivariant fibration 
\[
 X - \gg \xrightarrow{~\varpi~}  \Ht. \]
In particular, $\Gamma$ acts properly on the affine space $X- \gg$ with 
quotient a complete affine three-manifold $M$. 
The 
quotient map
\[
M :=  \Gamma \backslash (X - \gg) \longrightarrow\Gamma_0 \backslash \Ht = \Sigma\]
is an %
affine line bundle over 
$\Sigma$ %
with total space $M$.  %
\end{prop}
\noindent
The proof of Proposition~\ref{prop:proper} is straightforward, following the 
same ``contracting fixed point" idea from Section~\ref{sec:Lipschitz}% .
:
\begin{proof}
Fix a point $p \in \Ht$. 
For any Killing vector field $\xi\in\gg$, 
the vector field $X-\xi$ is also $k$-lipschitz. 
For 
sufficiently large $R > 0$ 
(depending on $\Vert (X-\xi)(p) \Vert$ and $k$), 
the vector field 
$X-\xi$ points inward along 
$\partial B_R(p)$. %
Thus, $X - \xi$ has a zero $z$ inside $B_R(p)$, 
by the well-known vector field analogue of the Brouwer fixed point theorem.

Furthermore $z$ is the {\em unique\/} zero of $X-\xi$:
by the %
contraction property~\eqref{eq:lipschitz2}, $X- \xi$ pushes any point $x \neq z$  closer to $z$. 
Hence, 
define \[ \varpi(X-\xi) := z\]  to be this unique zero. 
$\varpi$ is a continuous map 
intertwining the affine action of $\Gamma$ on $X - \gg$ with the $\rho_0$-action 
of $\Gamma$ on $\Ht$. 
Since the $\rho_0$-action on the base $\Ht$ is properly discontinuous, 
the action 
on $X - \gg$ is also properly discontinuous. 

In fact, $\varpi$ is a fibration. 
The fiber $\varpi\inv(p)$ over $p\in\Ht$ consists of all vector fields $X-\xi$ vanishing  at $p$. 
In particular, 
every tangent vector $\vv\in \T_p\Ht$  is the value of 
a Killing vector field $\xi_\vv$ at $p$.
Now take $\vv = X(p)$.
The vector field $ X - \xi_\vv$ vanishes at $p$,
so $\varpi(X-\xi_{\vv}) = p$. 
Thus $\varpi\inv(p)\neq\emptyset$.
If $\xi_1,\xi_2\in\gg$ and 
\[ \varpi(X - \xi_1) = \varpi(X - \xi_2) = p, \]
then $\xi_1 -\xi_2$ is a Killing vector field which vanishes at $p$. 
Therefore, two elements of $\varpi\inv(p)$ differ by an element of the infinitesimal 
stabilizer of the point $p$, 
a copy of $\mathfrak{so}(2)$ inside of $\gg$, 
which is an affine line of negative (timelike) signature for the Killing form. 
\end{proof}

Danciger--Gu\'eritaud--Kassel show that \emph{all} Margulis spacetimes arise 
from contracting infinitesimal deformations. 
The following theorem was proved in the case that $\rho_0$ is convex 
cocompact in~\cite{MR3465975} and in the general case in~\cite{dgknew}. 

\begin{thm}[Danciger--Gu\'eritaud--Kassel]\label{thm:dgk}
Consider a discrete 
embedding $\Gamma \xrightarrow{~\rho_0~} G$ of a finitely generated, 
non-elementary group $\Gamma$, 
and a deformation cocycle $\Gamma \xrightarrow{~\u~} \gg$.  
Suppose that $\dd \ell_\gamma(\u) \leq 0$ for at least one $\gamma \in \Gamma$. 
Then the affine action $\Phi_G(\rho_0, \u)$ of $\Gamma$ on $\Eto$ is 
properly discontinuous if and only if there exists a $(\rho,\u)$-equivariant 
vector field which is $k$-lipschitz for some $k < 0$. 
In particular, 
any proper affine action of a non-abelian free group $\Gamma$ on $\mathbb R^3$ 
is conjugate to one as in Proposition~\ref{prop:proper}.
\end{thm}

Note that the assumption that $\dd \ell_\gamma(\u) \leq 0$ for a least one 
$\gamma \in \Gamma$ is satisfied by either $\u$ or $-\u$, and that the affine 
actions $\Phi_G(\rho_0, \u)$ and $\Phi_G(\rho_0, -\u)$ are conjugate by an 
orientation reversing affine transformation.
The proof of Theorem~\ref{thm:dgk} follows the same strategy as the work of 
Gu\'eritaud-Kassel~\cite{MR3626591} discussed in Section~\ref{sec:Lipschitz}. 
The key point is that if the infimum $k_{\min}$ of lipschitz constants for 
$(\rho_0, \u)$-equivariant vector fields is non-negative, 
then any $(\rho_0, \u)$-equivariant vector field realizing $k_{\min}$ must infinitesimally stretch (the lift of) a geodesic lamination in the convex core of $\Sigma$ at a rate precisely equal to $k_{\min}$.

However, contrary to the setting of Lipschitz maps, 
the Arzel\`a-Ascoli compactness theorem does \emph{not} hold for lipschitz vector fields. 
Indeed, the limit of a (bounded) sequence of $k$-lipschitz vector fields is not necessarily a vector field, 
but instead a convex set valued section of 
the tangent bundle, called a \emph{convex field}. Much technical care is 
needed in adapting the arguments of~\cite{MR3626591} to this setting.

\subsection{Tameness of Margulis spacetimes}

The topology of a Margulis spacetime may be read off from Theorem~\ref{thm:dgk} 
and Proposition~\ref{prop:proper}.
By Theorem~\ref{thm:dgk}, every proper affine action of a free group $\Gamma$ 
on $\Eto$ comes from a contracting infinitesimal deformation of a 
non-compact hyperbolic surface $\Sigma$ as in Proposition~\ref{prop:proper}. 
The quotient Margulis spacetime $M = \Gamma \backslash \Eto$ is an affine 
line bundle over the surface $\Sigma$. 
This implies the topological tameness of $M$ (Theorem~\ref{thm:TopologicalTameness}). 
See the discussion in Section~\ref{sec:tameness}.

\subsection{The moduli space of Margulis spacetimes: Strip deformations}\label{sec:strips}
Fix a discrete embedding $\rho_0$ of a free group $\Gamma$ into $G = \SOto$. 
It follows from Proposition~\ref{prop:PropernessLengths} or Theorem~\ref{thm:dgk} 
that the set of cohomology classes $[\u] \in \mathsf{H}^1(\Gamma_0, \gg_{\Ad})$ 
of infinitesimal deformations of $\rho_0$ for which the affine action 
$\Phi_G(\rho_0, \u)$ is proper is an open cone in $\mathsf{H}^1(\Gamma_0, \gg_{\Ad})$, sometimes called the \emph{admissible cone} 
or \emph{cone of proper deformations}. The admissible cone is the disjoint union of a properly convex cone and its negative.
One convex component contains the infinitesimal deformations which uniformly 
contract the geometry of the surface $\Sigma_0 = \rho_0(\Gamma)\backslash \Ht$ 
in the sense of~\eqref{eqn:shrinking} and Theorem~\ref{thm:dgk}. 
The other component
contains the infinitesimal deformations which uniformly lengthen. 
The projectivization of the admissible cone will be denoted $\adm(\rho_0)$; 
it is the moduli space of Margulis spacetimes associated to a fixed 
hyperbolic surface~$\Sigma_0$, considered up to affine equivalence. 

In~\cite{MR3480555, dgknew}, 
Danciger-Gu\'eritaud-Kassel give a combinatorial parameterization of $\adm(\rho_0)$ 
in terms of the \emph{arc complex} of $\Sigma_0$, 
in a similar spirit to Penner's cell decomposition of the decorated 
Teichm\"uller space of a punctured surface~\cite{penner}.
The parameterization realizes each contracting deformation 
$[\u] \in \mathsf{H}^1(\Gamma_0, \gg_{\Ad})$ as an 
\emph{infinitesimal strip deformation}. 
The following construction goes back to Thurston~\cite{ThurstonStretch} 
(see also Papadopolous-Th\'eret~\cite{MR2587462}). 
Let us assume the hyperbolic surface $\Sigma_0$ has no cusp, so that all 
infinite ends are funnels, that is
$\rho_0$ is convex cocompact (this assumption was present in~\cite{MR3480555}, 
but removed in~\cite{dgknew}). Starting with the hyperbolic surface $\Sigma_0$: 
\begin{itemize}
\item Choose a collection of disjoint non-isotopic properly embedded 
geodesic arcs $\alpha_1, \ldots, \alpha_r \subset \Sigma_0$.
\item For each $1 \leq i \leq r$ choose an arc $\alpha_i'$ disjoint from, 
but very close to $\alpha_i$ (in particular isotopic to $\alpha_i$) so 
that $\alpha_i$ and $\alpha_i'$ bound a \emph{strip} in $\Sigma_0$, that is a 
region isometric to the region between two ultraparallel geodesics in the 
hyperbolic plane $\Ht$. Let $p_i$ and $p_i'$ be the points on 
$\alpha_i$ and $\alpha_i'$ respectively with minimal distance. 
The geodesic segment $[p_i, p_i']$ is called the \emph{waist} of the strip.
\end{itemize}
Then, for each $i$, delete the strip bounded by $\alpha_i$ and $\alpha_i'$ 
(we assume the strips are disjoint), and glue $\alpha_i$ to $\alpha_i'$ by 
the isometry that identifies $p_i$ to $p_i'$. The result is a hyperbolic 
surface $\Sigma_1 = \rho_1(\Gamma) \backslash \Ht$ equipped with a natural 
$1$-Lipschitz map 
$ \Sigma_0 \xrightarrow{~f~} \Sigma_1$ 
(which collapses the strips). 
The holonomy representation 
$\Gamma \xrightarrow{~\rho_1~} G$ 
(defined here up to conjugation) 
is a new representation of~$\Gamma$, which we call a \emph{strip deformation} of 
$\rho_0$. See Figure~\ref{fig:strip}. One can show that, even though $f$ is only 
$1$-Lipschitz, in fact if the arcs $\alpha_1, \ldots, \alpha_r$ cut $\Sigma_0$ 
into disks, then the Lipschitz constant may be improved to $< 1$ by deforming $f$; 
in particular, \eqref{eqn:macro-length} holds. 

\begin{figure}[h]
\centering
\vspace{0.2cm}
\includegraphics[width=10.0cm]{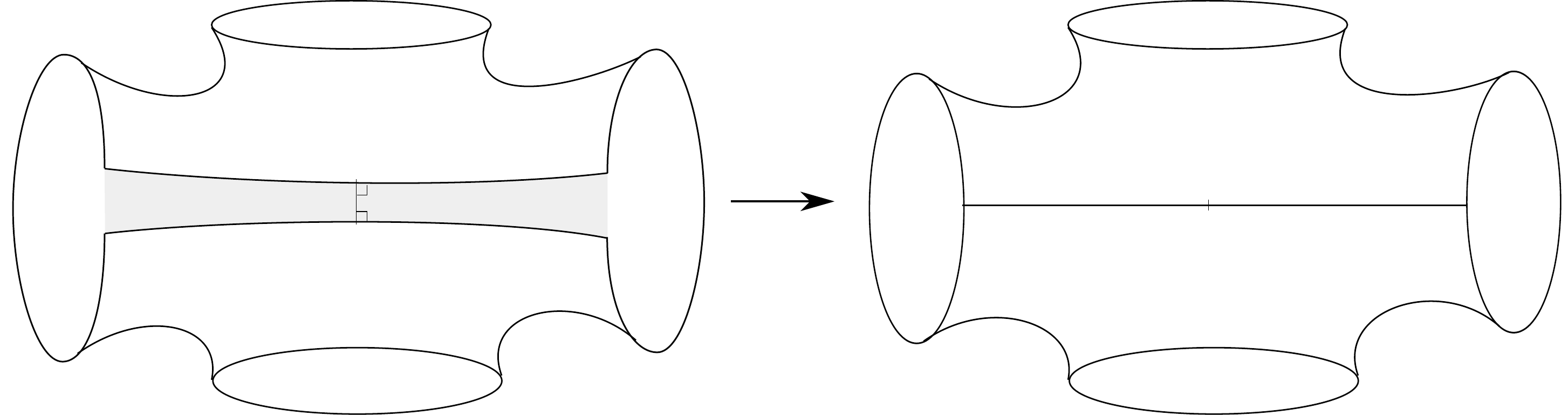}
\vspace{8pt}
\caption{A strip deformation along a single arc.}
\label{fig:strip} 
\end{figure}
 
Now consider a family $\rho_t$ of strip deformations of $\rho_0$, as follows. For each 
$1 \leq i \leq r$, let $\alpha_i'$ move closer and closer to the fixed arc 
$\alpha_i$ in such a way that the endpoint $p_i \in \alpha_i$ of the waist 
remains constant and the width  
\[
 d(p_i, p_i') = w_i t + O(t^2) 
 \]  
tends to zero at some linear rate $w_i \in \R^+$.
The cohomology class $[\u] \in \mathsf{H}^1(\Gamma_0, \gg_{\Ad})$ of the derivative 
of the path $\rho_t$ is called an \emph{infinitesimal strip deformation} of $\rho_0$. 
The points $p_i \in \alpha_i$ are called the \emph{waists} and the coefficients 
$w_i$ are called the \emph{widths} of the infinitesimal strip deformation.
Note that if the arcs cut $\Sigma_0$ into disks, then every closed curve must 
cross the arcs a number of times roughly proportional to its length, which 
should make plausible the fact that lengths of closed curves are decreasing at a 
uniform rate as in~\eqref{eqn:shrinking};  hence $\u \in \adm(\rho_0)$ in this case.

Danciger-Gu\'eritaud-Kassel~\cite{MR3480555} proved that every contracting infinitesimal 
deformation $\u$ is realized by an infinitesimal strip deformation. 
The realization becomes unique if further requirements are put on the strips. 
For example, let us require that each $\alpha_i$ crosses the boundary of the convex core 
$\Omega \subset \Sigma_0$ at a right angle, 
and that $p_i$ is the midpoint  of $\Omega \cap \alpha_i$.
Strip deformations of this type are naturally organized into 
an abstract simplicial complex  $\AC$, 
with:

\begin{itemize}
\item a vertex for each geodesic arc $\alpha$ which exits the convex core
$\Omega$ orthogonal to $\partial \Omega$ at both ends;
\item a $k$-dimensional simplex for each collection of $k+1$ pairwise disjoint 
geodesic arcs $\alpha_1, \ldots, \alpha_{k+1}$. 
\end{itemize}
This combinatorial object $\AC$ is the {\em arc complex\/} of $\Sigma_0$.
Note that it depends only on the topology of $\Sigma_0$.

Consider the map 
\[ 
\AC \xrightarrow{~\bfstrip~} \mathsf{H}^1(\Gamma_0, \gg_{\Ad})\]
defined as follows. 
Write any element $x \in \overline{\mathcal X}$ as a formal weighted sum of arcs
\begin{align*}
x &= w_1 \alpha_1 + \cdots + w_{k+1} \alpha_{k+1}
\end{align*}
with each $w_i > 0$ and $\sum w_i = 1$.
Then define $\bfstrip (x)$ to be the infinitesimal strip deformation for the 
arcs $\alpha_1, \ldots, \alpha_{k+1}$, where for $1 \leq i \leq k+1$, the 
waist of the infinitesimal strip at $\alpha_i$ is the midpoint of 
$\alpha_i \cap \Omega$ and the width is~$w_i$. 
Denote by $\mathcal X$ the subset of $\overline{\mathcal X}$ obtained by removing all open faces corresponding to collections of arcs which fail to cut the surface into disks. 

Penner~\cite{penner} showed that $\mathcal X$ is  homeomorphic to a ball of dimension 
one smaller than dimension of the Fricke-Teichm\"uller space 
of complete hyperbolic structures on $\Sigma_0$.  
The map $\bfstrip$ sends $\mathcal X$  
into the contracting half of the admissible cone in $\mathsf{H}^1(\Gamma, \gg_{\Ad \rho_0})$. 
The projectivization of the restriction of $\bfstrip$, denoted 
\[ 
\mathcal X \xrightarrow{~\strip~} \adm(\rho_0),
\]
is then a map between balls of the same dimension.
The main theorem of~\cite{MR3480555} (extended in~\cite{dgknew} to  when 
$\Sigma_0$ may have cusps), is:
\begin{thm}\label{thm:strips}
$\mathcal X \xrightarrow{~\strip~} \adm(\rho_0)$ is a homeomorphism.
\end{thm}
\noindent
The  proof has two parts:
first,  $\mathsf{Strip}$ is a local homeomorphism,
and second, $\mathsf{Strip}$ is proper. 
Both are nontrivial, but let us comment only on the second. 
Consider a sequence $x_n$ going to infinity in $\mathcal X$. 
There are two ways this can happen. 
First, it could be that, up to subsequence, $x_n$ converges in $\overline{\mathcal X}$ to a point 
$x_\infty \in \overline{\mathcal X} \setminus \mathcal X$, 
which is supported on arcs whose complement includes a subsurface of nontrivial topology. 
The limit \[[\u_\infty] = [\bfstrip(x_\infty)]\] of the projective classes 
$[\u_n] = \mathsf{Strip}(x_n)$  
is a projective class of infinitesimal deformations leaving 
unchanged the lengths of closed curves in this subsurface; 
hence $[\u_\infty] \in \partial \adm(\rho_0)$.
Consider second the case that $x_n$ diverges even in~$\overline{\mathcal X}$. 
Then the supporting arcs of $x_n$ become more and more complicated and, after 
taking a subsequence, converge in the Hausdorff sense to (up to twice as many) 
geodesic arcs $\beta_1, \ldots, \beta_s$ which are no longer properly embedded, 
but rather accumulate in one direction around a geodesic lamination $\Lambda$ 
in the convex core $\Omega$. The limit $[\u_\infty]$ of the strip deformations 
$[\u_n] = \mathsf{Strip}(x_n)$ should be thought of as a strip deformation for 
which the waists of the strips are infinitely deep in the lamination $\Lambda$; 
in other words, $[\u_\infty]$ is obtained by removing (infinitesimal) parabolic 
strips, each of whose thickness goes to zero as the strip winds closer and 
closer to $\Lambda$. The lengths of longer and longer closed curves 
$\gamma \in \Gamma$ which travel very close to $\Lambda$ are affected 
(proportionally) less and less by $[\u_\infty]$, showing that uniform
 contraction~\eqref{eqn:shrinking} fails for $[\u_\infty]$, so that again 
$[\u_\infty] \in \partial \P\big(\adm(\rho_0)\big)$. 
Thus $\mathsf{Strip}$ is proper.

Note that Minsky~\cite{Lams} used strip deformations with parabolic strips to 
show that there exist affine deformations of a one-holed torus which are not  proper, 
but for which the Margulis spectrum is positive. 
See the discussion in Section~\ref{sec:misc}.

\subsection{Strip deformations and crooked planes}\label{sec:strips-and-CPs}

One consequence of Theorem~\ref{thm:strips} is the resolution of the 
\emph{Crooked Plane Conjecture}, see Section~\ref{sec:crooked}.

\begin{cor}\label{cor:CPs}
Consider a discrete  embedding $\Gamma \xhookrightarrow{~\rho_0~} G$ of a free group 
$\Gamma$ of rank $r \geq 1$,  and a deformation cocycle 
$\Gamma \xrightarrow{~\u~} \gg$. Suppose that the affine action 
$\Phi_G(\rho_0, \u)$ of $\Gamma$ on $\Eto$ is properly discontinuous. 
Then there exists a fundamental domain in $\Eto$ bounded by $2r$ 
pairwise disjoint crooked planes.
\end{cor}

Before explaining the proof, a quick note about more general fundamental domains. 
Recall from Proposition~\ref{prop:proper} that any $(\rho_0, \u)$-equivariant 
vector field $Y$ which is $k$-lipschitz, for some $k < 0$, 
determines a $\Gamma$-equivariant fibration 
\[ \Eto \xrightarrow{~\varpi~}  \Ht. \]
If $\Delta\subset\Ht$ is a fundamental domain, 
then $\varpi\inv\Delta\subset\Eto$ is a fundamental domain.
The surfaces bounding $\varpi\inv\Delta$ are ruled by affine lines, 
but do not have any other particularly nice structure, 
and are far from canonical. 
Indeed much freedom exists in choosing $Y$. 
However, 
Theorem~\ref{thm:strips} 
implies $\u$ is realized uniquely as an infinitesimal strip deformation. 

As described  below,
an infinitesimal strip deformation is a $(\rho_0, \u)$-equivariant piecewise Killing vector field 
$X$ on $\Ht$. 
The vector field $X$ is discontinuous along a $\rho_0(\Gamma)$-invariant collection 
$\widetilde{\mathscr A}$ of pairwise disjoint geodesic arcs, 
namely the lifts of the arcs $\mathscr A = \{\alpha_1, \ldots, \alpha_r\}$ 
supporting the strip deformation. 
Although $X$ is only $0$-lipschitz, 
it is sufficiently contractive to define a singular version of the fibration from 
Proposition~\ref{prop:proper}. 
The surfaces in $\Eto$ that lift arcs $\widetilde\alpha \in \widetilde{\mathscr A}$ of the strip 
deformation are precisely crooked planes! 
Indeed, 
crooked planes are seen in the limit of the fibrations for $k$-lipschitz vector fields $Y$ 
converging to $X$, 
with $k \to 0^-$. 
See Figure~\ref{fig:quartic}.
Here is the precise recipe for finding crooked planes from strip deformation data. 

\begin{figure}[h]\includegraphics[width=3.5cm]{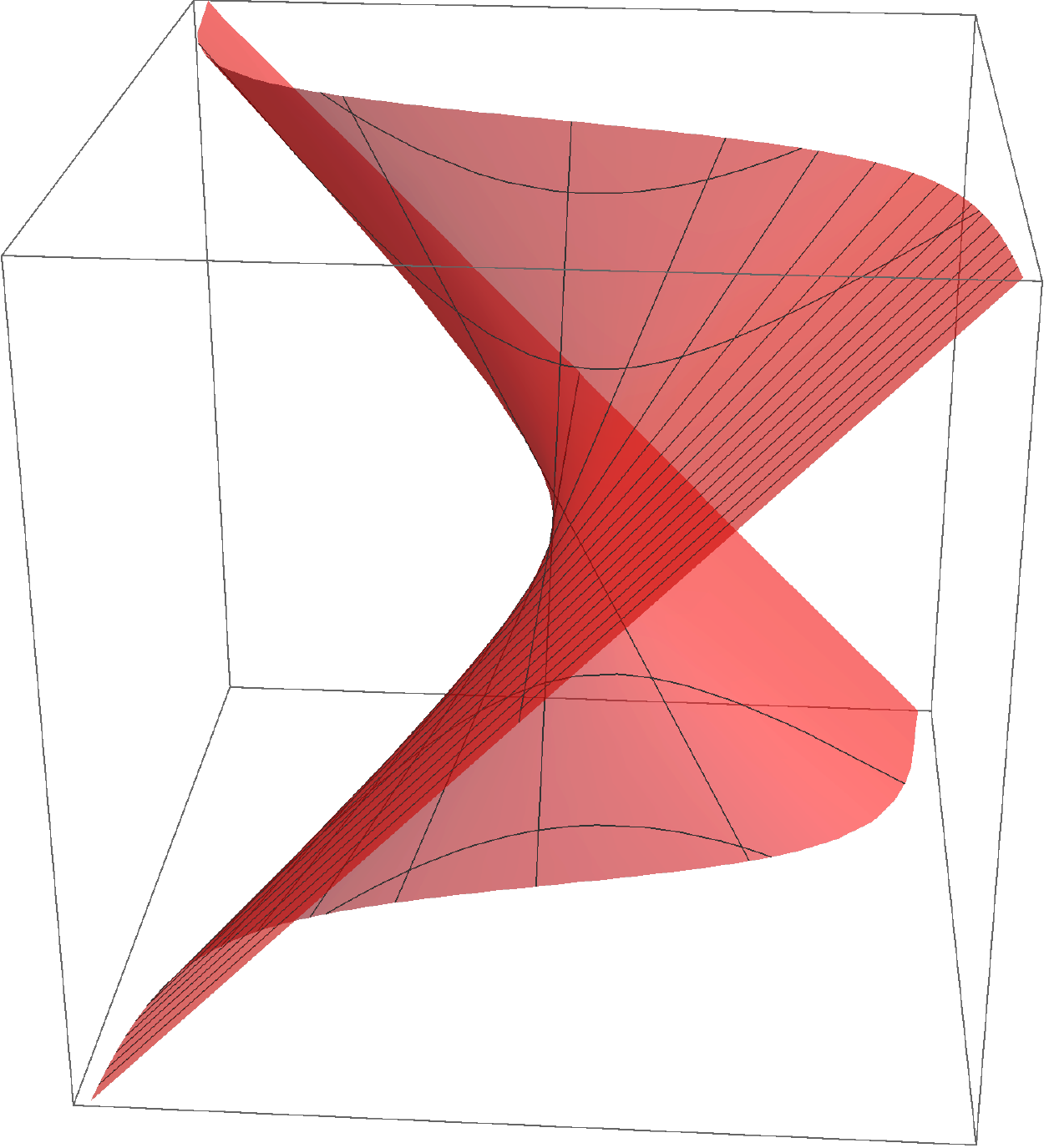}
\hspace{0.15cm}
\includegraphics[width=3.5cm]{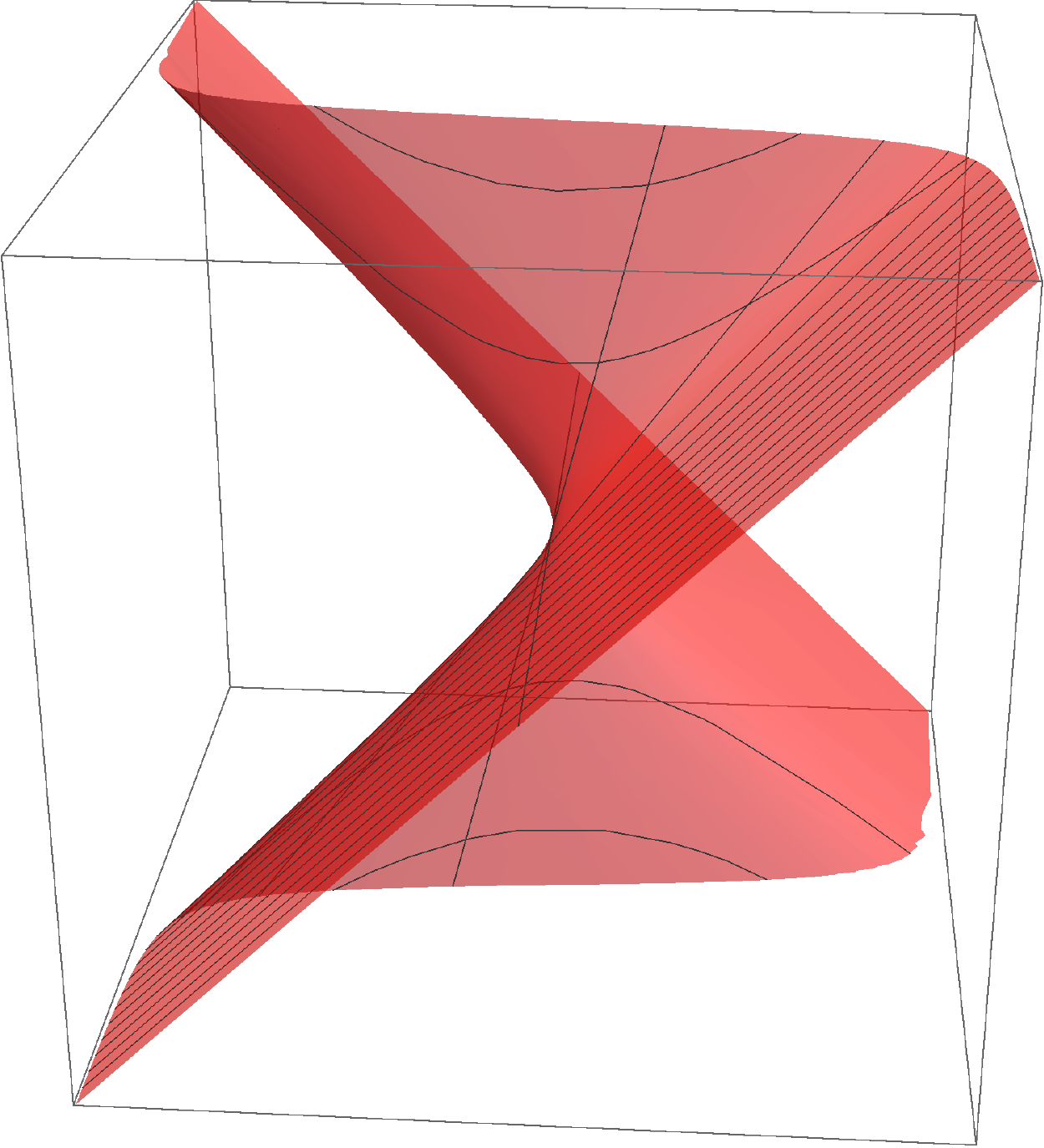}
\vspace{8pt}
\caption{The preimage of an arc $\widetilde \alpha \subset \Ht$ under the 
fibration determined by a $k$-lipschitz vector field $Y$, 
as $Y$ converges to an infinitesimal strip deformation $X$ along $\widetilde \alpha$ with 
\ $d_Y'(p,q) \to 0$ for all $p,q\in\tilde\alpha$,
as $Y\to Y_\infty$.  
The limit is a crooked plane.}
\label{fig:quartic} 
\end{figure}

First, we describe in more detail the $(\rho_0, \u)$-equivariant vector field 
$X$ associated to the strip deformation realizing $X$. The connected components 
of the complement of $\alpha_1 \cup \cdots \cup \alpha_r$ in $\Sigma_0$ are 
each homeomorphic to a disk (if the collection $\mathscr A$ of arcs is maximal, 
each component is a hyper-ideal triangle). We denote the set of these components 
by~$\mathscr T$. The lift to $\Ht$ of $\mathscr T$ is denoted $\widetilde{\mathscr T}$; 
its elements are the \emph{tiles} of a $\rho_0(\Gamma)$-invariant tiling of $\Ht$. Then:
\begin{itemize}
\item The restriction of $X$ to each of the tiles $\Delta \in \widetilde{\mathscr T}$ 
is a Killing field $\xi_\Delta \in \gg$. 
\item If two tiles $\Delta, \Delta' \in \widetilde{\mathscr T}$ are adjacent 
along an arc $\widetilde{\alpha} \in \widetilde{\mathscr A}$, then the 
relative motion of $\Delta$ with respect to $\Delta'$, namely the difference 
$$\psi_{\Delta, \Delta'} := \xi_\Delta - \xi_{\Delta'} \in \gg,$$ is an 
infinitesimal translation along an axis orthogonal to $\widetilde{\alpha}$ in the 
direction of $\Delta'$. If $\widetilde{\alpha}$ is a lift of $\alpha_i$, then 
the axis of $\psi_{\Delta,\Delta'}$ intersects $\widetilde{\alpha}$ at the lift 
of the waist $p_i \in \alpha_i$ and the velocity of the translation is equal to 
the width $w_i$, as defined in Section~\ref{sec:strips}. Note 
$\psi_{\Delta,\Delta'} = - \psi_{\Delta', \Delta}$.
\item Since $X$ must be discontinuous along $\widetilde{\alpha}$, 
we define $X$ along $\widetilde{\alpha}$ to agree with the Killing field
 $$\mathbf{v}_{\widetilde{\alpha}} := (\xi_{\Delta} + \xi_{\Delta'})/2$$ 
which is the average of the Killing fields associated to the adjacent tiles 
$\Delta$ and $\Delta'$. Think of $\mathbf{v}_{\widetilde{\alpha}}$ as the 
infinitesimal motion of the arc $\widetilde{\alpha}$ under the deformation.
\end{itemize}

Each arc $\widetilde{\alpha} \in \widetilde{\mathscr A}$ together with its 
infinitesimal motion $\mathbf{v}_{\widetilde{\alpha}}$ determines a crooked plane,
$$\CP_{\widetilde \alpha} := \CP(\mathbf{v}_ {\widetilde \alpha}, \widetilde \alpha).$$
Here we identify $\Eto$ with the (affine space of the) Lie algebra $\gg$, 
as in Section~\ref{sec:crooked}, and recall that for $\ell$ a geodesic in $\Ht$ 
and $\mathbf v \in \gg$ a Killing vector field, the crooked plane 
$\CP(\mathbf v, \ell) \subset \gg$ is the collection of Killing fields 
$\mathbf w \in \gg$ such that $\mathbf w - \mathbf v$ has a non-attracting
fixed point on the closure $\overline{\ell}$ of $\ell$ in $\overline{\Ht}$. 
Equipping $\ell$ with a transverse orientation, the closed 
crooked halfspace $\CH(\mathbf v, \ell) \subset \gg$ is the collection of 
Killing fields $\mathbf w \in \gg$ such that $\mathbf w - \mathbf v$ has a 
non-attracting fixed point on the closure  in $\overline{\Ht}$ of the 
positive halfspace $h_\ell$ bounded by $\ell$.

\begin{lemma}\label{lem:ArcHalfspace}
Let $\widetilde{\alpha}, \widetilde{\alpha}' \in \widetilde{\mathscr A}$ and 
endow each arc with a transverse orientation so that the positive halfspace 
of  $\widetilde{\alpha}$ is contained in that of $\widetilde{\alpha}'$. Then 
\begin{align}\label{eqn:inclusion}
\CH(\mathbf{v}_{\widetilde{\alpha}}, \widetilde{\alpha}) \subset
 \mathrm{Int}\left(\CH(\mathbf{v}_{\widetilde{\alpha}'}, \widetilde{\alpha}')\right).
\end{align} 
\end{lemma}

\begin{proof} 
First, consider the case that $\widetilde{\alpha}, 
\widetilde{\alpha}' \in \widetilde{\mathscr A}$ are two distinct arcs on the 
boundary of a common tile $\Delta''$. 
Let $\Delta$ (respectively $\Delta'$) denote the tile on the other side of 
$\widetilde{\alpha}$ (respectively $\widetilde{\alpha}'$) from $\Delta''$. 
Then the vertices of the crooked planes $\CP_{\widetilde \alpha}$ and 
$\CP_{\widetilde \alpha'}$ may be written:
\begin{align*}
\mathbf{v}_{\widetilde{\alpha}} &= \xi_{\Delta''} + (1/2)\psi_{\Delta, \Delta''},
 &\mathbf{v}_{\widetilde{\alpha}'} &= \xi_{\Delta''} + (1/2)\psi_{\Delta', \Delta''}.
\end{align*}
Hence the crooked halfspace 
$\CH_{\widetilde \alpha}:= \CH(\mathbf{v}_{\widetilde{\alpha}},\widetilde{\alpha})$ 
is obtained from $\CH(\xi_{\Delta''}, \widetilde \alpha)$ by translating in the direction 
$(1/2)\psi_{\Delta, \Delta''}$ and similarly 
$\CH_{\widetilde \alpha'}:= \CH(\mathbf{v}_{\widetilde{\alpha}'}, \widetilde{\alpha}')$ 
is obtained from $\CH(\xi_{\Delta''}, \widetilde \alpha')$ by translating in the 
direction $(1/2)\psi_{\Delta', \Delta''}$. 

In fact, the two crooked halfspaces, 
$\CH(\xi_{\Delta''},\widetilde \alpha)$ and 
$\CH(\xi_{\Delta''}, \widetilde \alpha')$
are nested and their bounding crooked planes meet only at the vertex:
\begin{align}\label{eqn:nested}
\CH(\xi_{\Delta''}, \widetilde \alpha) \subset  \mathrm{Int}\left(\CH(\xi_{\Delta''}, 
\widetilde \alpha')\right) \cup \{\xi_{\Delta''}\}.
\end{align}
The key observation is Lemma~\ref{lemma:StemQuadrants}.
If $\mathbf{w} \in \gg$ is an infinitesimal translation along an axis orthogonal 
to $\widetilde{\alpha}$ and pushes toward the negative side of $\widetilde{\alpha}$, 
then affine translation by $\mathbf{w}$ pushes the crooked halfspace 
$\CH(\mathbf{0}, \widetilde{\alpha})$ inside of itself. In particular,
\begin{align}
\CH_{\widetilde{\alpha}} &=  \mathbf{v}_{\widetilde{\alpha}} + 
\CH(\mathbf{0}, \widetilde{\alpha})\label{eqn:nested2} = 
\xi_{\Delta''} + (1/2)\psi_{\Delta, \Delta''} + \CH(\mathbf{0}, \widetilde{\alpha}) \\
&\subset \xi_{\Delta''} + \CH(\mathbf{0}, \widetilde{\alpha}) = 
\CH(\xi_{\Delta''}, \widetilde{\alpha}). \nonumber
\end{align}
Similarly, $\CH(\xi_{\Delta''}, \widetilde{\alpha}') + 
(1/2)\psi_{\Delta'', \Delta'} \subset  \CH(\xi_{\Delta''}, \widetilde{\alpha}')$, and hence:
\begin{align}\label{eqn:nested3} 
\CH(\xi_{\Delta''}, \widetilde{\alpha}') &\subset \CH(\xi_{\Delta''}, 
\widetilde{\alpha}') - (1/2)\psi_{\Delta'', \Delta'}\\
&= \CH(\xi_{\Delta''}, \widetilde{\alpha}') + (1/2)\psi_{\Delta', \Delta''}\nonumber \ = 
\ \CH(\mathbf{v}_{\widetilde{\alpha}'}, \widetilde{\alpha}'). \nonumber
\end{align}
So~\eqref{eqn:inclusion} follows from~\eqref{eqn:nested}, \eqref{eqn:nested2}, 
and ~\eqref{eqn:nested3} upon observing that the vertex $\xi_{\Delta''}$ is not 
contained in $\CH_{\widetilde{\alpha}}$.

Now a simple inductive argument shows that~\eqref{eqn:inclusion} indeed holds for 
\emph{any} pair of arcs $\widetilde{\alpha}, \widetilde{\alpha}'$ oriented so 
that the positive halfspace of $\widetilde{\alpha}$ is contained in that 
of $\widetilde{\alpha'}$.
\end{proof}
Observe that by Lemma~\ref{lem:ArcHalfspace}, the crooked planes in the collection 
\begin{align}\label{eqn:collection}
\left\{\CP_{\widetilde \alpha} := \CP(\mathbf{v}_ {\widetilde \alpha}, \widetilde \alpha):
 \widetilde \alpha \in \widetilde{\mathscr A} \right\}
\end{align}
are pairwise disjoint. Further, the halfspaces bounded by these crooked planes 
obey the same inclusion relations that hold for halfplanes in $\Ht$ bounded by 
the corresponding arcs.
To find a fundamental domain in $\Eto \cong \gg$ bounded by disjoint crooked planes, 
one simply chooses the crooked planes associated to a subset of arcs of
$\widetilde{\mathscr{A}}$ that bound a fundamental domain for the action on 
$\Ht$. This proves Corollary~\ref{cor:CPs}.

In the same spirit of Section~\ref{sec:Lipschitz}, there is a parallel theory of 
strip deformations and crooked planes in the setting of three-dimensional anti 
de Sitter geometry, see~\cite{MR3533195} and \cite{MR3323635}.

\subsection{Two-generator groups}\label{sec:ex}
We now focus on the special case that the free group $\Gamma$ has rank two, 
corresponding to Euler characteristic $\chi(\Sigma_0) = -1$. 
There are four possible topological types for~$\Sigma_0$. 
In each case, the arc complex $\overline{\mathcal X}$ is two-dimensional, 
but the combinatorics is quite different across the cases (see Figure~\ref{fig:small}). 
This results in a substantially different picture of the (projectivized) 
cone of proper deformations $\adm(\rho_0)$, 
depending on the topology of $\Sigma_0$ (see again Figure~\ref{fig:small}). 
We describe the qualitative behavior in each of the four cases below in 
the language of Theorem~\ref{thm:strips}.
However, we remark that the understanding of $\adm(\rho_0)$ in the rank two case, 
in particular each description below, 
predates Theorem~\ref{thm:strips}. 
Charette- Drumm-Goldman~\cite{MR2653729, MR3180618, MR3569564} % 
described a tiling of 
$\adm(\rho_0)$ according to which isotopy classes of crooked planes embed 
disjointly in the associated Margulis spacetime. 
From this, they deduced the Crooked Plane Conjecture, Corollary~\ref{cor:CPs}, 
in the rank two case. 
The relationship between $\adm(\rho_0)$ and the arc complex 
$\overline{\mathcal X}$ of $\Sigma_0$ is already apparent in this 
work, which was an important precursor to Theorem~\ref{thm:strips} and Corollary~\ref{cor:CPs}. 

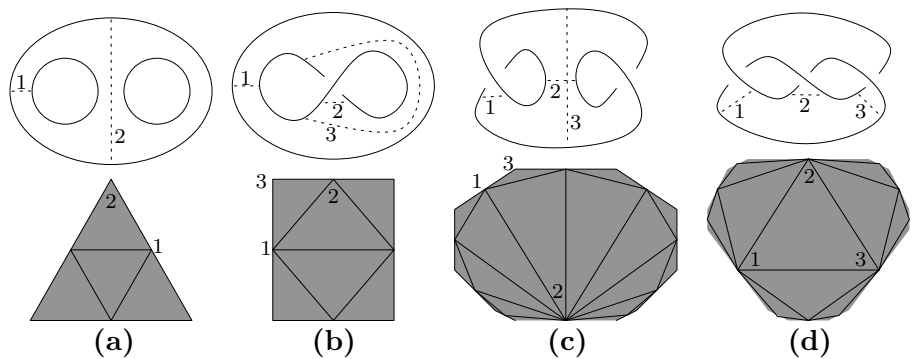
\begin{figure}[h]
\labellist
\small\hair 2pt
\pinlabel {\bf (a)} [b] at 47 -18
\pinlabel {\bf (b)} [b] at 145 -18
\pinlabel {\bf (c)} [b] at 249 -18
\pinlabel {\bf (d)} [b] at 357 -18
\pinlabel {${}_1$} at 6 107
\pinlabel {${}_1$} at 67 34
\pinlabel {${}_1$} at 106 109
\pinlabel {${}_1$} at 115 33
\pinlabel {${}_1$} at 215 96
\pinlabel {${}_1$} at 209 63
\pinlabel {${}_1$} at 326 94
\pinlabel {${}_1$} at 333 28
\pinlabel {${}_2$} at 50 83
\pinlabel {${}_2$} at 46 54
\pinlabel {${}_2$} at 147 94
\pinlabel {${}_2$} at 145 56
\pinlabel {${}_2$} at 244 103
\pinlabel {${}_2$} at 245 14
\pinlabel {${}_2$} at 355 97
\pinlabel {${}_2$} at 357 65
\pinlabel {${}_3$} at 144 84
\pinlabel {${}_3$} at 113 63
\pinlabel {${}_3$} at 253 89
\pinlabel {${}_3$} at 223 71
\pinlabel {${}_3$} at 380 94
\pinlabel {${}_3$} at 380 28
\endlabellist
\centering
\includegraphics[width=12cm]{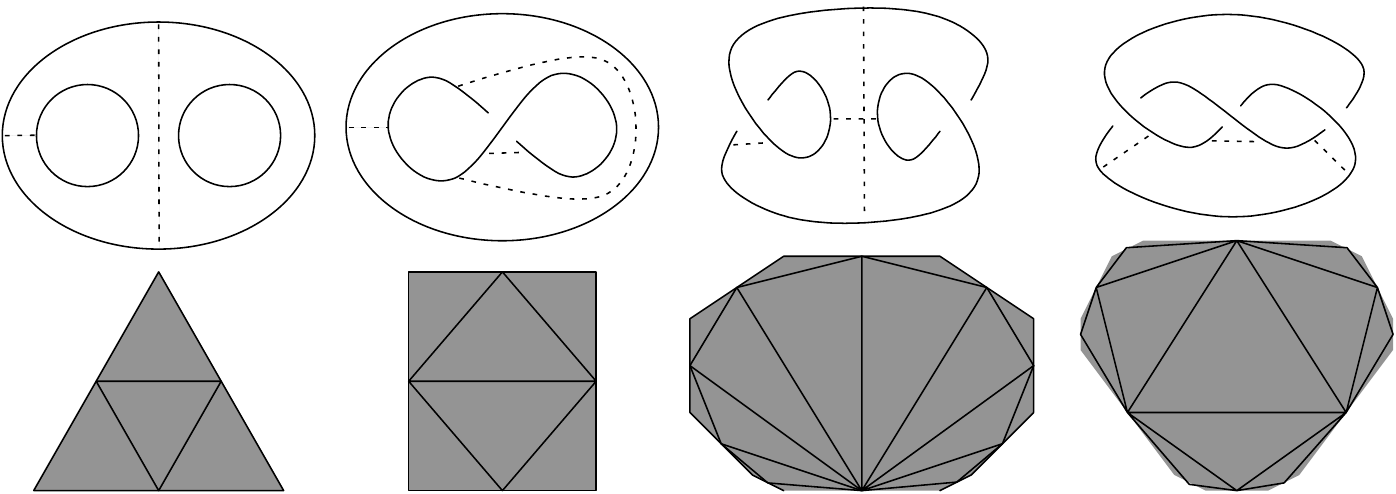}
\vspace{8pt}
\caption{Four surfaces of small complexity (top) and their arc complexes, 
mapped under $\mathsf{Strip}$ to the closure of $\adm(\rho_0)$ in an affine 
chart of $\P(\mathsf{H}^1(\Gamma, \gg_{\Ad \rho_0}))$ (bottom). 
Some arcs are labeled by  Arabic numerals. 
Figure originally appeared in~\cite{MR3480555}.}
\label{fig:small} 
\end{figure}

\smallskip

\noindent {\bf (a)} \textbf{Three-holed sphere}:
The arc complex $\overline{\mathcal X}$ has $6$ vertices, $9$ edges, $4$ faces.
Its image $\mathsf{Strip}(\overline{\mathcal X})$ is a triangle whose sides stand 
in natural bijection with the three boundary components of the convex core of 
$\Sigma_0$: an infinitesimal deformation $\u$ of~$\rho_0$ lies in a side of 
the triangle if and only if it fixes the length of the corresponding boundary 
component, to first order.
The set $\adm(\rho)=\mathsf{Strip}(\mathcal X)$ is the interior of the triangle. 
See also the left part of Figure~\ref{fig:TwoGeneratorDefSpaces1}.

\smallskip

\noindent {\bf (b)} \textbf{Two-holed projective plane}:
The arc complex $\overline{\mathcal X}$ has $8$ vertices, $13$ edges, $6$ faces.
Its image $\mathsf{Strip}(\overline{\mathcal X})$ is a quadrilateral.
The horizontal sides of the quadrilateral correspond to infinitesimal 
deformations $\u$ that fix the length of a boundary component.
The vertical sides correspond to infinitesimal deformations that fix the 
length of one of the two simple closed curves running through the half-twist.
The set $\adm(\rho_0)=\mathsf{Strip}(\mathcal X)$ is the interior of the quadrilateral. 
See also the right part of Figure~\ref{fig:TwoGeneratorDefSpaces1}.

\smallskip

\noindent {\bf (c)} \textbf{One-holed Klein bottle}:
The arc complex $\overline{\mathcal X}$ is infinite, with one vertex of infinite degree 
and all other vertices of degree either $2$ or~$5$.
The closure of $\mathsf{Strip}(\overline{\mathcal X})$ is an infinite-sided 
polygon with sides indexed in $\mathbb Z \cup\{\infty\}$.
The exceptional side has only one point in $\mathsf{Strip}(\overline{X})$, 
and corresponds to infinitesimal deformations that fix the length of the 
only nonperipheral, two-sided simple closed curve $\gamma$, which goes 
through the two half-twists.
The group $\mathbb Z$ naturally acts on the arc complex $\overline{\mathcal X}$, 
via Dehn twists along~$\gamma$.
All nonexceptional sides are contained in $\mathsf{Strip}(\overline{\mathcal X})$ 
and correspond to infinitesimal deformations that fix the length of some curve, 
all these curves being related by some power of the Dehn twist along~$\gamma$.
The set $\adm(\rho_0)=\mathsf{Strip}(\mathcal X)$ is the interior of the polygon. 
See also the left part of Figure~\ref{fig:TwoGeneratorDefSpaces2}.
\smallskip

\noindent {\bf (d)} \textbf{One-holed torus}:
The arc complex $\overline{\mathcal X}$ is infinite, with all vertices of 
infinite degree; it is known as the \emph{Farey triangulation}.
The arcs are parameterized by $\mathbb P^1(\mathbb{Q})$.
The closure of $\mathsf{Strip}(\overline{\mathcal X})$ contains infinitely many 
segments in its boundary. 
These segments, also indexed by $\mathbb{P}^1(\mathbb{Q})$, are in natural 
correspondence with the simple closed curves.
However, the boundary is not the union of these segments; there are additional 
points corresponding to deformations for which the length of every curve 
decreases, but the length of some lamination remains constant.

The  structure of the boundary of $\adm(\rho_0)$ in this case was described in
Gu\'eritaud~\cite{MR3485334} and Goldman-Labourie-Minsky-Margulis~\cite{Lams}. 
See also the right part of Figure~\ref{fig:TwoGeneratorDefSpaces2}.
For more details on the affine deformations of nonorientable surfaces, 
compare also Goldman-Laun~\cite{GoldmanLaun} and Laun's thesis~\cite{MR3553574}.

\subsection{Beyond free groups: right-angled Coxeter groups}\label{sec:racg}

The existence of proper affine actions by nonabelian free groups suggests the 
possibility that other finitely generated groups which are not virtually 
solvable might also admit proper affine actions.
However, in the more than thirty years since Margulis's discovery, 
very few examples have appeared.
In particular, until recently, all known examples of word hyperbolic 
groups acting properly by affine transformations on~$\R^n$ were virtually free groups. 
To conclude this section, we summarize further work of 
Danciger-Gu\'ertaud-Kassel~\cite{dgk-racg} that generalizes the ideas 
of Section~\ref{sec:vector-fields} to give many new examples both word hyperbolic and not.%

\begin{thm}{\cite[Thm 1.1]{dgk-racg}} \label{thm:racg}
Any right-angled Coxeter group on $k$ generators admits proper affine 
actions on $\R^{k(k-1)/2}$.
\end{thm}

A right-angled Coxeter group $\Gamma$ is a finitely presented group of the form
$$\Gamma = \langle s_1, \ldots, s_g \ | \  (s_i s_j)^{m_{ij}}= 1, 
\forall \ 1 \leq i,j \leq g\rangle$$
where $m_{ii}= 1$, i.e. each generator $s_i$ is an involution, and 
$m_{ij} = m_{ji} \in \{2, \infty\}$ for $i \neq j$, meaning two 
distinct generators either commute ($m_{ij} = 2$) or have no relation ($m_{ij} = \infty$). 
Some examples come from \emph{reflection groups} in hyperbolic space. Indeed, the 
group generated by reflections in the faces of a right-angled polyhedron in $\Hn$ 
is a right-angled Coxeter group. Though simple to define, right-angled Coxeter 
groups have a rich structure and contain many interesting subgroups.
As a corollary to Theorem~\ref{thm:racg}, the following groups admit proper affine actions:
\begin{itemize}
\item the fundamental group of any closed orientable surface of negative Euler characteristic;
\item any right-angled Artin group, see~\cite{dj00};
\item any virtually special group, see~\cite{hw08};
\item any Coxeter group (not just right-angled), see~\cite{hw10};
\item any cubulated word hyperbolic group, using Agol's virtual 
specialness theorem~\cite{ago13};
\item therefore, all fundamental groups of closed hyperbolic $3$-manifolds, 
using \cite{sag95,km12}: see \cite{bw12};
\item the fundamental groups of many other $3$-manifolds, see \cite{wis11,liu13,pw18}.
\end{itemize}

Januszkiewicz--\'Swi\c{a}tkowski \cite{js03} found word hyperbolic right-angled 
Coxeter groups of arbitrarily large virtual cohomological dimension. See 
also \cite{osa13} for another construction.
Hence, another consequence of Theorem~\ref{thm:racg} is:

\begin{cor} \label{cor:vcd}
There exist proper affine actions by word hyperbolic groups of arbitrarily 
large virtual cohomological dimension.
\end{cor}

The Auslander Conjecture is equivalent to the statement that a group acting 
properly by affine transformations on~$\R^n$ is either 
virtually solvable, or has virtual cohomological dimension $<n$.
In the examples from Theorem~\ref{thm:racg}, the dimension $n = k(k-1)/2$ 
of the affine space grows quadratically in the number of generators $k$, while 
the virtual cohomological dimension of the Coxeter group acting is naively bounded above by~$k$.
Hence, Corollary~\ref{cor:vcd} is far from giving counterexamples to the Auslander Conjecture.

The affine actions from Theorem~\ref{thm:racg} come from infinitesimal 
deformations of representations into a Lie group $G$ as in 
Section~\ref{subsec:deform}, for $G$ an indefinite orthogonal group.
Indeed, a right-angled Coxeter group $\Gamma$ on $k$ generators 
(say, infinite and irreducible) admits explicit families of discrete 
\emph{reflection group} embeddings 
\[ \Gamma\xrightarrow{~\rho~}\Orth(p,q+1) =: G \] 
for $k = p+q+1$, 
which have long been studied by Tits, Vinberg, and others. 
The strategy from Section~\ref{sec:vector-fields} of ensuring properness of the 
affine action from contraction of the deformation works well when $q = 0$.
In that case, each representation $\rho$ acts by reflections in the walls of a 
right-angled polytope~$\Delta_{\rho}$ in hyperbolic space $\Hp$. 
For two such representations $\rho, \rho'$, 
natural $(\rho, \rho')$-equivariant maps $f$ 
are described explicitly by mapping $\Delta_\rho$ to $\Delta_{\rho'}$ 
projectively, walls-to-walls, 
and extending equivariantly by reflections. 
Deformations $\rho'$ for which the maps $f$ are Lipschitz contracting are found, 
roughly, 
by pushing the walls of $\Delta_\rho$ closer together. 
The derivative of an appropriate path of such contracting Lipschitz deformations, 
for $\rho'$ smoothly converging to $\rho$, 
gives a $(\rho_0, \u)$-equivariant contracting lipschitz vector field and hence a proper affine action 
$\Phi_G(\rho_0, \u)$ by the argument given in Proposition~\ref{prop:proper}.

It should be noted that, still in the case $q =0$,  
the dimension of the representations, 
and hence of the corresponding affine actions, 
may sometimes be reduced: 
If for some $n \geq 2$, $\Gamma$ admits an action on $\Hn$ generated by reflections in some polytope $\Delta$, 
then a contracting deformation as above may be found in $\Hyp^{n+c-1}$ 
if the faces of $\Delta$ may be colored with $c$ colors so that neighboring faces have different color 
(see~\cite[Prop. 4.1]{dgk-racg}). 
For example, 
if $\Gamma$ is the group generated by reflections in a right-angled $2m$-gon in the hyperbolic plane (for $m \geq 3$), then we may take $c = 2$. 
There exists a path of deformations of this reflection group into % 
the isometry group of 
$\Hyp^3 = \Hyp^{2+2-1}$ for which tangent vectors to the 
path give proper affine actions in dimension 
$6 = \dim(\mathfrak{so}(3,1))$. Note that in this example, 
$\Gamma$ contains surface subgroups of finite index.

The general case of Theorem~\ref{thm:racg} requires %
indefinite orthogonal groups of higher $\R$-rank, that is,  $q > 0$; 
indeed, not all right-angled Coxeter groups may be realized as reflection groups in some hyperbolic space.
Here, one could attempt the contraction strategy of Section~\ref{sec:vector-fields} 
in the higher rank Riemannian symmetric space $\o{X}$ of $G$. 
However, the most natural space  in which to see the geometry of the Tits--Vinberg representations 
$\rho: \Gamma \to G$ is in a \emph{pseudo-Riemannian symmetric space}, 
namely the pseudo-Riemannian analogue $\Hpq \subset \R\P^{k-1}$ 
of $\Hp$ in signature $(p,q)$. 
Indeed, as above, the $\rho$-action of $\Gamma$ is by reflections in the walls of a natural fundamental domain, 
a certain polytope $\Delta_\rho \subset \Hpq$, and natural $(\rho,\rho')$-equivariant 
maps $f$ are defined by taking $\Delta_\rho$ projectively to $\Delta_{\rho'}$, 
walls-to-walls. 
Further, since the ``distances'' in $\Hpq$ are computed by a 
simple cross-ratio formula, 
similar to $\Hp$ (in the projective model), the ``contraction'' 
properties of the maps $f$ are easy to check locally in the fundamental domain $\Delta_\rho$.
Theorem~\ref{thm:racg} is proved by employing a version of the contraction 
strategy from the $\Hp$ setting, adjusted and reinterpreted appropriately to 
work in the pseudo-Riemannian space $\Hpq$.
Despite the obvious hurdle that $\Hpq$ is not a metric space, 
enough structure survives for this approach to work. 
One key observation is that $\rho(\Gamma)$-orbits 
in $\Hpq$ escape only in \emph{spacelike} (that is, positive) directions, 
in which their growth resembles that of actions on $\Hp$.

\section{Higher dimensions}

\subsection{Non-Milnor representations}\label{sec:nonMilnorHigherDim}
Margulis's original work can be reinterpreted as the discovery of the first 
known non-Milnor representation (see
Definition~\ref{def:Milnor}), namely the standard representation of~$\SO(2, 1)$ 
on~$\R^3$. We now discuss the question of identifying Milnor and non-Milnor 
representations in higher dimensions. Recall Proposition~\ref{prop:EigenvalueOne}, 
that if $\rho$ does not have the property that every element acts with one as 
an eigenvalue, it is automatically Milnor.
Observe, as in Section~\ref{sec:Deformations}, that the standard representation of 
$\SOto$ is isomorphic to the adjoint representation of $\SOto$ and more 
generally, that the adjoint representation of any semisimple Lie group $G$ 
has the property that every element %
of infinite order %
acts with $1$ as an eigenvalue. 

\begin{thm}[Smilga~\cite{MR3477891}]  
\label{thm:proper_affine_adjoint}
For every semisimple real linear Lie group~$G$, the adjoint 
representation is non-Milnor whenever $G$~is not compact.
\end{thm}

Note that the proper affine actions by right-angled Coxeter groups of 
Section~\ref{sec:racg} have linear part in adjoint representations of special orthogonal groups.

In a different direction, the work of Abels-Margulis-Soifer~\cite{AMS02, AMS11} 
definitively settles the case of the standard representation of the special orthogonal groups
$\SOpq$ on $\R^{p,q}$:

\begin{thm}[Abels-Margulis-Soifer~\cite{AMS02, AMS11}]
\label{proper_affine_orthogonal}
Let $p \geq q$.
Then the standard representation of~$\SO(p, q)$ on~$\R^{p+q}$ is:
\begin{enumerate}
\item Milnor if
\begin{itemize}
\item  $p - q \neq 1$, or   
\item  $p - q = 1$ and $p$~is odd;
\end{itemize}  
\item non-Milnor 
if~$p - q = 1$ and $p$~is even.
\end{enumerate}
\end{thm}
\noindent
Observe that when $p-q$ is even, it is not the case that every element of 
$\SO(p,q)$ has one as an eigenvalue and Proposition~\ref{prop:EigenvalueOne} 
implies that $\SO(p,q)$ is Milnor. 
The general case, however, involves the detailed analysis from Margulis's original argument.
The case when $q = p-1$ is the most interesting. 
Then a Margulis invariant $\alpha$ may be defined for elements with regular linear holonomy 
and the Opposite Sign Lemma holds.
If $p$ is odd, then $\alpha(\gamma) = -\alpha(\gamma^{-1})$, and hence no proper 
affine actions of $\Ft$ exists. %
However, if $p$ is even, then $\alpha(\gamma) = \alpha(\gamma^{-1})$ as in Lemma~\ref{lem:alpha}.\eqref{alphan}, 
so there is no obvious sign obstruction to proper affine actions.
Indeed, Margulis's construction may be generalized in this case.

Smilga~\cite{MR3226793} constructs fundamental domains for the proper actions of $\Ft$ 
with linear part in  $\SO(p,p-1)$ with $p = 2k+2$ even. 
They are bounded by hypersurfaces inspired by the crooked planes of Section~\ref{sec:crooked}, but these hypersurfaces are curved rather than piecewise flat. Burelle-Treib~\cite{BT18}, on the other hand, have found a generalization of crooked planes to $\SO(2k+2,2k+1)$ constructed by flat hypersurfaces, which give rise to fundamental domains for the action of the group on the sphere (quotient of $\mathbb{R}^{4k+3} \setminus \{0\}$ by positive scalars) minus the limit set.
The Burelle-Treib construction 
 may very likely be extended to fundamental polyhedra in the affine space.
In the same setting, Ghosh-Treib~\cite{GhTr} proved an analogue of the 
Goldman-Labourie-Margulis properness criterion~\ref{thm:GLM}. 
Following Ghosh's  earlier work~\cite{MR3668058} on the dynamical structure of Margulis spacetimes, 
they interpret proper affine actions as an 
extension of Labourie's Anosov representations~\cite{MR2221137} to the non-reductive context.

Smilga gave a sufficient condition for an irreducible 
representation of a semisimple group to be non-Milnor,
which is conjectured~\cite{Smi16b} to be {\em necessary\/} as well.

Let $G$~be a semisimple real Lie group with Lie algebra $\gg$.
Choose in~$\gg$ a \emph{Cartan subspace} $\aa$, and a system $\Sigma^+$ of positive restricted roots.
Recall  that $\aa$ is a maximal abelian subalgebra of $\gg$ consisting of hyperbolic elements,
and a {\em restricted root\/} is an element $\alpha \in \aa^*$ such that the restricted root space
\[
\gg^\alpha := \setsuch{Y \in \gg}{\forall X \in \aa, \; [X, Y] = \alpha(X)Y }
\]
is nonzero. 
Restricted roots form a root system $\Sigma$; 
a system of positive roots $\Sigma^+$ is a subset of $\Sigma$ contained in a halfspace, 
such that 
\[\Sigma = \Sigma^+ \sqcup -\Sigma^+.\] 

Let $A := \exp(\mathfrak{a})$, 
and let $L$ be the centralizer of $\mathfrak{a}$ in $G$.
The {\em longest element\/} of the {\em restricted Weyl group\/} 
$W := N_G(A)/Z_G(A)$ 
is the unique element~$w_0$ such that $w_0(\Sigma^+) = \Sigma^-$. We choose some representative $\tilde{w_0} \in G$ of this element.
(Compare \cite{MR1920389}.)

\begin{thm}[Smilga \cite{Smi16b}]
\label{proper_affine_sufficient}
Suppose that
$G \xrightarrow{~\rho~} \GL(\V)$ 
is an irreducible representation
such that 
\begin{enumerate}[(i)]
\item $\forall l \in L,\; \rho(l) \cdot v = v$; 
\item  $\rho(\tilde{w}_0) \cdot v \neq v$
\end{enumerate}
for some $v\in\V.$
Then $\rho$~is non-Milnor.
\end{thm}
\noindent
Le Floch-Smilga~\cite{lFS18} have 
classified %
such $\rho$, when $G$~is split.

Since $L\supset A$, 
the first condition implies that 
the %
{\em restricted weight space associated to~$0$\/}
\[ \V^A = \setsuch{v \in \V}{\forall a \in A,\; \rho(a) \cdot v = v}  \]
is nonzero. 
Equivalently, every element of~$\rho(G)$ has~$1$ as an eigenvalue,  
consistent with Proposition~\ref{prop:EigenvalueOne}.

The proofs of Theorems~\ref{proper_affine_orthogonal}.(1), 
\ref{thm:proper_affine_adjoint} 
and~\ref{proper_affine_sufficient} all follow the same basic template 
as Margulis's original proof (see Section~\ref{sec:OriginalProof}), 
although the more general proofs are more complicated.

The main idea is to decompose the representation space~$\V$ as a direct sum of three subspaces
\[\V = \V^\subg \oplus \V^\sube \oplus \V^\subl,\]
and then construct a ``generalized Schottky group'' $\Gamma_\ell$ in~$\rho(G)$, so 
that every element $\gamma \in \Gamma_\ell$, conjugated by a suitable map, 
preserves all three spaces and has very large eigenvalues on~$\V^\subg$, 
very small eigenvalues on~$\V^\subl$ and eigenvalues ``close to~$1$'' 
on~$\V^\sube$. 
Moreover, there is then a further decomposition
\[\V^\sube = \V^t \oplus \V^r\]
such that 
every element of~$\Gamma_\ell$ 
has a conjugate which stabilizes~$\V^r$ 
and \emph{fixes} $\V^t$. A crucial point is that $\V^t \neq 0$ (this comes from 
condition~(i), as it turns out to be precisely the subspace of fixed points of~$L$).

Now let $\A$ be the affine space corresponding to the vector space~$\V$. 
Then every affine map~$g$ with linear part in~$\Gamma_\ell$, when 
conjugated by a suitable map, preserves the decomposition of~$\A$ into three subspaces
\[A = \V^\subg \oplus \A^\sube \oplus \V^\subl\]
(where $\A^\sube$ is an affine subspace of~$\A$ parallel to~$\V^\sube$). 
Its restriction to~$\A^\sube$ is then a sort of generalized ``screw-displacement'': 
it preserves the directions parallel to~$\V^t$ and~$\V^r$ and acts by pure 
translation along~$\A^t$. 
We call such affine transformations \emph{quasi-translations}, 
and we call the translation vector along~$\V^t$ of the (suitable conjugate of)~$g$ the 
\emph{Margulis invariant} $M(g)$.

In particular cases, these constructions can be simplified. For % 
$G = \SOto$ acting on $\Rto$, the space~$\V^t$ has dimension~$1$, 
so that the vector Margulis invariant $M(g)$ reduces to the classical (scalar) Margulis invariant $\alpha(g)$; 
the space~$\V^r$ is trivial, 
so that quasi-translations reduce to just translations.

Now in the general case, the analog of 
formula~\eqref{eq:scalar_margulis_invariant_of_inverse} 
is: 
\begin{equation}
\label{eq:vector_margulis_invariant_of_inverse}
M(g^{-1}) = \rho(\tilde{w_0}) \cdot M(g).
\end{equation}
Proving this turns out to be straightforward. 
Most of the effort goes into proving an analog of \eqref{eq:scalar_margulis_invariant_additivity}: 
for elements $g$ and~$h$ in~$G$ that are ``regular'', ``sufficiently transverse'' 
and ``sufficiently contracting'', 
\begin{equation}
\label{eq:vector_margulis_invariant_additivity}
M(gh) \approx M(g) + M(h).
\end{equation}
We then conclude in the same way as Margulis: we combine
\eqref{eq:vector_margulis_invariant_additivity} with
\eqref{eq:vector_margulis_invariant_of_inverse} to prove that, 
for a suitable choice of the translation parts of the generators, 
Margulis invariants of large elements of the group grow unboundedly. 
More precisely, we prescribe these translation parts in such a way that the 
Margulis invariants of the generators become equal to (some sufficiently 
large multiple of) the vector~$v$ supplied by the hypotheses 
of Theorem~\ref{proper_affine_sufficient}. Here condition~(ii) is crucial, 
as it allows the Margulis invariants of both the generators and their 
inverses to go in the same direction. 

We conjecture the converse of Theorem~\ref{proper_affine_sufficient},
which generalizes part (1) of  Theorem~\ref{proper_affine_orthogonal}. We have the following partial result. 

Say that an irreducible representation $G\xrightarrow{~\rho~}\GL(\V)$ is {\em non-swinging\/}
if and only if  $\rho$ has no nonzero $w_0$-invariant weight. 
In particular if $G$ has no simple factor of type $A_{n \geq 2}$, 
$D_{2n+1}$ or $E_6$, 
then $w_0 = -\Id$,
and every representation is non-swinging.

\begin{thm}[\cite{Smi18}]
Let $G$ be a semisimple Lie group. 
Furthermore, suppose:
\begin{itemize}
\item the group $G$ is split; 
\item 
$\rho$ is a non-swinging irreducible representation;
\item $\rho$ does \emph{not} satisfy the hypotheses of Theorem~\ref{proper_affine_sufficient}.
\end{itemize}
Then 
$\rho$ is Milnor.
\end{thm}
\noindent

\subsection{Auslander's conjecture in dimension at most six}

\begin{thm}[Abels-Margulis-Soifer~\cite{AMSAuslander}]
Let $n \leq 6$. Suppose $\Gamma < \Aff(\A^n)$ is a discrete subgroup acting 
properly and cocompactly on~$\A^n$.  Then $\Gamma$ is virtually solvable.
\end{thm}

In the case that $n\le 5$, this result follows from Fried-Goldman~\cite{MR689763} for $n\le 3$, 
and independently Tomanov~\cite{MR3628466}  and Abels-Margulis-Soifer~\cite{AMSAuslander} 
for $n=4,5$.
(As Abels-Margulis-Soifer point out in \cite{AMSAuslander} 
an earlier version of Tomanov's work  contained a gap, 
which was subsequently filled in  \cite{MR3628466}.)
Furthermore Tomanov~\cite{MR3628466} 
proposed a suggestive generalization of Auslander's
conjecture to arbitrary algebraic groups of mixed type, 
and proved this stronger statement for $n\le 5$:

\begin{conj} 
Let $G$ be a real algebraic group, and suppose that $H\subgroup G$ 
contains a maximal reductive subgroup of $G$. 
Suppose that $\Gamma\subgroup G$ acts {\em crystallographically\/}
on $G/H$ (that is, $\Gamma$ is a discrete subgroup, acts properly on
$G/H$ and the quotient $\Gamma\backslash G/H$ is compact). 
Then $\Gamma$ is virtually polycyclic.
\end{conj}

For another perspective and possible attack on Auslander's conjecture,
compare Labourie~\cite{LabourieEntropy}.

Here are 
some components of the proof of Auslander's conjecture in low dimensions.
For convenience, consider the equivalent formulation: 
Assume $\Gamma < \Aff(\A^n)$ is a discrete subgroup acting properly 
on~$\A^n$ and  $\Gamma$ is not virtually solvable. 
Then $\Gamma\backslash \A^n$ is not compact.

As in Proposition~\ref{prop:vsolvable}, consider the semisimple part 
$S$ of the identity component $G = \big(\ZClosure{\L(\Gamma)}\big)^0$ 
of the Zariski closure. 
Write $S = S_1 \cdots S_k$ as an almost direct product of simple Lie groups.
The first key ingredient is the following theorem. 

\begin{thm}[Soifer~\cite{MR1331393},
Tomanov~\cite{MR1072918,MR3628466}]
If each factor $S_i$ has real rank $\leq 1$, 
then $\Gamma$ is not cocompact.
\end{thm}

Hence assume that at least one simple factor $S_1$ has real rank $\geq 2$.
By a dynamical argument,  for certain $S$ satisfying the hypotheses, 
$\Gamma$ does not act properly.
In other cases, a cohomological dimension argument
(as in \S\ref{sec:CohomologicalDimension})  rules out cocompactness. 
What remains is a handful of interesting cases requiring some more sophisticated arguments. 

Most interesting is
the standard (six dimensional) representation of 
$G = \SOto \times \SL(3,\R)$ on $\Rto \oplus \R^3$. 
Since $G$ is not Milnor,
proper affine deformations of  non-virtually-solvable discrete subgroups of $G$ exist. 
Since $\cd(\Gamma) \le \dim(G/K) = 7$, 
where $\Gamma$ is a torsionfree discrete subgroup of $G$
and $K$ is the maximal compact subgroup of $G$,
cohomological dimension does not obstruct the existence of a affine crystallographic action %%%
on $\A^6$.
In this case, Abels-Margulis-Soifer define a Margulis invariant $\alpha(g)$ for elements $g \in \Gamma$ 
whose linear part $\L(g)$ is regular enough. 
It is essentially the Margulis invariant of the $\SOto \ltimes \Rto$ part of $g$. 
Then they prove the Opposite Sign Lemma in this setting. 
Note that the dynamics in this setting is more complicated than in the setting of Margulis spacetimes. 
One sign that things are more complicated is that the attracting subspace for a regular element $\L(g)$ 
and for its inverse $\L(g)^{-1}$ have different dimension. 
To conclude the proof in this case, they show that cocompactness implies that some subset of an orbit escapes in a timelike direction of the $\Rto$ factor and derive a contradiction with the Opposite Sign Lemma.

\bibliographystyle{amsplain}
\bibliography{Margulis}

\end{document}